\documentclass[11pt]{article}
\usepackage[dvipsnames]{xcolor}
\usepackage{amsmath,amsthm,amssymb}
\numberwithin{equation}{section}
\usepackage{enumerate}
\usepackage{graphicx,tikz}

\usepackage[margin=1in]{geometry}

\usepackage[T1]{fontenc}
\usepackage{ae,aecompl}

\title{Patterson--Sullivan theory for Anosov subgroups}
\author{Subhadip Dey \and Michael Kapovich}

\date{March 20, 2022}

\usepackage[linktocpage,backref=page]{hyperref}
\hypersetup{
    colorlinks=true,
    linkcolor=OliveGreen,
    hypertexnames=false,
    filecolor=black,      
    urlcolor=black,
    citecolor=RawSienna
}
\renewcommand*{\backref}[1]{}
\renewcommand*{\backrefalt}[4]{%
    \ifcase #1 (Not cited.)%
    \or        (Cited on page~#2.)%
    \else      (Cited on pages~#2.)%
    \fi}
\theoremstyle{plain}
\newtheorem{thm}{Theorem}[section]
\newtheorem*{thmn}{Theorem}
\newtheorem{lem}[thm]{Lemma}
\newtheorem{sublem}[thm]{Sublemma}
\newtheorem{prop}[thm]{Proposition}
\newtheorem{cor}[thm]{Corollary}
\newtheorem*{claim}{Claim}
\newtheorem{ques}[thm]{Question}

\newtheorem{innercustomthm}{Theorem}
\newenvironment{customthm}[1]
  {\renewcommand\theinnercustomthm{#1}\innercustomthm}
  {\endinnercustomthm}
  
\newtheorem{innercustomclaim}{Claim}
\newenvironment{customclaim}[1]
  {\renewcommand\theinnercustomclaim{#1}\innercustomclaim}
  {\endinnercustomclaim}

\theoremstyle{definition}
\newtheorem{asm}{Assumption}
\newtheorem{defn}[thm]{Definition}

\newtheorem{exmp}[thm]{Example}
\newtheorem{rem}[thm]{Remark}
\newtheorem*{rem*}{Remark}
\newtheorem*{notation}{Notation}

\def\a{\mathfrak{a}}
\def\acts{\curvearrowright}
\def\con{{\mathrm{con}}}
\def\const{\mathrm{const}}
\def\d{\delta}
\def\D{\Delta}
\def\df{d_{\tb}}
\def\dF{\d_{\tb}}
\def\dH{d_\mathrm{H}}
\def\dhor{\mathcal{B}^{\tb}}
\def\idhor{\mathcal{B}^{\iota\tb}}
\def\dr{d_\mathrm{Riem}}
\def\dR{\d_\mathrm{Riem}}
\def\e{\epsilon}
\def\Flag{\mathrm{Flag}}
\def\fmod{F_{\mathrm{mod}}}
\def\Ft{\mathrm{Flag}({\tau_\mathrm{mod}})}
\def\g{\gamma}
\def\G{\Gamma}

\def\gr{\mathrm{Gr}}
\def\H{\mathcal{H}}
\def\i{\iota}

\def\L{\Lambda}
\def\Lt{\Lambda_{\tau_\mathrm{mod}}}
\def\LT{\Lambda_{\tau_\mathrm{mod}}(\Gamma)}
\def\Ltc{\Lambda_{\tau_\mathrm{mod}}^\mathrm{con}(\Gamma)}

\def\N{\mathbb{N}}

\def\R{\mathbb{R}}
\def\ra{\rightarrow}
 
\def\smod{{\sigma_\mathrm{mod}}}

\def\t{\tau}
\def\tb{{\bar\theta}}
\def\T{\Theta}

\def\tmod{{\tau_\mathrm{mod}}}

\def\vb{\partial_\infty}
\def\ve{\varepsilon}
\def\Xt{\bar{X}^\tmod}

\newcommand{\be}{\begin{equation}}
\newcommand{\dg}[1]{D^{\tb,#1}}
\newcommand{\ee}{\end{equation}}
\newcommand{\fg}[1]{\widehat{#1}}
\newcommand{\LR}[1]{\left( {#1} \right)}
\newcommand{\LRB}[1]{\left\{ {#1} \right\}}
\newcommand{\mb}[1]{\mathbb{#1}}
\newcommand{\mc}[1]{\mathcal{#1}}
\newcommand{\mf}[1]{\mathfrak{#1}}
\newcommand{\mr}[1]{\mathrm{#1}}
\newcommand{\norm}[1]{\langle {#1} \rangle}

\newcommand{\PSL}[1]{\mathrm{PSL}(#1,\mathbb{R})}

\newcommand{\rg}[1]{\overline{#1}}
\newcommand{\SL}[1]{\mathrm{SL}(#1,\mathbb{R})}
\newcommand{\SO}[1]{\mathrm{SO}(#1,\mathbb{R})}

\DeclareMathOperator{\B}{\mathfrak{B}}
\DeclareMathOperator{\card}{card}

\DeclareMathOperator{\diag}{diag}
\DeclareMathOperator{\diam}{diam}

\DeclareMathOperator{\dst}{\partial st}

\def\hd{\mathrm{dim}_{\mathrm{Haus}}}
\DeclareMathOperator{\id}{id}
\DeclareMathOperator{\isom}{Isom}
\DeclareMathOperator{\ost}{ost}

\DeclareMathOperator{\rank}{rank}

\DeclareMathOperator{\st}{st}

\DeclareMathOperator{\supp}{supp}
\DeclareMathOperator{\tr}{tr}

\begin{document}

\maketitle

\begin{abstract}
We extend several notions and results from the classical Patterson--Sullivan theory to the setting of Anosov subgroups of higher rank semisimple Lie groups, working primarily with invariant Finsler metrics on associated symmetric spaces. In particular, we prove the equality between the Hausdorff dimensions of  flag limit sets, computed with respect to a suitable Gromov (pre-)metric on the flag manifold, and the {\em Finsler critical exponents} of Anosov subgroups.
\end{abstract}


Consider a discrete group $\G$ of isometries of the $n$-dimensional hyperbolic space $\mb{H}^n$. The {\em critical exponent} 
$\delta$ is a fundamental numerical invariant associated with $\Gamma$ which measures the asymptotic growth rates of 
$\G$-orbits in $\mb{H}^n$. The relation between the Hausdorff dimension of the limit set $\L(\G)$ of $\G$ and 
its critical exponent is now a classical result. In an influential paper \cite{MR556586}, Sullivan proved the following theorem extending pioneering work by Patterson \cite{MR0450547} on  Fuchsian groups: 

\begin{thmn}[{\cite[Thm. 8]{MR556586}}]
Let $\G$ be a convex-cocompact subgroup of the isometry group of $\mb{H}^n$. Then the critical exponent $\d$ of $\G$ equals to the Hausdorff dimension of $\L(\G)$.
\end{thmn}

Later Sullivan generalized this theorem for geometrically finite Kleinian groups \cite{MR766265}.
An important ingredient of Sullivan's proof of this theorem is the existence of a finite, non-null Borel measure on $\L(\G)$ that changes {\em conformally} under the  $\Gamma$-action. The construction of such measure goes back to Patterson's original idea in \cite{MR0450547}.
Measures of this type (resp. a class of ``well-behaved'' measures) are commonly referred as {\em Patterson--Sullivan measures} (resp. {\em densities}). We refer to Nicholls' book \cite{nicholls1989ergodic} for a self-contained exposition on these results.

Since its introduction, the theory of Patterson and Sullivan has attracted a lot of attention. Further developments have been made by various people who  analyzed more general classes of discrete groups and their limit sets. 
We list some of these developments here. 
Corlette \cite{MR1074486} and Corlette--Iozzi \cite{MR1458321} proved the above theorem for geometrically finite groups of isometries of rank-one symmetric spaces, and 
Bishop-Jones \cite{MR1484767} extended these results to arbitrary discrete isometry groups  of rank-one symmetric spaces.
Yue \cite{MR1348871} and Ledrappier \cite{Ledrappier} studied the case of Hadamard spaces of  negative curvature. 

There has been a considerable amount of development to understand the Patterson--Sullivan theory for discrete subgroups of higher rank semisimple Lie groups acting on its symmetric space, starting with Bishop--Steger \cite{MR1208564} and Burger \cite{MR1230298} in the rank-two case. Later, Albuquerque \cite{MR1675889}, Quint \cite{MR1935549,MR1933790}, and Link \cite{Lin04} considered the case of Zariski-dense discrete subgroups in the isometry groups of general higher rank symmetric spaces. Link \cite{MR2629900} also studied the case of products of rank-one symmetric spaces.
In Appendix \ref{appen:PShigherrank}  we discuss  these papers in relation to our work in more detail.

In the more abstract setting of Gromov hyperbolic spaces, much of Sullivan's work in  \cite{MR556586} was generalized by Coornaert \cite{MR1214072} to the class of quasi\-convex-cocompact groups.
See also work of Paulin \cite{MR1480536} on  actions of  subgroups of Gromov hyperbolic groups.
Recent developments by Das--Simmons--Urba\'{n}ski \cite{MR3558533} 
achieved greater generalizations of the Patterson--Sullivan theory (e.g., a generalization of Bishop-Jones' theorem) in the case of ``infinite-dimensional'' Gromov hyperbolic spaces. 

The goal of this paper is to study the Patterson--Sullivan theory for {\em Anosov subgroups}.
The notion of Anosov subgroups was first introduced by Labourie \cite{MR2221137} to study $\PSL{n}$-Hitchin representations \cite{hitchin} of closed surface groups.
This was further developed by Guichard-Wienhard \cite{MR2981818} and Kapovich-Leeb-Porti \cite{Kapovich:2014aa,MR3736790,kapovich2014morse}. Notably, Anosov subgroups extend the class of convex-cocompact subgroups of rank-one semisimple Lie groups to higher rank. 

In this paper, we primarily work with some of the Kapovich-Leeb-Porti's characterizations of Anosov subgroups.
We briefly review these characterizations, the ones which we will need for this paper (namely, {\em $\tmod$-URU}, {\em $\tmod$-Morse}, and {\em $\tmod$-RCA}), in Subsection \ref{sec:discretegroups}. 
Since we do not use the original notion of Anosov representations as introduced by Labourie \cite{MR2221137}, and this notion have become classical at the time of the writing the paper, we do not include Labourie's definition.
Readers who are interested to understand the connection between Labourie's definition and Kapovich-Leeb-Porti's characterizations are encouraged to read \cite[Subsec. 5.11]{MR3736790}. For instance, it is shown in \cite[Subsec. 5.8]{MR3736790} that Anosov and Morse properties are equivalent. In the same paper \cite{MR3736790}, the authors show that Morse (or any other equivalent notion, including RCA) implies the URU property. Finally, in \cite{kapovich2014morse}, it is shown that URU implies Morse.

\subsection*{Main results}
Let $G$ be a noncompact real semisimple Lie group, 
$X = G/K$ be the associated symmetric space and 
$\G$ be a {\em $\tmod$-Anosov} subgroup of $G$. We will be assuming 
several conditions on $G$ and $X$; they are labeled as ``assumption'' in Section \ref{sec:gp}.
We consider two types of $G$-invariant (pseudo-)metrics on $X$, namely, one is a $G$-invariant Riemannian metric $\dr$ of the symmetric space $X$, and the other one is a $G$-invariant Finsler distance $d_\tb$ which depends on the choice of a direction $\tb$ in the Weyl chamber (see Section \ref{sec:CE}). The {\em critical exponents} of $\G$ with respect to these two metrics, denoted by $\dR$ and $\dF$, respectively, are defined in the usual fashion, i.e., as the exponents of convergence  of an associated {\em Poincar\'e series} (see Section \ref{sec:CE}).
Using the classical construction of Patterson, we define a {\em $\G$-invariant $\tb$-conformal density} on the flag limit set of  $\G$ (see Section \ref{sec:CD}).

Throughout this paper, the Finsler metric $d_\tb$ is given more emphasis than its Riemannian counterpart. For example, the construction of the above mentioned Patterson--Sullivan density is carried out in terms of the Finsler metric. The main reason for this choice is that Finsler metrics reflect the asymptotic geometry of $\G$ better than the Riemannian metric. In fact, an interesting feature of $d_\tb$ (for suitably chosen $\tb$) is that the orbits of Anosov subgroups in $X$ are Gromov-hyperbolic spaces when equipped with $d_\tb$. See Corollary \ref{cor:hypano}.

Let $\smod$ be a maximal simplex in the Tits building of $X$, $\i:\smod\ra\smod$ be the opposition involution, $\tmod$ be an $\i$-invariant face of $\smod$, $P$ be the maximal parabolic subgroup of $G$ that stabilizes $\tmod$, and $\Ft = G/P$ be the partial flag manifold associated to the face $\tmod$. We fix an {\em $\iota$-invariant type}\footnote{Most of our main results are still valid without the assumption that $\tb$ is $\iota$-invariant; see the remark after Theorem \ref{mainB}.} $\tb\in \mr{int}(\tmod)$.

\begin{customthm}{A}\label{mainA}
Let $\G$ be a nonelementary $\tmod$-Anosov subgroup of $G$ and $\dF$ be the Finsler critical exponent for the action of $\G$ on the symmetric space $(X,d_\tb)$.
Then the Patterson--Sullivan density\footnote{See Definition \ref{definition:PSdensity}.} $\mu^\tb$ on the flag limit set $\LT\subset \Ft$ is the unique (up to scaling) $\G$-invariant $\tb$-conformal density.
Moreover,
\begin{enumerate}[(i)]
\setlength\itemsep{0em}
\item The density $\mu^\tb$ is non-atomic and its dimension equals to $\dF$.  \label{mainthm:part1}
\item The support of $\mu^\tb$ is $\LT$ and the action $\G\acts \LT$ is ergodic with respect to $\mu^\tb$. \label{mainthm:part11}
\item The critical exponent $\dF$ (as well as the Riemannian critical exponent $\dR$) is positive and finite.\label{mainthm:part2}
\item The $\tb$-Poincar\'e series of $\G$ diverges at the critical exponent $\dF$. In other words, $\G$ has $\tb$-divergence type.\label{mainthm:part3}
\item The $\dF$-dimensional Hausdorff measure on $\LT$ with respect to a Gromov (pre-)metric\footnote{See Section \ref{sec:dg}.} is a member of a $\G$-invariant conformal density (called the Hausdorff density). In particular, the Hausdorff dimension of $\LT$ is $\dF$.\label{mainthm:part4}
\end{enumerate}
\end{customthm}

While constructions of conformal densities for discrete subgroups of semisimple Lie groups were done earlier (in the Zariski-dense case) by Albuquerque \cite{MR1675889}  and Quint \cite{MR1933790}\footnote{See Appendix \ref{appen:PShigherrank} where we give a brief discussion about these papers.}, most of the results of our Theorem \ref{mainA} are not contained in their work, even in the Zariski-dense setting.
Note that the theorem is false in general for nonelementary discrete subgroups of rank-one Lie groups which are not convex-cocompact.
The proof of the theorem strongly relies on the $\tmod$-Anosov condition.

The fact that $\mu^\tb$ is the unique $\G$-invariant $\tb$-conformal density is proven in Corollary  \ref{cor:uniqueAnosovDensity}. The main ingredients in the proof are a generalization of Sullivan's {\em shadow lemma} proven in Theorem \ref{thm:SL}, and an ergodicity argument (see Theorem \ref{thm:ergodic}) due to Sullivan.
The proof of part (\ref{mainthm:part1}) of the theorem is given in Corollaries \ref{cor:nonconicalAtom} and \ref{cor:dimAno}.
The second half of part (\ref{mainthm:part11}) follows from Theorem \ref{thm:ergoAno} while the first half follows from the facts that the support of $\mu^\tb$ is a closed $\G$-invariant subset of $\LT$ and the action $\G \acts \LT$ is minimal.
The part (\ref{mainthm:part2}) is proven in Propositions \ref{prop:finiteCE} and \ref{cor:positiveCE}. See also the remarks following these propositions where $\dR$ is analyzed.
The part (\ref{mainthm:part3}) follows from Corollary \ref{cor:div}.
The Hausdorff density in part (\ref{mainthm:part4})  is studied in Section \ref{sec:HD} (cf. Theorem \ref{thm:hd}). The background Gromov (pre-)metric is introduced in Section \ref{sec:dg} where we also prove that the action $\G \acts \LT$ with respect to this metric is {\em conformal} (see Corollary \ref{cor:gmetricano}).

We should note that some of the results in this paper are proven for more general classes of discrete subgroups of $G$ 
with the hope that the results may be useful, for instance, in the study of {\em relatively Anosov subgroups}.\footnote{Relatively Anosov subgroups defined by Kapovich-Leeb \cite{relmorse} are an extension  of the class of geometrically finite groups into the higher rank.}

For a much wider class of (uniformly) {\em $\tmod$-RA}\footnote{See Subsection \ref{sec:discretegroups} for this definition.} subgroups we prove the following results.

\begin{customthm}{B}\label{mainB}
 Let $\G < G$ be a nonelementary $\tmod$-RA subgroup of $G$, and let $\dF$ be its $\tb$-critical exponent. Then $\dF \in (0,\infty]$. Let $\mu$ be a $\beta$-dimensional $\G$-invariant $\tb$-conformal density (if exists).
 \begin{enumerate}[(i)]
 \setlength\itemsep{0em}
 \item (Shadow Lemma) Fix $x,x_0\in X$. There exists $r_0>0$ such that for all $r\ge r_0$ and all $\g\in\G$ satisfying $\dr(x,\g x_0)> r$,  
$
\mu_{x}(S(x: B(\g x_0,r))) \asymp \exp \left(-\beta \df(x,\g x_0)\right)
$
(Theorem \ref{thm:SL}).
 \item $\beta \ge \dF - \dF^\con$ (Theorem \ref{thm:lowerbound}).
 \end{enumerate}
 If we further assume that $\G$ is uniformly $\tmod$-regular, then
\begin{enumerate}[(i)]
\setlength\itemsep{0em}
\setcounter{enumi}{2}
 \item $\dF$ is finite (Proposition \ref{prop:finiteCE}), and
 \item the density $\mu$ cannot have atoms at conical limit points (Corollary \ref{cor:nonconicalAtom}).
\end{enumerate}
\end{customthm}

\begin{rem*}
 Theorem \ref{mainA}, with the exception of the item (\ref{mainthm:part4}), and Theorem \ref{mainB} remain valid without the assumption that the type $\tb\in\mr{int}(\tmod)$ is {\em $\iota$-invariant}.
 In Appendix \ref{appen:without_iota_invariance}, we show how to generalize these statements without this assumption.
\end{rem*}

\medskip\noindent
{\bf Some historical remarks.} The early work on critical exponent  the Patterson--Sullivan theory 
in the higher rank was mostly developed for general (but, typically, Zariski dense) 
discrete subgroups of higher rank Lie groups; we discuss this early work (in relation to our paper) in Appendix \ref{appen:PShigherrank}.  Since 
the introduction of Hitchin representations of surface groups (and proof of their Anosov property by Labourie) and, more generally, Anosov representations of hyperbolic groups, a substantial work was done investigating different versions of critical exponent 
and the Patterson--Sullivan theory, and their applications. Below is a brief discussion of this work.  

For certain classes of Anosov subgroups, the Patterson--Sullivan theory was used by Sambarino in \cite{MR3229035,MR3449196}  to 
solve certain counting problems, while in  \cite{BCLS15}  Bridgemann--Canary--Labourie--Sambarino used related thermodynamic formalism to construct pressure metrics on spaces of Hitchin representations. 
Moreover, Glorieux--Monclair \cite{GM16} studied the Patterson--Sullivan theory in the case of {convex-cocompact} subgroups of the isometry group of $\mb{H}^{p,q}$ equipped with the pseudo-Riemannian metric. 
In their  work \cite{PS17}, Potrie and Sambarino  prove some interesting inequalities for critical exponents 
of Hitchin representations and a beautiful  rigidity theorem characterizing ``Fuchsian'' representations 
in the Hitchin component, which are reminiscent of the earlier inequalities for critical exponents and rigidity theorems (going back to the work of R.~Bowen) for Kleinian groups, except that the inequalities are in the opposite direction. 

While working on this article, we came to know about two very recent developments by Pozzetti--Sambarino--Wienhard \cite{Pozzetti:2019aa} and Glorieux--Monclair--Tholozan \cite{GMT19}, which are related to our work.
Independently, the authors of these articles proved that the Hausdorff dimension of the limit set of a {\em projective} Anosov subgroup $\G$ 
in the real projective space with respect to the Riemannian metric is bounded above by a certain critical exponent, called the
``simple root critical exponent'' in the second article.
The main result of \cite{Pozzetti:2019aa} is stronger than this inequality for a special class of {\em (1,1,2)-hyperconvex} representations, in which case the Hausdorff dimension equals to the simple root critical exponent.
They went further to prove that for hyperconvex subgroups $\Gamma$ having $\partial\Gamma$ homeomorphic to a sphere, the limit set of $\Gamma$ in the projective space is a $C^1$ sphere.
As a corollary of these two results, they obtained an earlier result of Potrie--Sambarino \cite{PS17} on the entropy of Hitchin representations.
In \cite{GMT19}, the authors 
also aimed to get a lower bound for the Hausdorff dimension of the limit set of general projective Anosov subgroups $\Gamma$.
As it is mentioned in \cite{GMT19}, initially, the authors aimed  to obtain such a lower bound with respect to the Riemannian metric;  eventually, they proved such a lower bound for a certain Gromov metric on the limit set.
Using our Theorem \ref{mainA} and relying on previously known computations of  Busemann functions (see Example \ref{ex:2.3}), we obtain a lower bound for this Hausdorff dimension with respect to the Riemannian metric (see Theorem \ref{thm:PARhd}).  

After this work was completed, Andr\'{e}s Sambarino informed us that Ledrappier's methods from \cite{Ledrappier} (in conjunction with results of \cite[Sec. 3.2]{BCLS15}) can be used to obtain some of the results of our paper; we refer the reader to \cite{MR3229035,MR3449196} for similar applications of 
Ledrappier's work. 

\subsection*{Acknowledgement} 
This project was a part of the first author's dissertation work at UC Davis.
The second author was partly supported by the NSF grant DMS-16-04241. We are grateful to Olivier Glorieux and Hee Oh for pointing out some mistakes in a previous version of this paper and to Andr\'{e}s Sambarino for telling us about Ledrappier's work and related results.
Finally, we are thankful to the anonymous referee for carefully reading our paper and making numerous suggestions to improve the text.

\setcounter{tocdepth}{2}
\tableofcontents

\section*{Notations}
\begin{itemize}
\setlength\itemsep{0em}
\item $\B(Y)$: Class of Borel subsets of a topological space $Y$
\item $B(x,r)$: (Closed) ball of radius $r$ centered at $x$
\item $\df, \dr$: Finsler and Riemannian metrics, respectively, on $X$ (see Sec. \ref{sec:CE})
\item $\fg{xy}$, $\rg{xy}$: Finsler\footnote{Note that Finsler geodesic segments connecting two points in $X$ are usually non-unique.} and Riemannian geodesic segments, respectively, connecting $x,y\in X$
\item $\dF$, $\dR$: Finsler and Riemannian critical exponents, respectively, of $\G$ (see Sec. \ref{sec:CE})\
\item $\dg{\e}_x$: Gromov premetric (see Def. \ref{def:dg})
\item $\dhor_\t$: Busemann cocycle  (see (\ref{eqn:horodist}))
\end{itemize}

\section{Geometric preliminaries}\label{sec:gp}
In this section, we briefly present some background material needed for the paper.

\subsection{Symmetric spaces} A {\em symmetric space} $X$ is a Riemannian manifold that has an {\em inversion symmetry} or {\em point-reflection} with respect to each point $x\in X$: This is an isometric involution $s_x: X \ra X$ fixing $x$ and sending each tangent vector at $x$ to its negative. In this paper we only consider symmetric spaces which are simply-connected and have {\em noncompact type}. The later means that $X$ has no flat deRham factor and the sectional curvature of $X$ is non-positive. In particular, $X$ is a Hadamard manifold and, hence, is diffeomorphic to a euclidean space. We refer to Eberlein's book \cite{MR1441541} for a detailed discussion of symmetric spaces.

\begin{asm}
The symmetric spaces $X$ is simply-connected and of noncompact type.
\end{asm}

A symmetric space $X$ can be written as $G/K$ where $G$ is a semisimple Lie group whose Lie algebra does not have compact and  abelian factors, and $K$ is a maximal compact subgroup of $G$. Moreover, this group $G$ can be chosen to have finite center and be commensurable with the isometry group $\isom(X)$ of $X$. For example, one can choose $G$ to be the identity component of $\isom(X)$. 

\begin{asm}
The semisimple Lie group $G$ has finite center and is commensurable with the isometry group $\isom(X)$ of the symmetric space $X$.
\end{asm}

Each point $x\in X$ determines a canonical decomposition  of the Lie algebra $\mathfrak{g}$ of $G$ called the {\em Cartan decomposition},
\[
\mathfrak{g} = \mf{k} + \mf{p}
\]
where $\mf{k}$ is tangent to the stabilizer of a point $x \in X = G/K$, and $\mf{p}$ can be realized as the tangent space of the symmetric space $X$ at $x$. The dimension of a maximal abelian subalgebra $\mf{a}\subset \mf{p}$ is called the {\em rank} of $X$. The exponential map $\exp_x: \mf{p} \ra X$ identifies $\mf{a}$ with a maximal flat $F\subset X$ through $x$ and, hence, the rank of $X$ can also be defined as the dimension of a maximal totally geodesic flat  subspace  in $X$.
A chosen maximal flat $\fmod \subset X$ is called the {\em model flat} which we isometrically identify with $\R^{k}$ where $k=\rank(X)$. The image in $\isom(F)$ of the $G$-stabilizer of 
$\fmod$ is isomorphic to $\R^k \rtimes W$, where the first factor acts on $\fmod \cong \R^k$ by translations while the second factor $W$, called the {\em Weyl group}, is finite, fixes the origin, and is generated by hyperplane reflections. The closures of the connected components of the complement of the reflecting hyperplanes (for hyperplane reflections in $W$) in $\fmod$ are called \emph{chambers}. A chosen chamber is called the \emph{model Weyl chamber};  we denote it by $\D$.

\subsection{Boundary at infinity}\label{sec:boundary}
For a symmetric space $X$, there are multiple notions of (partial) boundary at infinity.
The space of equivalence classes of asymptotic rays is called the {\em visual boundary} of $X$ and denoted $\vb X$. The visual boundary is naturally identified with the unit tangent sphere $T^1_x X$ at any point $x\in X$. The topology it gets from this identification is called the {\em visual topology}.
Attaching the visual boundary to $X$ provides a compactification of $X$.

Another (strictly finer) topology on $\vb X$ is given by the $G$-invariant {\em Tits angle metric}:
\[
\angle_{\mr{Tits}}(\zeta,\eta) = \sup_{x\in X}\angle_x(\zeta,\eta)
\]
where $\angle_x(\zeta,\eta)$ denotes the angle between the rays emanating from $x$ and asymptotic to $\zeta$ and $\eta$. The boundary $\vb X$ with this topology is  called the {\em Tits boundary} $\partial_{\mr{Tits}} X$.

The Tits boundary $\partial_{\mr{Tits}} X$ carries a canonical $G$-invariant structure of a spherical simplicial complex called the {\em Tits building} of $X$. This can be understood as follows: Consider the ideal boundary $\vb\fmod$ of $\fmod$ where $k=\rank(X)$. This is identified with the unit sphere $\a^1$ of $\a$ and thus, we have an action of the Weyl group $W\acts \partial_{\mr{Tits}} \fmod$. The pair $(\partial_{\mr{Tits}} \fmod, W)$ is a spherical Coxeter complex which generates a spherical simplicial complex structure on entire $\partial_{\mr{Tits}} X$ by the $G$-action.

\begin{asm}\label{asm:3}
We assume that the Tits building is \emph{thick}, i.e., every simplex of codimension one is a face of three maximal simplices.\footnote{This is a standing assumption on spherical buildings in the papers by Kapovich, Leeb and Porti we rely upon in our work. See for instance, \cite[Lemma 2.4]{MR3736790} and its usage elsewhere in that paper.} 
\end{asm}

We do not have to worry about this assumption when $X$ is an irreducible symmetric space. The Tits building of $X$, in that case, is thick. Nevertheless, we impose this assumption to avoid situations like in the following example.

\begin{exmp}
 Let $X = \mathbb{H}^2 \times \mathbb{H}^2$, and $G = \PSL{2}\times\PSL{2}\times (\mathbb{Z}/2\mathbb{Z})$, where the nontrivial element in $\mathbb{Z}/2\mathbb{Z}$ acts by swapping the factors of $X$.
 The corresponding Tits building of $X$ is not thick.
\end{exmp}

We denote the intersection of  $\D$ with the unit sphere in $\fmod$ centered at the origin by $\smod$. This is a fundamental domain for the action $W\acts \partial_{\mr{Tits}} \fmod$ where $\partial_{\mr{Tits}} \fmod$ is identified with the unit sphere in $\fmod$ centered at the origin. 
We call $\smod$  the {\em model chamber}. Any other chamber (i.e., a top-dimensional simplex) in the Tits building is naturally identified with $\smod$ via a $G$-equivariant map, called the {\em type map},
\[\theta : \partial_{\mr{Tits}} X \ra \smod.\]
We reserve the notation $\tmod$ for the faces of $\smod$. 
An ideal point $\zeta \in \partial_{\mr{Tits}} X$ (resp. a simplex $\t \subset \partial_{\mr{Tits}} X$) is called of {\em type $\tb\in\smod$} (resp. of {\em type $\tmod \subset \smod$}) if $\theta(\zeta) = \tb$ (resp. $\theta(\t) = \tmod$). 
For $\tb\in\tmod$ and a simplex $\t$ of type $\tmod$, we use the notation $\tb(\t)$ to denote the unique point in $\t$ of type $\tb$.
The {\em opposition involution} $\iota$ is an automorphism of $\smod$ which is defined as the negative of the longest element in the Weyl group.

Two simplices $\tau_1,\tau_2$ in the Tits building are called \emph{antipodal} if there exists a point-reflection $s_x$ swapping these two. Their types are related by $\theta(\t_1) = \iota \theta(\t_2)$. In particular, when $\t_1$ has an $\i$-invariant type $\tmod$, then any antipodal simplex $\t_2$ also has type $\tmod$. In this paper, we only consider types that are $\i$-invariant.  

We now describe an important class of partial boundaries of $X$ which are central to our study. Consider the action of $G$ on the Tits building. The stabilizer of a face $\tmod$ of $\smod$ is a parabolic subgroup $P_\tmod$ of $G$ and we identify the quotient $G/P_\tmod$ with the set of all simplices  of type $\tmod$ in the Tits building. This quotient $G/P_\tmod$ is a smooth compact manifold, 
called the {\em partial flag manifold} of type $\tmod$ and is denoted  $\Ft$.
The partial compactification of $X$ by attaching $\Ft$ is denoted 
\[
\Xt = X\cup\Ft
\]
which is topologized via the topology of {\em flag convergence} (see Subsection \ref{sec:discretegroups}).
In the special case when $\tmod = \smod$, the associated parabolic subgroup $P_\smod$ is minimal and  $\Flag(\smod) = G/P_\smod$ is the {\em full flag manifold}, also called the {\em Furstenberg boundary} of $X$. 

A subset $A\subset \Ft$ is called {\em antipodal} if any two distinct simplices in $A$ are antipodal.

\subsection{$\D$-valued distances and a generalized triangle inequality}
There is a canonical map $d_\D: X\times X \ra \D$ which is defined as follows: For a pair of points $(x,y)$ in $X$, 
there is an element $g\in G$ which maps $x$ to the origin in $\Delta$ and $y$ to a point $v\in\D$. We define $d_\D(x,y) = v$.
Note that the norm $\| d_\D(x,y)\|$ (induced by the euclidean inner product on $\fmod \cong\R^k$) equals $\dr(x,y)$ where $\dr$ denotes the distance function induced by the Riemannian metric on $X$.

For a pair $(x,y)\in X\times X$, the value $d_\D(x,y)$ is called the {\em $\D$-valued distance} between $x$ and $y$. This is a complete $G$-congruence invariant for oriented line segments in $X$.
The $\D$-valued distances satisfy generalized triangle inequalities (see \cite{MR2472176}). In the paper we will need the following triangle inequality. For $x,y,z\in X$,
\be\label{eqn:ti}
\| d_\D(x,y) - d_\D(x,z) \| \le \dr(y,z).
\ee

\subsection{Parallel sets, cones, and diamonds}\label{sec:star}
For a detailed discussion on this subsection, we refer to \cite[Subsec. 2.4]{Kapovich:2014aa}, \cite[Subsec. 2.5]{MR3736790}.

Let $\t_\pm$ be a pair of antipodal simplices in the Tits building of $X$. The {\em parallel set} $P(\t_+,\t_-)$ is the union of all maximal flats in $X$ whose ideal boundary contains $\t_+\cup \t_-$ as a subset. This is a totally geodesic submanifold of $X$. 

For a simplex $\t$, the {\em star} $\st(\t)$ of $\t$ is the union of all chambers in the Tits building containing $\t$. The {\em open star} $\ost(\t)$ of $\t$ is the union of all the open simplices whose closures contains $\t$. For a face $\tmod$ of $\smod$ (viewed as a complex), define the open star $\ost(\tmod)$ similarly. The {\em boundary} $\dst(\tmod)$ is the complement of $\ost(\tmod)$ in $\smod$. 

Let $\tmod$ be an $\i$-invariant face of $\smod$.
An ideal point $\xi \in\vb X$ is called {\em $\tmod$-regular} if its type is contained in $\ost(\tmod)$. 
Moreover,  given an $\i$-invariant compact subset $\Theta\subset \ost(\tmod)$, an ideal point $\xi \in\vb X$ is called {\em $\Theta$-regular} if its type is contained in $\Theta$. 
A nondegenerate geodesic segment (or line or ray) in $X$ is called {\em $\tmod$-regular}  (resp. {\em $\Theta$-regular}) if the ideal endpoints of its line extension are $\tmod$-regular (resp. $\Theta$-regular). 

For a simplex $\t$ in the Tits building and a point $x\in X$, the {\em $\tmod$-cone $V(x,\st(\t))$ with apex $x$} is the union of all rays emanating from $x$ asymptotic to a point $\xi\in\st(\t)$. 
For a $\tmod$-regular geodesic segment $\rg{xy}\subset X$, the {\em $\tmod$-diamond} $\diamondsuit_\tmod(x,y)$ is the intersection of the opposite cones $V(x,\st(\t_+))$ and $V(y,\st(\t_-))$ containing it. The points $x$ and $y$ are called the {\em endpoints} of $\diamondsuit_\tmod(x,y)$. The cones and parallel sets can be interpreted as limits of diamonds where, respectively, one or both endpoints diverges to infinity. All of these are convex subsets of $X$ (see \cite[Prop. 2.14]{Kapovich:2014aa}, \cite[Prop. 2.10]{MR3736790}). In particular, the cones are {\em nested}: For every $y\in V(x, \st(\tau))$, $V(y, \st(\tau))\subset V(x, \st(\tau))$.

Let $\Theta$ be an $\i$-invariant compact subset of $ \ost(\tmod)$. In a similar way as above, the {\em $\Theta$-cone $V(x,\ost_\Theta(\t))$ with apex $x$} is the union of all rays emanating from $x$ asymptotic to a point $\xi\in\st(\t)$ of type $\Theta$. Note that $V(x,\ost_\Theta(\t))$ is strictly contained inside $V(x,\st(\t))$.

\subsection{Morse embeddings}\label{sec:morse}
 The {\em Morse property} in higher rank was introduced by Kapovich-Leeb-Porti in \cite{Kapovich:2014aa}.

Recall that a {\em quasigeodesic} in $X$ is a {quasiisometric} embedding $\phi:I\ra X$ of an interval $I\subset\R$. We say that $\phi$ is {\em $\tmod$-regular quasigeodesic} if for all sufficiently separated points $t_1,t_2 \in I$, the segment $\rg{\phi(t_1)\phi(t_2)}$ is $\tmod$-regular. We say that $\phi$ is a {\em $\tmod$-Morse quasigeodesic} if it is $\tmod$-regular and for all sufficiently separated points $t_1,t_2 \in I$, the image $\phi([t_1,t_2])$ is uniformly close to $\diamondsuit_\tmod(\phi(t_1),\phi(t_2))$.

Let $Z$ be a geodesic Gromov-hyperbolic metric space (cf. Definition \ref{def:GH}). 

\begin{defn}[Morse embeddings]
 A quasiisometric map $\phi:Z\ra X$ is called a {\em $\tmod$-Morse embedding} if the image of every geodesic is a $\tmod$-Morse quasigeodesic with uniformly controlled coarse-geometric quantifiers: There exists a constant $D>0$ and an $\i$-invariant compact subset $\Theta\subset\ost(\tmod)$ such that if ${z_1z_2}$ is a geodesic segment in $Z$ of length $\ge D$, then $\rg{\phi(z_1)\phi(z_2)}$ is a $\T$-regular geodesic 
in $X$ and the image $\phi([z_1,z_2])$ is $D$-close to $\diamondsuit_\tmod(\phi(z_1),\phi(z_2))$.
\end{defn}

\subsection{Discrete subgroups of $G$ and their limit sets}\label{sec:discretegroups}
We consider discrete subgroups with various levels of regularity and their flag limit sets. Most of these notions 
first appear in the work of Benoist \cite{MR1437472}; our discussion follows 
 \cite{Kapovich:2014aa}  and \cite{MR3736790}. 

We first recall the notion of {\em regular} sequences in $X$. Let $\tmod$ be an $\i$-invariant face of $\smod$. Let $V(0,\dst(\tmod))$ denote the union of all rays in $\D$ emanating from $0$ asymptotic to points $\xi\in\dst(\tmod)$. A sequence $(x_n)$ on $X$ diverging to infinity is {\em $\tmod$-regular} if for all $x\in X$, the sequence $\LR{d_\D(x,x_n)}_{n\in\N}$ in $\D$ diverges away from $V(0,\dst(\tmod))$.
Furthermore, a $\tmod$-regular sequence $(x_n)$ is called {\em uniformly $\tmod$-regular} if the sequence 
$\LR{d_\D(x,x_n)}_{n\in\N}$ in $\D$ diverges away from $V(0,\dst(\tmod))$ at a linear rate,
\[
\liminf_{n\ra\infty} \frac{d\LR{d_\D(x,x_n), V(0,\dst(\tmod))}}{d(0,d_\D(x,x_n))} >0.
\]
where $d$ denotes the euclidean distance on $\D$. Accordingly, a sequence $(g_n)$ in $G$ is $\tmod$-regular (resp. 
uniformly $\tmod$-regular) if for some (equivalently, every) $x\in X$, the sequence $(g_n(x))$ is 
$\tmod$-regular (resp. uniformly $\tmod$-regular). 

Recall from Subsection \ref{sec:boundary} that we have identified $\Ft$ with the set of all simplices of type $\tmod$ in the Tits building of $X$.
Also, recall the notion of the stars $\st(\t)$, and cones $V(x,\st(\t))$ from Subsection \ref{sec:star}.

\begin{defn}[Shadows]\label{defn:shadows}
 For $x\in X$ and $A\subset X$, the {\em shadow of $A$ in $\Ft$ from $x$} is 
\be\label{defn:shadow}
S(x:A) = \{\t \in\Ft \mid A \cap V(x,\st(\t)) \ne \emptyset \}.
\ee
\end{defn}

\begin{rem}\label{defn:shadow_topology}
 The notion of shadows is used in \cite[Subsec. 3.8]{kapovich2014morse} to produce a topology, called the {\em shadow topology}, in $\Ft$.
 This topology is generated by the following basic subsets:
 \[
  S(x:B(y,r)), \quad \text{where }x,y\in X, \text{ and } r>0.
 \]
 Moreover, the shadow topology coincides with the standard topology (i.e., the underlying topological space of a $K$-invariant Riemannian metric) on $\Ft$. See  \cite[Lem. 3.82]{kapovich2014morse}.
\end{rem}

 Let $(g_n)$ be a $\tmod$-regular sequence in $G$.
A sequence $(\t_n)$ in $\Ft$ is called a {\em shadow sequence} of  $(g_n)$ if 
there exists $x\in X$ such that, for every $n\in\N$, $\t_n = S(x:\{g_n x\})$.
A $\tmod$-regular sequence $(g_n)$ is said to be {\em $\tmod$-flag-convergent} 
to $\t\in \Ft$ if a(ny) shadow sequence $(\t_n)$ of $(g_n)$ converges to $\t$. 
This notion of flag-convergence is the same as 
the one originally introduced in  \cite{kapovich2014morse}, where the shadow topology was not yet defined. 

The notion of flag-convergence leads to the definition of  {\em flag limit sets} of  discrete subgroups $\G< G$.

\begin{defn}[Limit sets]
 The {\em $\tmod$-flag limit set}  of a discrete subgroup $\G$ of $G$, denoted by $\LT$, is the subset of $\Ft$ which consists of 
all limit simplices of $\tmod$-flag-convergent sequences on $\G$.
\end{defn}

\begin{rem}
  The flag limit set $\Lt$ is $\G$-invariant. 

\end{rem}

More generally, one defines $\tmod$-flag-limit sets of a subset $Z\subset X$ as the accumulation subset of $Z$ in $\Ft$ with respect to the topology of flag-convergence.

Now, we review definitions of several classes of discrete subgroups of $G$ with various levels of regularities:

\begin{description}
\item[(R)] A discrete subgroup $\G < G$ is {\em $\tmod$-regular} if for all $x\in X$ and all sequences of distinct elements 
$(\g_n)$ in $\G$, the sequence $(\g_n x)$ is $\tmod$-regular. For $\tmod$-regular subgroups $\G$, the flag limit set $\LT$ provides a compactification of the orbit $\Gamma x\subset X$, i.e., 
$\Gamma x\sqcup \LT$ is compact.

\item[(RA)] A $\tmod$-regular subgroup $\G$ is {\em $\tmod$-RA} ({\em regular antipodal}) 
if its limit set $\LT$ is antipodal, i.e., every two distinct elements of $\LT$ are antipodal to each other. 
For $\tmod$-RA subgroups $\G$, the action $\G\acts\LT$ is a {\em convergence action}\footnote{Recall, that an action $\G\acts Z$ is called a {\em convergence action} if the induced action $\G \acts Z^{(3)}$ is properly discontinuous. Here $Z^{(3)}$ denotes the space of all triples of pairwise distinct points in $Z$. The action $\G\acts Z$ is called {\em uniform} convergence action if, in addition, $\G \acts Z^{(3)}$ is a cocompact action.} (see \cite[Prop. 5.38]{Kapovich:2014aa}).
A $\tmod$-RA  subgroup $\G$ is called {\em nonelementary} if $\LT$ consists of at least three (hence infinitely many) points; otherwise $\G$ is called {\em elementary}. If $\G$ is nonelementary then the action $\G\acts\LT$ is {\em minimal}, i.e., every orbit of $\G$ is dense, and $\LT$ is perfect.\footnote{This follows from a general result for convergence actions by Gehring--Martin \cite{MR896224} and Tukia \cite{MR1313451}. See also \cite[Subsec. 3.2]{Kapovich:2014aa} or \cite[Subsec. 3.3]{MR3736790}.}

\item[(RC)] For a $\tmod$-regular subgroup $\G$, a limit simplex $\t \in \LT$ is a {\em conical limit point} 
if there exists $x\in X$, $c>0$ and a sequence $(\g_n)$ of pairwise distinct isometries on $\G$ such that
\[
\dr(\g_n x, V(x,\st(\t))) \le c
\]
where $\dr$ denotes the Riemannian distance on $X$.
The set of all conical limit simplices is denoted by $\Ltc$. A subgroup $\G< G$ is called {\em $\tmod$-RC} if $\LT = \Ltc$.

\item[(RCA)] A subgroup $\G$ is {\em $\tmod$-RCA} if it is both $\tmod$-RA and $\tmod$-RC. 

\item[(U)] A finitely generated subgroup $\G< G$ (equipped with the word metric) 
is said to be {\em undistorted} if one (equivalently, every) orbit map $\Gamma\to \Gamma x\subset X$ is a quasiisometric embedding.

\item[(UR)] A discrete subgroup $\G < G$ is {\em uniformly $\tmod$-regular} if for all $x\in X$ and all sequences of distinct elements  
$(\g_n)$ in $\G$, the sequence $(\g_n x)$ is uniformly $\tmod$-regular.

\item [(URU)] A subgroup $\G<G$ is said to be $\tmod$-URU if it is both $\tmod$-uniformly regular and undistorted.

\item[(Morse)] A discrete finitely generated subgroup (equipped with a word metric) $\G < G$ is called {\em $\tmod$-Morse} if it is word-hyperbolic
and, for an(y) $x\in X$, the orbit map $\G \ra \G x$ is a $\tmod$-Morse embedding. See Subsection \ref{sec:morse}.
\end{description}

In \cite[Equiv. Thm. 1.1]{MR3736790} and \cite{kapovich2014morse}, the properties Morse, RCA and URU are proven to be equivalent to the Anosov property defined by Labourie \cite{MR2221137} and Guichard-Wienhard \cite{MR2981818}. 

\begin{thm}[{\cite[Equiv. Thm. 1.1]{MR3736790}}] \label{thm:equiv}
The following classes of nonelementary discrete subgroups of $G$ are equal:
\begin{enumerate}[(i)]
\setlength\itemsep{0em}
\item $\tmod$-RCA,
\item $\tmod$-Morse,
\item $P_\tmod$-Anosov,
\item  $\tmod$-URU.
\end{enumerate}
\end{thm}

In the sequel, any discrete subgroup that satisfies any of equivalent conditions in the theorem will be called a {\em $\tmod$-Anosov subgroup}.

\subsection{Illustrating examples}
In this paper, we consider the following two classes of examples.

\begin{exmp}[Product of rank-one symmetric spaces]\label{ex:1}
Let $X$ be a product of $k$ rank-one symmetric spaces $(X_i,d_i)$,
\[
X = X_1 \times\dots\times X_k.
\]
The rank of $X$ is $k$.
Let $G$ be a semisimple Lie group commensurable with the isometry group of $X$. (For example, we may take 
$G = \isom(X_1)\times \dots\times \isom(X_k)$.) The  Assumption \ref{asm:3} in page \pageref{asm:3} amounts to the requirement that $G$ preserves the factors of the 
direct product decomposition of $X$. 

The model maximal flat $\fmod$ can be viewed as the product of some chosen geodesic lines (coordinate axes), one for each deRham factor. The Weyl group $W$ is generated by reflections along the coordinate hyperplanes and the longest element in it is the reflection about the origin. The model Weyl chamber $\D$ can be realized as the nonnegative orthant. 
Here is a formula for the $\Delta$-valued distances: For $(x_1,\dots,x_k), (y_1,\dots,y_k) \in X_1 \times\dots\times X_k$,
\begin{equation}\label{eqn:DeltaValuedRankOne}
 d_\Delta((x_1,\dots,x_k), (y_1,\dots,y_k)) = \left( d_1(x_1,y_1),\dots,d_k(x_k,y_k) \right) \in \mathbb{R}_{\ge0}^k.
\end{equation}
It follows that the opposition involution $\i$ acts on $\Delta$ it trivially.

Recall that the Tits boundary of a product of two symmetric spaces is the simplicial join of their individual Tits buildings and, for rank-one symmetric spaces, the Tits boundary is discrete. These two facts imply that the $(p-1)$-simplices in the Tits building of $X$ for $1\le p\le k$ can be parametrized by $p$-tuples
$(\xi_{r_1},\dots,\xi_{r_p}) \in \vb X_{r_1} \times\dots\times \vb X_{r_p}$, $1\le r_1<\dots< r_p \le k$,
\[
(\xi_{r_1},\dots,\xi_{r_p}) \leftrightarrow \tau = \mr{span}\{\xi_{r_1},\dots,\xi_{r_k}\}.
\]  We say that such a simplex $\tau$ has type $\tmod = (r_1,\dots, r_p)$. The incidence structure can be understood as follows: Two simplices have a common $q$-face if and only if they have $q$ equal coordinates. 

The star $\st(\t)$ of $\t = (\xi_{r_1},\dots,\xi_{r_p})$ is the minimal subcomplex of the Tits building containing all chambers $(\zeta_1,\dots,\zeta_k)$ satisfying $\zeta_{r_i} = \xi_{r_i}$, for all $i\in\{1,\dots,p\}$.

Since the opposition involution $\i$ fixes each chamber point-wise, every face $\tmod$ of $\smod$ and every type is $\i$-invariant.
Every two chambers (resp. faces of the same type) in $\partial_{\mr{Tits}} X$ 
are antipodal to each other unless they have a common face (resp. sub-face).
\end{exmp}

\begin{exmp}[$X = \SL{k+1}/\SO{k+1}$]\label{ex:2}
We take $G = \SL{k+1}$, $K = \SL{k+1}$; the symmetric space $X=G/K$ is identified with the set of all positive definite, symmetric
matrices  in $\SL{k+1}$. In this case $\rank(X) = k$ and $X$ is irreducible. The standard choice of a model flat $\fmod$ 
is the subset of all diagonal matrices $a = \diag(a_1,\dots, a_{k+1})\in  \SL{k+1}$ with positive diagonal entries. 
We identify the model flat with $\a$ via the logarithm map
$$
\log: a= \diag(a_1,\dots, a_{k+1}) \mapsto (\log a_1,\dots, \log a_{k+1})$$ 
where $\a$ is viewed as the hyperplane in $\R^{k+1}$ consisting of all points with zero sum of coordinates.

The Weyl group $W = \mr{Sym}_{k+1}$ acts on $\a$ by permuting the coordinates.
The standard choice for the model Weyl chamber $\D = \a_+$ consists of all the points in $\a$ with decreasing coordinate entries.
 The Cartan projection\footnote{Or the $\D$-valued distance in the sense that $d_\D(x,gx) = \rho(g)$.}
  $\rho: \SL{k+1} \ra \a_+$ can be written as
 $g\mapsto \log a$ where $a$ is associated to $g$ via the 
 {\em singular value decomposition} $g = uav$, $u,v\in\SO{k+1}$. The logarithm of $i$-th singular value of $g$ will be denoted by $\sigma_i(g)$.
 The opposition involution $\i$ maps $(\sigma_1,\dots, \sigma_{k+1})\in\a_+$ to $(-\sigma_{k+1},\dots,-\sigma_1)$.

The Tits building of $X$ can be identified with the incidence geometry of flags in $\R^{k+1}$. 
 The Furstenberg boundary consists of full flags
 \[
 V_1 \subset \dots \subset V_{k+1} = \R^{k+1}, \quad \dim(V_i) = i.
 \]
The partial flags are
 \[
V : \ V_{r_1} \subset \dots \subset V_{r_p} \subset V_{r_{p+1}}= \R^{k+1}, \quad \dim(V_{r_i}) = r_i,
\]
$1\le r_1 < \dots < r_p < r_{p+1} = k+1$, which are elements of $\Ft$ where $\tmod = (r_1,\dots,r_p)$.
The opposition involution maps $\tmod$ to $\i\tmod = (k+1 - r_p,\dots, k+1 - r_1)$. It follows that $\tmod$ is $\i$-invariant if and only if $r_i + r_{p+1 -i} = k+1$, for each $i=1,\dots, p$. The partial flag manifold $\Ft$ consisting of all partial flags $V$ of type $\tmod= (r_1,\dots,r_p)$ naturally embeds into the product of Grassmannians $\gr_{r_1}(\R^{k+1})\times \dots\times  \gr_{r_p}(\R^{k+1})$.

Suppose that $\tmod=(r_1,\dots,r_p)$ is $\i$-invariant.
A pair $V^\pm \in \Ft$ is antipodal if and only if $V^+_{r_i} + V^-_{r_{p+1-i}} = \R^{k+1}$ for each $i=1,\dots, p$.
\end{exmp}

\section{Critical exponent}\label{sec:CE}

On a symmetric space $X = G/K$, we consider two natural (pseudo-)metrics. Let $\dr(\cdot,\cdot)$ 
denote the distance function on $X$ of the (fixed) $G$-invariant Riemannian metric on $X$.
Throughout the paper, we fix an $\iota$-invariant face $\tmod$ of $\smod$, and fix an $\iota$-invariant type $\tb$ in the interior of $\tmod$.

Let $\df$ denote the polyhedral  Finsler (pseudo-)metric\footnote{Our definition is same as the one in \cite[Subsec. 5.1.2]{MR3811766}. It is remarked in the second paragraph of \cite[p. 2571]{MR3811766} that  $\df(x,y) = -b^\tb \circ d_\Delta(x,y)$, where $b^\tb = -\langle \cdot | \tb \rangle$ is the Busemann function on the flat $F_\tmod\supset \Delta$, with the gradient $\tb$ and normalized at the origin.}
 on $X$:
\be\label{def:df}
\df(x,y) = \langle d_\D(x,y) | \tb \rangle.
\ee
The inner product  above is the euclidean inner product on $\fmod$ coming from the Riemannian metric on $X$.
Since $\tb$ is in the unit sphere of $\fmod$, and since the diameter of a Weyl chamber for the spherical metric is at most $\pi/2$,\footnote{This follows from the fact that the action $W\acts \mathfrak{a}$ is {\em essential}, i.e., $W$ does not fix any proper  subspace of $\mathfrak{a}$.} we have,
\begin{equation}\label{eqn:dist_ineq}
0\le \df(x,y) \le \dr(x,y).
\end{equation}

\begin{rem}
 {The distance function $d_\tb$ depends on the choice of $\tb\in\tmod$.}
\end{rem}

The metric space $(X,\dr)$ is a complete Riemannian manifold and, in particular, it is {\em geodesic}: Any two points in $X$ can be connected by a geodesic segment. The (pseudo-)metric space $(X,\df)$ is also a geodesic space.
The geodesics in $(X,\df)$ are called {\em Finsler geodesics}. 
All the Riemannian geodesics are also Finsler, however, the converse is generally false: There are non-Riemannian Finsler geodesics when $\rank(X) \ge 2$.
The precise description of all Finsler geodesics is given in \cite[Subsec. 5.1.3]{MR3811766}.
We merely use this description as a definition of Finsler geodesics.

\begin{defn}[Finsler geodesics]\label{defn:fg}
Let $I\subset \mb{R}$. A path $\ell : I \ra X$ is called a {\em Finsler geodesic} if there exists a pair of antipodal flags $\t_\pm\in\Ft$ such that $\ell(I)\subset P(\t_+,\t_-)$ and
\[
\ell(t_2) \in V(\ell(t_1),\st(\t_+)),
\quad\forall t_1\le t_2.
\]
Moreover, given an $\i$-invariant compact subset $\Theta\subset \ost(\tmod)$, a Finsler geodesic $\ell : I \ra X$ is called a {\em $\Theta$-Finsler geodesic} if, in addition to the above, it satisfies the following stronger condition:
\[
\ell(t_2) \in V(\ell(t_1),\ost_\Theta(\t_+)),
\quad\forall t_1\le t_2.
\]
\end{defn}

\begin{rem}
Finsler geodesics give alternative description of diamonds, namely, the $\tmod$-diamond $\diamondsuit_\tmod(x,y)$ is the union of all Finsler geodesics connecting the endpoints $x$ and $y$. See \cite[Subsec. 5.1.3]{MR3811766}.
\end{rem}

\begin{notation}
In this paper, we use the notation $\rg{xy}$ to denote the Riemannian geodesic segment connecting a pair of points $x,y\in X$. To denote a Finsler geodesic segment connecting $x$ and $y$, we use the notation $\fg{xy}$. 
\end{notation}

Below we let $*$ be either ``$\mathrm{Riem}$'' or $\tb$. Let $\G< G$ be a  subgroup, and $x,x_0\in X$. Define the {\em orbital counting function} $N_*(r,x,x_0):[0,\infty)\ra[0,\infty]$,
\[
N_*(r) = N_*(r,x,x_0) = \card \{\g\in \G \mid d_*(x,\g x_0)< r\}.
\]
Using $N_*(r)$, following \cite{MR1675889} and \cite{MR1935549}, we define the {\em critical exponent} $\d_*$ of $\G$ by
\be\label{def:CE}
\d_* = \limsup_{r\ra\infty} \frac{\log N_*(r)}{r}\in [0,\infty].
\ee
The critical exponents $\dF$ and $\dR$ will be called the {\em $\tb$-critical exponent} and {\em Riemannian critical exponent}, respectively. 

\begin{rem}\label{rem:remtwopointfour}
The discussion in \cite{MR1675889} and \cite{MR1935549} is mostly limited to the case when $\tb$ is regular, i.e., belongs to the interior of $\smod$. 
\end{rem}

The critical exponent is independent  of the chosen points $x$ and $x_0$. The proof is standard: Consider the Poincar\'e series
\begin{equation}\label{eqn:poincare}
g_s^*(x,x_0) = \sum_{\g\in \G} \exp({-sd_*(x,\g x_0)}).
\end{equation}
It is a well-known  fact that $g_s^*(x,x_0)$ converges if $s> \d_*(x,x_0)$ and diverges if $s< \d_*(x,x_0)$ where $\d_*(x,x_0)$ denotes the right side of (\ref{def:CE}). Using the triangle inequality, we obtain
\[
\exp\LR{-sd_*(x,x_0)} g_s^*(x_0,x_0) \le g_s^*(x,x_0) \le \exp\LR{sd_*(x,x_0)} g_s^*(x_0,x_0).
\]
Hence, convergence or divergence of $g_s^*(x,x_0)$ is independent of the choice of $x$ and so is $\d_*(x,x_0)$. For a similar reason, it is also independent of the choice of $x_0$.

\begin{defn}\label{def:convergencetype}
A discrete subgroup $\G$ of $G$ is of {\em $\tb$-convergence type} if the $\tb$-Poincar\'e series $g_s^{\tb}(x,x_0)$ converges at the critical exponent $\dF$. Otherwise, we say that $\G$ has {\em $\tb$-divergence type}.
\end{defn}

Since the action $\G\acts X$ is properly discontinuous, $\dR$ is bounded above by the {\em volume entropy} of $X$ which is finite.\footnote{Finiteness of the volume entropy of a symmetric space follows, for instance, from the fact that $X$ has curvature bounded below combined with the Bishop--G\"unter volume comparison theorem, see e.g. \cite[Sec. 11.10, Cor. 4]{BC}. 
 }
For the $\tb$-critical exponent, (\ref{eqn:dist_ineq}) implies the following lower bound,
\begin{equation}\label{eq:CEcomparison}
\dR \le \dF.
\end{equation}
Finiteness of $\dF$ is more subtle because, in general, $\df$ is only a pseudo-metric and therefore, the orbital counting function $N_\tb$ may take infinity as a value. 
However, if the angular radius of the model Weyl chamber $\smod$ with respect to $\tb$ is $<\pi/2$, then $\df$ is a metric equivalent to $\dr$ and, consequently, $\dF$ is finite in this case. In particular, when $G$ is simple, then diameter of $\smod$ is $<\pi/2$ and therefore, $\dF$ is finite. 

The following finiteness result holds in the general pseudo-metric case.

\begin{prop}\label{prop:finiteCE}
For a uniformly $\tmod$-regular subgroup $\G< G$, the $\tb$-critical exponent $\dF$ is finite. 
\end{prop}

\begin{proof}
When $\G$ is uniformly $\tmod$-regular, the distance functions $\dr$ and $\df$ restricted to an orbit $\G x$ are coarsely equivalent: There exist $L\ge1 ,A\ge0$ such that, for all $x_1,x_2\in \G x$,
\be\label{eqn:metricRF}
L^{-1}\dr(x_1,x_2) -A \le \df(x_1,x_2) \le \dr(x_1,x_2).
\ee
The right side of this inequality comes from (\ref{eqn:dist_ineq}).
From this we get $\dR \le \dF \le L\dR$. Since $\dR$ is finite, $\dF$ is also finite.
\end{proof}

\begin{rem}\label{rem:positiveCE}\leavevmode
\begin{enumerate}
\item It is clear from the proof of  Proposition \ref{prop:finiteCE} that when $\G$ is uniformly $\tmod$-regular, then $\dF$ is positive if and only if $\dR$ is positive.

\item As Anosov subgroups are uniformly regular (see Theorem \ref{thm:equiv}), the above proposition applies to the class of Anosov subgroups.
\end{enumerate}
\end{rem}

Before closing this section, we compute  Finsler distances $\df$ in two examples.

\begin{exmp}[Product of rank-one symmetric spaces]\label{ex:1.2}
We continue with the discussion from Example \ref{ex:1}. 
Let $\tmod = (r_1,\dots, r_p)$ be a face of the model chamber, let $\tb = (1/\sqrt{p}, \dots, 1/\sqrt{p})$ be its barycenter, and let $\df$ be the corresponding metric on $X$. 
Using the formula for $d_\Delta$ from (\ref{eqn:DeltaValuedRankOne}), we get
 \be\label{eqn:dfrank1}
 \df(x,y) = \frac{1}{\sqrt{p}}\sum_{j=1}^{p}{d_{r_j}(x_{r_j},y_{r_j})}.
 \ee
\end{exmp}

\begin{exmp}[$X = \SL{k+1}/\SO{k+1}$]\label{ex:2.1}
We continue with the discussion from Example \ref{ex:2}. 
The Riemannian metric
on $X$ is given by the restriction of the Killing form $B$ of $\mf{g} = \mf{sl}(k+1,\R)$ to $\mf{p}$,
\be\label{eqn:SLKilling}
B(P,Q) = 2(k+1) \tr(PQ),\quad P,Q\in\mf{g}.
\ee
Note that the inner product $B$  on $\a$ (which we identify with $\fmod$) can be written as
\be\label{eqn:SLmetric}
 \norm{(\sigma_1,\dots,\sigma_{k+1})|(\sigma'_1,\dots,\sigma'_{k+1})}=2(k+1) \sum_{i=1}^{k+1}\sigma_i \sigma'_i.
\ee

Let $\tmod=(r_1,\dots,r_p)$ be an $\i$-invariant face of the model chamber $\smod$ and let $\D_\tmod$ be the corresponding face of the model euclidean Weyl chamber $\D$,
\[
\D_\tmod = \big\{ \pmb{\sigma}\in\a_+ 
\mid  \pmb{\sigma} = ( \underbrace{\sigma_1,\dots,\sigma_1}_{r_1 \text{-times}},\dots,\underbrace{\sigma_i,\dots,\sigma_i}_{(r_{i} - r_{i-1})\text{-times}},\dots,\underbrace{\sigma_{p+1}\dots, \sigma_{p+1}}_{(k+1-r_p) \text{-times}}) \big\}.
\]
For notational convenience we denote $\pmb{\sigma}$ in the above expression simply by the $(p+1)$-vector $(\sigma_1,\dots,\sigma_{p+1})$ (by identifying the repeated entries). With this convention, the opposition involution acts by \[\i(\sigma_1,\dots,\sigma_{p+1}) = (-\sigma_{p+1},\dots,-\sigma_{1}).\]
We identify $\tmod$ with the unit sphere (w.r.t. the metric in (\ref{eqn:SLmetric})) in $\D_\tmod$ centered at the origin, i.e., $\tmod$ consists of all elements $(\sigma_1,\dots,\sigma_{p+1}) \in \D_\tmod$ satisfying
$2(k+1)\sum_{i=1}^{p+1} (r_i - r_{i-1}) \sigma_i^2 = 1$.
An element $\tb = (\sigma_1,\dots,\sigma_{p+1})\in\tmod$ lies in the interior of $\tmod$ if and only if $\sigma_1>\dots>\sigma_{p+1}$. Moreover, $\tb$ is $\i$-invariant if and only if $\sigma_{i} + \sigma_{p+2-i} = 0$ for all $i= 1,\dots,p+1$.

The Finsler distance $\df$ can be calculated explicitly in terms of the above formulas.
In the special case\footnote{We will only focus on this case from now on.} when $\tmod = (1,k)$ and 
\[
\tb = (1/2\sqrt{k+1}, 0 , -1/2\sqrt{k+1}),
\]
the unique $\iota$-invariant type in the interior of $\tmod$, we have a simple formula for the Finsler distance: For all $g\in \SL{k+1}$, 
\be\label{eq:ex:2.1}
\df(x_0,g x_0) = \sqrt{k+1} \LR{\sigma_1 (g) - \sigma_{k+1}(g)},
\ee
where the point $x_0$ is the identity matrix.
\end{exmp}

\section{$\tb$-conformal densities}\label{sec:CD}

Recall that Busemann functions define the notion of ``distance from infinity.'' For $\t\in\Ft$, let $b^\tb_\t: X\ra \R$ denote the Busemann function based at the ideal point $\tb(\t) \in \vb X$ normalized at $x_0$, i.e., $b^\tb_\t(x_0) = 0$. Using Busemann functions, one defines the {\em Busemann cocycle} as
\be\label{eqn:horodist}
\dhor_\t(x,y) = b^\tb_\t(x) - b^\tb_\t(y).
\ee
$\dhor_\t$ satisfies the {\em cocycle condition:} For each triple $x,y,z\in X$,
\be\label{eqn:cocycle}
\dhor_\t(x,y) + \dhor_\t(y,z) = \dhor_\t(x,z).
\ee
 These functions are related to the Finsler distance functions by
\be\label{eq:B-limit}
\dhor_\t(x,y)= \lim_{n\to \infty} \left( \df(x, z_n) - \df(y, z_n)\right) 
\ee
whenever $(z_n)$ is a sequence in $X$ flag-converging to $\tau$, cf. \cite[Prop. 5.43]{MR3811766}. 
Note that 
\[
-\df(x,y) \le \dhor_\t (x,y) \le \df(x,y).
\]

\begin{rem}
 As usual, the Busemann functions and cocycles depends on the choice of $\tb$. Also, note that these functions can take negative values. However, $|\dhor_\t(x,y)|$ satisfy the triangle inequality and, hence, are pseudo-metrics on $X$.
\end{rem}

We define our notion of ``{conformal densities}'' on $\Ft$ using these Busemann cocycles.

For a topological  space $S$, we let $M_+(S)$ denote the set of Borel probability measures on $S$. Recall that a group $H$ of 
self-homeomorphisms of $S$  acts on $M_+(S)$ by {\em pull-back}: For every $B\in\B(S)$, $h\in H$,
\[
\mu\mapsto h^*\mu, \quad h^*\mu(B)=\mu(h^{-1}(B)). 
\]

\begin{defn}[$\tb$-conformal density]\label{def:conformal_density}
 Let $\G< G$ be a discrete 
subgroup and let $A\subset X$ be a nonempty $\G$-invariant subset. By a {\em $\G$-invariant $\tb$-conformal $A$-density} $\mu$ of {\em dimension $\beta\in[0,\infty)$}  on $\Ft$, we mean a $\G$-equivariant map
\[
\mu : A \ra M_+(\Ft), \quad a\mapsto \mu_a,
\]
 satisfying the following properties:
\begin{enumerate}[(i)]
\setlength\itemsep{0em}

\item For each $a\in A$, $\supp(\mu_a)\subset \L_\tmod(\G)$. \label{def:CDitemi}

\item {\em $\G$-invariance:} $\mu$ is $\G$-invariant, i.e., $\g^*\mu_a = \mu_{\g a}$ for each $\g\in \G$ and each $a\in A$. \label{def:CDitemii}

\item {\em $\tb$-conformality:} For every pair $a,b \in A$, $\mu_a \ll \mu_b$, i.e., $\mu_a$ is absolutely continuous with respect to $\mu_b$, and the Radon--Nikodym derivative $d\mu_a/d\mu_{a}$ can be expressed as \label{def:CDitemiii}

\begin{equation}\label{eqn:RN}
\frac{d\mu_a}{d\mu_{b}}(\t) = \exp\left({-\beta \dhor_{\t}(a,b)}\right),
\quad \forall\t\in\Ft.
\end{equation}

\end{enumerate}

\end{defn}

\begin{rem}\leavevmode

\begin{enumerate}
 \item In the literature, the item (\ref{def:CDitemi}) in the above definition is not required to define $\G$-invariant conformal densities.
 Nevertheless, the uniqueness of such conformal densities  for $\tmod$-Anosov subgroups, which we establish later in the paper (see Corollary \ref{cor:uniqueAnosovDensity}), is {\em possibly} false unless we impose this extra condition.
 
 \item Though we define $\tb$-conformal densities for general discrete subgroups of $G$, for the purpose of this paper we restrict our discussion only to $\tmod$-regular subgroups.
\end{enumerate}

\end{rem}

A $\tb$-conformal $X$-density is simply called a {\em $\tb$-conformal density}. Note that $\tb$-conformal $X$-densities and $\tb$-conformal $A$-densities are in a one-to-one correspondence: 
\be\label{eqn:Adensity}
\{\text{$\tb$-conformal $X$-densities}\} \longleftrightarrow \{\text{$\tb$-conformal $A$-densities}\}.
\ee
From a $X$-density $\mu$, define an $A$-density by restricting the family. On the other hand, given an $A$-density $\mu$, extend it to a $X$-density $\{\mu_x\}_{x\in X}$ by 
\[
d\mu_x(B) = \int_B \exp\left({-\beta \dhor_{\t}(x,a)}\right) d\mu_{a}(\t), \quad B\in\B(\Ft)
\]
where $\mu_a$ is a density in the family $\mu$. 
Note that this extension is unique because $\mu_x$ and $\mu_a$ are absolutely continuous with respect to each other.
To check $\G$-invariance, note that
\begin{align*}
\g^*\mu_x(B) 
&= \int_{\g^{-1}B} \frac{d\mu_x}{d\mu_a}(\t)d\mu_a(\t)
= \int_{B} \exp\left({-\beta \dhor_{\g^{-1}\t}(x,a)}\right) d\mu_a(\g^{-1}\t)\\
&=  \int_{B} \exp\left({-\beta \dhor_{\t}(\g x,\g a)}\right) d\mu_{\g a}(\t)
=  \int_{B} \frac{d\mu_{\g x}}{d\mu_{\g a}}(\t)d\mu_{\g a}(\t) = \mu_{\g x}(B),
\end{align*}
for every $B \in \B(\Ft)$. The other two defining properties are also satisfied.

\subsection{The Patterson-Sullivan construction}
For every $\tmod$-regular group $\G$ with finite $\tb$-critical exponent $\dF$, we follow the the Patterson--Sullivan construction to construct a $\tb$-conformal density. This construction  is standard and already appeared in the work of Albuquerque and Quint, although only in the setting of Zariski dense subgroups $\G< G$ and regular vectors $\tb$; we present it here for the sake of completeness. We let $\G< G$ be a $\tmod$-regular subgroup and let 
$Z$ denote the $\G$-orbit of a point $x_0\in X$. The union 
\[
\bar Z= Z\cup \LT \subset \Xt,
\]
equipped with the topology of flag-convergence, is a compactification of $Z$. 

For $s> \dF$, the $\tb$-critical exponent of $\G$, we define a family of positive measures  $\mu^\tb_s = \{\mu^\tb_{x,s}\}_{x\in X}$ on $\bar Z$ by 
\begin{equation}\label{eqn:PSM}
\mu^\tb_{x,s} =  \frac{1}{ g^{\tb}_s(x_0,x_0)}\sum_{\g\in \G} \exp\LR{-s\df(x, \g x_0)}D(\g x_0),
\end{equation}
where $D(\g x_0)$ denotes the Dirac point mass of weight one at $\g x_0$. Note that $\mu^\tb_{x,s}$ is a probability measure when $x\in Z$. 
Also, note that $\Lt(\G)$ is a null set for these measures. For $\g\in \G$ a straightforward computation shows that
\begin{equation}\label{eqn:equivariance1}
\g^*\mu^\tb_{x,s} = \mu^\tb_{\g x,s}.
\end{equation}
Moreover, it is easy to see that $\mu^\tb_{s}$ is an absolutely continuous family of measures. 
Using (\ref{eqn:PSM}) we compute the Radon--Nikodym derivatives $d\mu^\tb_{x,s}/d\mu^\tb_{x_0,s}$,
\be\label{eqn:PSder}
\psi^\tb_{x,x_0,s}(z) = \frac{d\mu^\tb_{x,s}}{d\mu^\tb_{x_0,s}}(z),
\ee
where for $s\ge 0$,
\[
\psi^\tb_{x,x_0,s}(z) := \exp\LR{-s\LR{\df(z,x) - \df(z,x_0)}}.
\]
The formula for $\psi^\tb_{x,x_0,s}$ above only makes sense when $z\in Z$.
Since $\LT$ is a null set, we extend $\psi^\tb_{x,x_0,s}$ continuously to $\Lt(\G)$ by setting
\[
\psi^\tb_{x,x_0,s}(\t) = \exp\LR{-s\dhor_\t(x,x_0)}.
\]
The continuity of this function can be verified using  properties of $\df$ (e.g., see \cite[Sec. 5.1.2]{MR3811766} and \eqref{eq:B-limit}). 

Next we prove that $\psi^\tb_{x,x_0,s} \ra \psi^\tb_{x,x_0,\dF}$ uniformly as $s\ra \dF$. 
For $S\ge s, s'>\dF$  and $z\in Z$,
\begin{align*}
|\psi^\tb_{x,x_0,s'}&(z) - \psi^\tb_{x,x_0,s}(z)| \\
&= | \exp\LR{-s'\LR{\df(z,x) - \df(z,x_0)}} - \exp\LR{-s\LR{\df(z,x) - \df(z,x_0)}} |\\
&= \exp\LR{-s\LR{\df(z,x) - \df(z,x_0)}} | \exp\LR{(s-s')\LR{\df(z,x) - \df(z,x_0)}} -1 |\\
&\le \exp\LR{S \dr(x,x_0)}| \exp\LR{(s-s')\LR{\df(z,x) - \df(z,x_0)}} -1 |.
\end{align*}
Switching $s$ and $s'$ in the above, we also get
\begin{align*}
|\psi^\tb_{x,x_0,s'}&(z) - \psi^\tb_{x,x_0,s}(z)|\\
&\le \exp\LR{S \dr(x,x_0)}| \exp\LR{(s'-s)\LR{\df(z,x) - \df(z,x_0)}} -1 |.
\end{align*}
Combining the above two inequalities, we get
\begin{align*}
|\psi^\tb_{x,x_0,s'}&(z) - \psi^\tb_{x,x_0,s}(z)|\\
&\le \exp\LR{S \dr(x,x_0)} \LR{\exp\LR{|s'-s|\cdot|\df(z,x) - \df(z,x_0)|} -1}\\
&\le \exp\LR{S \dr(x,x_0)} \LR{\exp\LR{|s'-s|\dr(x,x_0)} -1}.
\end{align*}
Since $Z$ is dense in $\bar{Z}$, the above yields
\[
\| \psi^\tb_{x,x_0,s'} - \psi^\tb_{x,x_0,s} \|_\infty \le  \exp\LR{S \dr(x,x_0)} \LR{\exp\LR{|s'-s|\dr(x,x_0)} -1}
\]
Therefore, $\psi^\tb_{x,x_0,s} \ra \psi^\tb_{x,x_0,\dF}$ uniformly as $s\ra \dF$. 

Now we construct a $\tb$-conformal density as a limit of the family of densities $\{\mu^\tb_s\}_{s>\dF}$.
We first assume that $\G$ has $\tb$-divergence type.\footnote{This will be the case for Anosov subgroups. See Corollary \ref{cor:div}.} Then, as $s$ decreases to $\dF$, the family $\mu^\tb_s = \{\mu^\tb_{x,s}\}_{x\in X}$ weakly accumulates to a density $\mu^\tb$ supported on some subset of $\Lt(\G)$. 
By (\ref{eqn:equivariance1}) we have the $\G$-invariance of $\mu^\tb$, namely, for $\g\in\G$,
\begin{equation}\label{eqn:invariance}
\g^*\mu^\tb_{x} = \mu^\tb_{\g x}.
\end{equation}
Moreover, since $\mu^\tb_x$ is obtained as a weak limit of the measures $\mu^\tb_{x,s}$ and the derivatives $\psi^\tb_{x,x_0,s} =   d\mu^\tb_{x,s}/d\mu^\tb_{x_0,s}$ converge uniformly to $\psi^\tb_{x,x_0,\dF}$, it follows that the Radon-Nikodym derivative $d\mu^\tb_{x}/d\mu^\tb_{x_0}$ exists and equals to the limit
\[
\lim_{s\ra \dF}\frac{d\mu^\tb_{x,s}}{d\mu^\tb_{x_0,s}} = \psi^\tb_{x,x_0,\dF},
\]
or more explicitly,
\be\label{eqn:derivative}
\frac{d\mu^\tb_{x}}{d\mu^\tb_{x_0}}(\t) = \exp\LR{- \dF \dhor_\t(x,x_0)}.
\ee
Note that in general weak limits are not unique.
In Corollary \ref{cor:uniqueAnosovDensity} we will prove 
that for Anosov subgroups $\Gamma$ we get a unique density in this limiting process.

When $\G$ has $\tb$-convergence type, we change weights of the Dirac masses by a small amount (\cite[Lem. 3.1]{MR0450547}, see also \cite[Sec. 3.1]{nicholls1989ergodic}) in the definition (\ref{eqn:PSM}) to force the Poincar\'e series to diverge. Define
\[
\mu^\tb_{x,s} =  \frac{1}{ {\bar g}^{\tb}_s(x_0,x_0)}\sum_{\g\in \G} \exp\LR{-s\df(x,\g x_0)}h\LR{\df(x,\g x_0)} D(\g x_0)
\]
where $h: \R_+ \ra \R_+$ is a subexponential function such that the following modified Poincar\'e series 
\[
{\bar g}^\tb_{s}(x,x_0) = \sum_{\g\in \G} \exp\LR{-s\df( x,\g x_0)} h\LR{\df(\g x,x_0)}
\]
diverges at $s = \dF$. In this case also, a limit density $\mu^\tb$ has the properties (\ref{eqn:invariance}) and (\ref{eqn:derivative}).

\begin{defn}[Patterson-Sullivan density]\label{definition:PSdensity}
 Let $\G$ be a $\tmod$-regular subgroup of $G$ such that $\dF(\G) <0$, and let $\tb\in\tmod$ be an $\iota$-invariant interior point.
 Any weak limit $\mu^\tb$ appears from the construction above is called a {\em Patterson-Sullivan density} of {\em type $\tb$}.
\end{defn}

\subsection{Positivity of the $\tb$-critical exponent}
The existence of a $\tb$-conformal density implies that the $\tb$-critical exponent of $\G$ is positive.

\begin{prop}\label{cor:positiveCE}
Suppose that $\G$ is a nonelementary $\tmod$-regular antipodal subgroup and $\dF$ is finite.
Then, $\dF$ is also positive. 
\end{prop}
\begin{proof}
Suppose to the contrary that $\dF = 0$. Let $\mu^\tb$ be a Patterson--Sullivan density constructed above. It follows from the $\G$-invariance and $\tb$-conformality that for all $\g\in\G$,
\be\label{eqn:zeroCrit}
\mu^\tb_x(\g A) = \mu^\tb_{\g^{-1} x}(A) = \mu^\tb_x(A),
\quad\forall A\in\B(\LT).
\ee

We use the convergence action property of a $\tmod$-RA subgroups (see Subsection \ref{sec:discretegroups}). 

We first show that $\mu^\tb$ is atom-free. For if $\t\in\LT$ were an atom, then, since $\G\t$ is an infinite orbit, $\LT$ would have infinite $\mu^\tb_x$-mass by (\ref{eqn:zeroCrit}).

Moreover, using the converge action $\G\acts\LT$, we have an infinite sequence $\g_n\in\G$ and points $\t_\pm\in\LT$ such that on $\LT - \{\t_-\}$,
\[
 \g_n|_{\LT - \{\t_-\}} \ra \mathrm{const}_{\t_+} 
\]
uniformly on compact sets.

Now, pick a compact set $A\subset \Lt(\G)$ not containing $\t_\pm$ such that 
\[
\mu^\tb_x(A) \ge (2/3)\mu^\tb_x(\LT)
\]
(this is possible because $\mu^\tb$ is atom-free). 
Pick a large enough $n$ so that $\g_n(A) \cap A = \emptyset$. Then,
$$\mu^\tb_x(\LT) \ge \mu^\tb_x(A\cup\g_{n}A ) = \mu^\tb_x(A) + \mu^\tb_x(\g_{n}A) =2\mu^\tb_x(A) \ge \frac{4}{3}\mu^\tb_x(\LT),$$
yields a contradiction.
\end{proof}

\begin{rem}
As a corollary to the above proposition, the Riemannian critical exponent $\dR$ of a nonelementary uniformly $\tmod$-regular  antipodal subgroup is also non-zero. See the remark after Proposition \ref{prop:finiteCE}.
\end{rem}

Since Anosov subgroups are uniformly regular and antipodal, we have the following result.

\begin{cor}
 Let $\G$ be a nonelementary $\tmod$-Anosov subgroup of $G$.
 Then,
 \[
  0< \dR \le \dF <\infty.
 \]
\end{cor}

\section{Hyperbolicity of Morse image}\label{sec:hypmorse}

In this section we prove that the image of a $\tmod$-Morse map  is Gromov-hyperbolic with respect to the Finsler pseudo-metric $\df$. As a corollary, we prove that each orbit of an Anosov subgroup is also Gromov-hyperbolic with respect to $\df$.

We first recall two notions of hyperbolicity. We shall use the symbol $\delta$ to denote the hyperbolicity constant in this and the next sections only. We hope that the reader will not confuse this $\delta$ with the critical exponent.

\begin{defn}[Rips hyperbolic]
Let $(Z,d)$ be a geodesic metric space. Then, $(Z,d)$ is called {\em $\delta(\ge 0)$-hyperbolic in the sense of Rips} (or {\em Rips hyperbolic}) if every geodesic triangle $\triangle$ is $\d$-thin, i.e., each side of $\triangle$ lies in the $\delta$-neighborhood of the union of the other two sides.
\end{defn}

\begin{defn}[Gromov hyperbolic]\label{def:GH}
Let $(Z,d)$ be a metric space. For any three points $z,z_1,z_2\in Z$, the {\em Gromov product} is defined as
\[
\norm{z_1| z_2}_z = \frac{1}{2}[d(z,z_1) + d(z_2,z) - d(z_1,z_2)].
\]
Then $(Z,d)$ is called {\em $\d(\ge 0)$-hyperbolic in the sense of Gromov} (or {\em Gromov hyperbolic}) if the Gromov product satisfies the following {\em ultrametric inequality}: For all $z,z_1,z_2,z_3\in Z$,
\[
\norm{z_1|z_2}_z \ge \min\{\norm{z_1|z_3}_z, \norm{z_2|z_3}_z\} - \delta.
\]
\end{defn}

It should be noted that  Gromov's definition applies to all metric spaces whereas Rips' definition works only for geodesic metric spaces. Moreover, Gromov hyperbolicity is not quasiisometric invariant whereas Rips hyperbolicity is (as a consequence of Morse lemma, cf. \cite[Cor. 11.43]{MR3753580})). For geodesic metric spaces,  these two notions of hyperbolicity are equivalent (e.g., see \cite[Lemma 11.27]{MR3753580}). 

Let $(Z',d')$ be Rips hyperbolic and $f:(Z',d')\ra (X,\dr)$ be a $\tmod$-Morse map. We denote the image $f(Z')$ by $Z$. Recall that the $\df$ is coarsely equivalent to   $\dr$ on $Z$.\footnote{This is also true for any finite Riemannian tubular neighborhood of $Z$.} Therefore, since $f$ is a quasiisometric embedding with respect to $\dr$, it is also a quasiisometric embedding with respect to $\df$.
Moreover, the image of a geodesic (of length bounded below by a constant) in $Z'$ stays within a uniformly bounded Riemannian distance, say $\lambda_0\ge 0$, from a $\tmod$-regular Finsler geodesic connecting the images of the endpoints. This is a consequence of the Morse property \cite[Thm. 1.1]{kapovich2014morse}, see also \cite[{Prop. 12.2}]{MR3811766}.
A consequence of this is that $Z$ is $\lambda_0$-quasiconvex in $(X,\df)$. 

For $\lambda \ge \lambda_0$, let $Y= Y_\lambda$ be the Riemannian $\lambda$-neighborhood of $Z$ in $X$. From the discussion above, it is clear that any  two points $z_1,z_2\in Z$ (with $\dr(z_1,z_2)$ sufficiently large)  can be connected by a Finsler geodesic $\fg{z_1z_2}$ in $Y$.

\begin{prop}\label{prop:uniqueFinsler}
Let $c$ and $c'$ be two Finsler geodesics in $Y$ connecting two points $z_1,z_2$. Then they are uniformly Hausdorff close. Here the Hausdorff distance is induced by either of $\dr$ or $\df$.
\end{prop}

\begin{proof}
Since $\dr$ or $\df$ are comparable on $Y$, it is enough to prove the proposition for the Riemannian metric $\dr$.

Let $\bar c$ and $\bar c'$ be the respective nearest point projections of $c$ and $c'$ to $Z$. Applying the coarse inverse of $f$, $\bar c$ and $\bar c'$ map to uniform quasigeodesics $\tilde c$ and $\tilde c'$, respectively, in $Z'$. Since $Z'$ is Rips hyperbolic, $\tilde c$ and $\tilde c'$ are uniformly close. Applying $f$ to $\tilde c$ and $\tilde c'$, we see that $\bar c$ and $\bar c'$ are uniformly close. Hence $c$ and $c'$ are also uniformly close.
\end{proof}

Next we observe that geodesic triangles in $(Y,\df)$ with vertices on $Z$ are uniformly thin.

\begin{prop}\label{prop:thintriangle}
There exists $\d\ge 0$ such that every Finsler geodesic triangle $\triangle$  in $Y$  is $\delta$-thin both in Riemannian and Finsler sense.
\end{prop}

\begin{proof}
Since $Z'$ is Rips hyperbolic, geodesic triangles in $Z'$ are $\d'$-thin, for some $\d'\ge 0$. We map $\triangle$ to 
a uniformly quasigeodesic triangle  $\triangle' \subset Z'$ via the coarse inverse map $Y\to Z'$ of the map $f$. 
Since $Z'$ is Rips-hyperbolic, the Morse quasigeodesic triangle $\triangle'$ is uniformly thin. Therefore, 
 $\triangle$ is also uniformly thin as well.
\end{proof}

Imitating the proof of \cite[Lem. 11.27]{MR3753580}, we prove that $(Z,\df)$ is Gromov-hyperbolic

\begin{thm}[Hyperbolicity of Morse maps]\label{thm:hyp}
Let $Z\subset X$ be the image of a $\tmod$-Morse map $f:(Z',d')\ra (X,\dr)$.
Then $(Z,\df)$ is Gromov-hyperbolic.
\end{thm}

\begin{proof}
Let $\d$ be as in Proposition \ref{prop:thintriangle}. Then the following holds.

\begin{lem}
Let $z,z_1,z_2\in Z$, and let $\fg{z_1z_2}$ be any Finsler geodesic in $Y$ connecting $z_1$ and $z_2$. Then,
\[
\norm{z_1|z_2}_z \le \df(z,\fg{z_1z_2}) \le \norm{z_1|z_2}_z + 2\d.
\]
\end{lem}
\begin{proof}
The proof is exactly same as  \cite[Lem. 11.22]{MR3753580}. Note that the proof uses $\d$-thinness of a triangle with vertices $z,z_1,z_2$.
\end{proof}

Let $z,z_1,z_2,z_3$ be any four points in $Z$, and let $\triangle$ be a Finsler geodesic triangle in $Y$ with the vertices $z_1,z_2,z_3$.
 Let $m$ be a point on the side $\fg{z_1z_2}$ nearest to $z$. By Proposition \ref{prop:thintriangle}, since $\triangle$ is $\d$-thin, $\df(m, \fg{z_2z_3}\cup \fg{z_1z_3}) \le \d$.
Without loss of generality, assume that there is a point $n$ on $z_2,z_3$ which is $\d$-close to $m$. Then, using the above lemma, we get
\[\norm{z_2|z_3}_z \le \df(z,\fg{z_2z_3}) \le \df(z,\fg{z_1z_2}) + \d,\]
and
\[\df(z,\fg{z_1z_2}) \le \norm{z_1|z_2}_z + 2\d.\]
The theorem follows from this.
\end{proof}

Quasiisometry of hyperbolic metric spaces extends to a homeomorphism of their Gromov boundaries. At the same time, it is proven 
in \cite[Thm. 1.4]{kapovich2014morse} that each $\tmod$-Morse map 
$$
f: Z'\to Z=f(Z')\subset X
$$
extends continuously (with respect to the topology of flag-convergence) to a homeomorphism
$$
\partial_{\infty} f: \partial_{\infty} Z'\to \L\subset \Ft. 
$$

Thus, we obtain

\begin{cor}\label{cor:id-bdry}
The Gromov boundary $\partial_{\infty} Z$ of $(Z,\df)$ is naturally identified with the flag-limit set $\L\subset \Ft$ of $Z$: A sequence 
$(z_n)$ in $Z$ converges to a point in $\partial_{\infty} Z$ if and only if $(z_n)$ flag-converges to some $\tau\in \L$.  
\end{cor}

For a $\tmod$-Anosov subgroup $\G$ we know that the orbit map $\G \ra \G x_0$ is a $\tmod$-Morse embedding (see Subsection \ref{sec:discretegroups}). Then, using Theorem \ref{thm:hyp} we obtain: 

\begin{cor}[Hyperbolicity of Anosov orbits]\label{cor:hypano}
For $x_0\in X$, let $Z = \G x_0$ where $\G$ is a $\tmod$-Anosov subgroup. Then $(Z,\df)$ is Gromov hyperbolic. The Gromov boundary of $(Z,\df)$ is naturally identified with the $\tmod$-limit set $\LT$. 
\end{cor}

\section{Gromov distance at infinity}\label{sec:dg}

For a pair of antipodal simplices $\t_\pm\in\Ft$, the {\em Gromov product} with respect to a base point $x\in X$ is defined as
\be\label{eqn:GP}
\norm{\t_+|\t_- }^\tb_x = \frac{1}{2}\LR{\dhor_{\t_+}(x,z) + \dhor_{\t_-}(x,z)},
\ee
where $z$ is some point on the parallel set $P(\t_+,\t_-)$ spanned by $\t_\pm$.

The definition Busemann cocycles $\dhor_\t$ given in (\ref{eqn:horodist}) is free of choice of any particular normalization for the Busemann functions. 
We use this observation in the proof of the following lemma which shows that the Gromov products do not depend on the chosen $z\in P(\t_+,\t_-)$.

\begin{lem}\label{lem:GPinvariant}
For $z_1,z_2 \in P(\t_+,\t_-)$, one has $b^\tb_{\t_+}(z_1) + b^\tb_{\t_-}(z_1) = b^\tb_{\t_+}(z_2) + b^\tb_{\t_-}(z_2)$. 
\end{lem}

\begin{proof}
Let $z$ be the midpoint of $\rg{z_1z_2}$ and let $s_z: X\ra X$ be the point reflection about $z$. Assuming that Busemann functions are normalized at $z$, $s_z$ transforms $b^\tb_{\t_+}(z_1) + b^\tb_{\t_-}(z_1)$ into $b^\tb_{\t_-}(z_2) + b^\tb_{\t_+}(z_2)$. Hence the quantities are equal.
\end{proof}

Using the Gromov product, we define a {\em premetric}\footnote{A {\em premetric} on $X$ is a symmetric, continuous function $d:X\times X\ra [0,\infty)$ such that $d(x,x) = 0$ for all $x\in X$.} on $\Ft$.

\begin{defn}[Gromov premetric]\label{def:dg}
Given fixed $x\in X$, $\e>0$, define the {\em Gromov premetric} $\dg{\e}_x$ on $\Ft$ as
\[
\dg{\e}_x(\t_1,\t_2) = \left\{ 
\begin{array}{ll}
\exp\LR{-\e \norm{\t_1|\t_2}^\tb_x}, & \text{if $\t_1,\t_2$ are antipodal,}\\
0, &\text{otherwise.}
\end{array}
\right.
\]
\end{defn}

\begin{rem}
 A pair of points $\t_\pm\in \Ft$ is antipodal if and only if \[\dg{\e}_x(\t_+,\t_-)\ne 0.\]
\end{rem}

\begin{lem}
$\dg{\e}_x$ is a continuous function.
\end{lem}
\begin{proof}
The claim follows from \cite[Lem. 3.8]{beyrer}.
\end{proof}

\begin{defn}[Conformal maps]
Let $(Z,d)$ be a premetric space.
 A self-homeo\-morphism $f: (Z,d) \ra (Z,d)$ is called {\em $K$-quasiconformal} if, for every $z\in Z$, $\limsup_{r\ra 0} H_f(z,r) \le K$, where
\[
 H_f(z,r) := \frac{\sup\{d(f(y),f(z)) \mid d(y,z) \le r\}}{\inf\{d(f(y),f(z)) \mid d(y,z) \ge r\}}.
\]
The map $f$ is called {\em conformal} if it is $1$-quasiconformal.
\end{defn}

\begin{lem}\label{prop:dgconf}
Let $\gamma\in G$ and $\L\subset\Ft$ be a $\gamma$-invariant antipodal subset. Then the map $\gamma: \L\to \L$ is conformal with respect to the premetric $\dg{\e}_x$. 
\end{lem}
\begin{proof}
Given distinct points $\t_\pm\in\L$,
\begin{align*}
\dg{\e}_x(\g\t_+,\g\t_-)
&= \exp\LR{-\e \norm{\g\t_+|\g\t_- }^\tb_x}\\
&= \exp\LR{-\frac{\e}{2}\LR{\dhor_{\g\t_+}(x,z) + \dhor_{\g\t_-}(x,z)}}\\
&= \exp\LR{-\frac{\e}{2}\LR{\dhor_{\t_+}(\g^{-1}x,\g^{-1}z) + \dhor_{\t_-}(\g^{-1}x,\g^{-1}z)}}\\
&= \exp\LR{-\frac{\e}{2}\LR{\dhor_{\t_+}(\g^{-1}x,x) + \dhor_{\t_-}(\g^{-1}x,x)}}\dg{\e}_x(\t_+,\t_-),
\end{align*}
where the last equality follows from the cocycle condition (\ref{eqn:cocycle}).
Moreover, the continuity of Busemann functions $b^\tb_\t$ as a function of $\t$ implies that 
\[
\lim_{\t_-\ra \t_+}\dhor_{\t_-}(\g^{-1}x,x) = \dhor_{\t_+}(\g^{-1}x,x).
\]
Therefore,
\be\label{eqn:expansion1}
\lim_{\t_-\ra\t_+}\frac{\dg{\e}_x(\g\t_+,\g\t_-)}{\dg{\e}_x(\t_+,\t_-)} = 
E(\g,\t_+) := \exp\LR{-{\e}\dhor_{\t_+}(\g^{-1}x,x)}.
\ee
From this, it can be checked that $\limsup_{r\ra 0} H_\g(\t_+,r) = 1$.
\end{proof}

The premetric $\dg{\e}_x$ is not a metric in general since:
\begin{enumerate}[(i)]

\item Pairs of distinct non-antipodal points have zero distance.

\item The triangle inequality may fail. 
\end{enumerate}
However, as we shall see below, for all sufficiently small $\e>0$, $\dg{\e}_x$ is bilipschitz equivalent to an actual distance function when restricted to ``nice''  antipodal subsets $\L\subset\Ft$.

\begin{thm}\label{thm:dg}
{Let $Z\subset X$ be the image of a $\tmod$-Morse map $f:(Z',d')\ra (X,d)$, and let  $\L\subset\Ft$ be the flag limit set of $Z$. There exists $\e_0>0$ such that, for all $0<\e\le \e_0$ and all $x\in Z$, the premetric $\dg{\e}_x$ is 2-bilipschitz equivalent to a metric on $\L$.
Moreover, the topology induced by $\dg{\e}_x$ on $\L$ coincides with the subspace topology of $\L\subset \Ft$.}
\end{thm}
\begin{proof}
The idea of the proof of the first part is due to Gromov \cite{MR919829}: We show that the Gromov product defined in (\ref{eqn:GP}) restricted to $\L$ satisfies an ultrametric inequality (see (\ref{eqn:GPultrametric})).

Let $Y\subset X$ be a Riemannian $\lambda$-neighborhood of $Z$. We assume that $\lambda$ here is so large such that $x\in Y$ and the image of any complete geodesic $l$ in $Z'$ lies within distance $\lambda$  from the parallel set spanned by the images of the ideal endpoints of $l$ under $\bar f: \partial_{\infty} Z'\to \Ft$. Note that $\lambda$ satisfying the last condition exists 
as a consequence of the Morse property.

Observe that $(Y,\df)$ is a Gromov $\d$-hyperbolic\footnote{Since there is a possibility of confusion, we would like to remind our reader that this $\delta$ is {\em not} a critical exponent.} metric space for some $\d\ge 0$. 
This follows from the Gromov hyperbolicity of $(Z,\df)$ (cf. Theorem \ref{thm:hyp}) and the fact that $Z$ and $Y$ are (Hausdorff) $\lambda$-close to each other.

We recall from V\"{a}is\"{a}l\"{a} \cite[Sec. 5]{MR2164775} that there are multiple ways to define Gromov products on $\L$ viewed as the Gromov boundary of $(Z,\df)$ and, hence, of $(Y,\df)$. 
For a distinct pair $\t_\pm \in \L$, define using the Gromov product $\norm{\cdot|\cdot}^\tb_x$ on $(Y,\df)$ the following two products:
\[
{\norm{\t_+|\t_-}}^{\inf}_x = \inf\bigg\{ \liminf_{i,j\ra\infty} \norm{y_i^+|y_j^-}^\tb_x \mid (y_n^\pm)\subset Y,y_n^\pm\ra\t_\pm \bigg\}
\]
and 
\[
{\norm{\t_+|\t_-}}^{\sup}_x = \sup\bigg\{ \limsup_{i,j\ra\infty} \norm{y_i^+|y_j^-}^\tb_x \mid (y_n^\pm)\subset Y, y_n^\pm\ra\t_\pm \bigg\}.
\]
Then the difference of the above two quantities is uniformly bounded (see \cite[{Lemma 5.6}]{MR2164775}), namely, for all distinct pairs $\t_\pm\in\L$,
\be\label{eqn:vaisaladiff}
0\le {\norm{\t_+|\t_-}}^{\sup}_x - {\norm{\t_+|\t_-}}^{\inf}_x \le 2\d.
\ee
Finally, ${\norm{\cdot|\cdot}}^{\inf}_x$ satisfies the ultrametric inequality (see \cite[{5.12}]{MR2164775}), i.e., for distinct triples $\t_1,\t_2,\t_3\in\L$,
\be\label{eqn:vaisalaultrametric}
{\norm{\t_1|\t_2}}^{\inf}_x \ge \min\LRB{{\norm{\t_1|\t_3}}^{\inf}_x, {\norm{\t_2|\t_3}}^{\inf}_x} -\d.
\ee
By (\ref{eqn:vaisaladiff}), ${\norm{\cdot|\cdot}}^{\sup}_x$ also satisfies the ultrametric inequality but with a different constant, $5\d$.

Next we compare  V\"{a}is\"{a}l\"{a}'s Gromov products with ours (see (\ref{eqn:GP})). Let $\t_\pm\in\L$ be a pair of antipodal points and let $P = P(\t_+,\t_-)$. Note that our assumption on largeness of $\lambda$ implies that there exist uniformly $\tmod$-regular sequences $(y_n^+)$ and $(y_n^-)$ on $Y\cap P$ such that $y_n^\pm \ra \t_\pm$ as $n\ra\infty$. Let $p\in P(\t_+,\t_-)$. Then, the additivity of Finsler distances $\df$ on $\tmod$-cones (cf. \cite[Lem. 5.10]{MR3811766}) yields, for large $n$, $\norm{y_n^+|y_n^-}^\tb_p = 0$. By definition,
\[
\norm{y^+_n|y^-_n}^\tb_x = \norm{y^+_n|y^-_n}^\tb_p + \frac{1}{2}\left[ \LR{\df(x,y_n^+) - \df(p,y_n^+)} + \LR{\df(y_n^-,x) - \df(y_n^-,p)} \right],
\]
and for large $n$,
\begin{equation}\label{eqn:GPcomparison_prev}
\norm{y^+_n|y^-_n}^\tb_x = \frac{1}{2}\left[ \LR{\df(x,y_n^+) - \df(p,y_n^+)} + \LR{\df(y_n^-,x) - \df(y_n^-,p)} \right].
\end{equation}
The limit,  as $n\ra\infty$, of the right side of this equation equals $\norm{\t_+|\t_-}^\tb_x$ (cf. \eqref{eq:B-limit}). 
Therefore, 
\be\label{eqn:GPcomparison}
{\norm{\t_+|\t_-}}^{\inf}_x \le \norm{\t_+,\t_-}^\tb_x \le {\norm{\t_+|\t_-}}^{\sup}_x.
\ee
Hence, by (\ref{eqn:vaisaladiff}) and (\ref{eqn:vaisalaultrametric}), $\norm{\cdot|\cdot}^\tb_x$ satisfies the ultrametric inequality with constant $5\d$, i.e., for pairwise distinct points $\t_1,\t_2,\t_3\in\L$,
\be\label{eqn:GPultrametric}
{\norm{\t_1|\t_2}}^\tb_x \ge \min\left\{{\norm{\t_1|\t_3}}^\tb_x, {\norm{\t_2|\t_3}}^\tb_x\right\} -5\d.
\ee

{Applying \cite[Prop. 5.16]{MR2164775}, we get that, for all $0<\e \le \e_0$, $\dg{\e}_x$ is 2-bilipschitz equivalent to an actual metric\footnote{This metric can be constructed as follows: Let $0<\e \le \e_0$. On $\L$, define $\mathrm{dist}^{\tb,\e}_x(\tau,\tau') = \inf \sum_{i=1}^k \dg{\e}_x(\tau_{i-1},\tau_i)$, where the infimum is taken over all finite sequences $\tau = \tau_0,\tau_1,\dots,\tau_k = \tau'$ on $\L$. See \cite[{5.13}]{MR2164775}.} on $\L$. Here the constant $\e_0$ depend only on  $\delta$. This completes the proof of the first part of the theorem.}

For the second part, 
note that the inequality (\ref{eqn:GPcomparison}) implies that $\dg{\e}_x$ induces the standard topology on $\L$ as the Gromov boundary of $(Y, \df)$ (see \cite[{5.29}]{MR2164775}). Since, as we noted earlier, this topology is the same as the subspace topology of the flag-manifold $\Ft$, the second claim of the theorem follows as well. 
\end{proof}

\begin{cor}[Conformal metric on Anosov limit set]\label{cor:gmetricano}
{Let $\G$ be a $\tmod$-Anosov subgroup, $x\in X$. Then there exists $\e_0>0$ such that the following holds: Let $0<\e\le\e_0$. Then, for all $z\in \G x$,  the premetric $\dg{\e}_x$ is bilipschitz equivalent to an actual metric on $\LT$.}

Moreover, the action $\G\acts\LT$ is conformal with respect to $\dg{\e}_z$.
\end{cor}
\begin{proof} Since by Theorem \ref{thm:equiv} Anosov subgroups satisfy the Morse property, corollary follows from  Theorem \ref{thm:dg} combined with Lemma \ref{prop:dgconf}. \end{proof}

\begin{exmp}[Product of rank-one symmetric spaces]\label{ex:1.3}
We continue with Example \ref{ex:1.2}. Let $\t = (\xi_{r_1},\dots,\xi_{r_p})$ be a simplex in the Tits building of type $\tmod = (r_1,\dots, r_p)$ and $\tb = (1/\sqrt{p}, \dots, 1/\sqrt{p})\in\tmod$. We compute the Busemann cocycle and Gromov distance associated with $\tmod$ and type $\tb$.

Let $x = (x_1,\dots,x_k),\ y= (y_1,\dots,y_k)\in X$. Then 
\[
\dhor_\t(x,y) =
\lim_{t\ra\infty}\LR{\dr(\ell(t),x) - t}
\]
where $\ell(t)$
is a geodesic ray emanating from $y$ and asymptotic to $\tb(\t)$.
A direct computation yields
\[
\dhor_\t(x,y) = \frac{1}{\sqrt p}\sum_{j=1}^p \LR{b_{\xi_{r_j}}(x_{r_j}) - b_{\xi_{r_j}}(y_{r_j})}
= \frac{1}{\sqrt p}\sum_{j=1}^p\mathcal{B}_{\xi_{r_j}}\LR{x_{r_j}, y_{r_j}}.
\]
Hence the Gromov product can be written as
\[
\norm{\t_+| \t_-}^\tb_x=\frac{1}{\sqrt p}\sum_{j=1}^p \norm{\xi_{r_j}^+ | \xi_{r_j}^-}_{x_{r_j}},\quad
\forall\t_\pm = (\xi_{r_1}^\pm,\dots,\xi_{r_p}^\pm)\in \Ft
\]
and, finally the Gromov predistance is
\be\label{eqn:dgrank1}
\dg{\frac{1}{\sqrt p}}_x(\t_+,\t_-) = \prod_{j=1}^p\dg{\frac{1}{p}}_{x_{r_j}}\LR{\xi_{r_j}^+ , \xi_{r_j}^-}.
\ee
\end{exmp}

\begin{exmp}[$X = \SL{k+1}/\SO{k+1}$]\label{ex:2.3}
In this case the computations of  Busemann functions (see \cite{MR1317280}) are explicitly known, and therefore, the Gromov distance can also be computed explicitly.
We only give a formula for the Gromov distance in the special case when $\tmod = (1,k)$ that corresponds to the partial flags \{{\em line $\subset$ hyperplane}\} of $\R^{k+1}$. 

We continue with the notations from Example \ref{ex:2.1}. The unique $\i$-invariant type is 
$$\tb = \left(\frac{1}{2\sqrt{k+1}}, 0 , -\frac{1}{2\sqrt{k+1}}\right).$$
After equipping $\R^{k+1}$ with the inner product induced by the choice of $x\in X$, the Gromov product (with respect to $x = I_{k+1}$, the identity matrix) can be written as
\[
\norm{ (l_1, h_1) \ |\  (l_2,h_2) }^\tb_x = -\frac{\sqrt{k+1}}{ 2}\log\LR{{\sin\angle(l_1,h_2)}\cdot {\sin\angle(l_2,h_1)}}
\]
where $\angle(l,h)$ denotes the angle between the line $l$ and the hyperplane $h$.\footnote{This formula can be extracted directly from the formula of the Gromov products in Beyrer's paper \cite{beyrer} by applying it to our special case. Note that the determinant of the matrix in the claim of \cite[Appendix]{beyrer} is simply the sine of the angle between the line and hyperplane here. Careful readers may notice some discrepancy between the normalizing constant in our formula and the one in Beyrer's work. This has happened due to our choice of the Riemannian metric in $X$, which was obtained from the Killing form of $\mathfrak{g}$. In Beyrer's paper, the quantity $\lambda$ (which is $\tb$ in our paper) is not a unit vector of this metric. Also, compare with \cite{MR1317280}.}

Thus, the Gromov predistance formula is
\be\label{eqn:ex2.3}
\dg{\frac{1}{\sqrt{k+1}}}_x\LR{(l_1, h_1), (l_2,h_2)} = \sqrt{\sin\angle(l_1,h_2)}\sqrt{\sin\angle(l_2,h_1)}.
\ee
\end{exmp}

\section{Shadow lemma}\label{sec:SL}

In this section we prove a generalization Sullivan's shadow lemma in higher rank. The proof we present here is inspired by that of Albuquerque's \cite[Thm. 3.3]{MR1675889} who treated the case of full flag manifold and Quint \cite{MR1935549} who treated general flag-manifolds but only in the case of regular vectors $\tb$. 

Recall the notion of {\em shadow} from Definition \ref{defn:shadows}.
We mainly consider shadows of closed balls (with respect to the Riemannian metric) of non-zero radii in $X$ from a fixed base point $x\in X$. The topology generated by these shadows is the topology of flag convergence. See Remark \ref{defn:shadow_topology}.

\medskip
The main result in this section is the following.

\begin{thm}[Shadow lemma]\label{thm:SL}
Let $\G$ be a nonelementary $\tmod$-RA subgroup, $x\in X$, and $\mu$ a $\G$-invariant $\tb$-conformal density of dimension $\beta$. There exists $r_0>0$ such that for all $r\ge r_0$ and all $\g\in\G$ satisfying $\dr(x,\g x_0)> r$,
\[
\mu_{x}(S(x: B(\g x_0,r))) \asymp \exp \left(-\beta \df(x,\g x_0)\right).
\]
\end{thm}

Here the notation $\asymp$ means that the ratio of the the two sides is bounded above and below by positive constants.

It is worth emphasizing that this version of the Shadow lemma is valid for all $\tmod$-RA subgroups of $G$ which is a much larger class than $\tmod$-Anosov subgroups. For instance, the relatively $\tmod$-Morse subgroups of $G$, a higher rank generalization of the rank-one geometrically finite groups, are also $\tmod$-RA.

Before presenting the proof, we note two consequences of this theorem.

\begin{cor}\label{cor:nonconicalAtom}
Let $\G$ be a nonelementary uniformly $\tmod$-RA subgroup.  Then any $\tb$-conformal density $\mu$ does not have conical limit points as atoms.
\end{cor}
\begin{proof}
We observe that any conical limit point $\t\in\LT$ lies in infinitely many shadows $S(x,B(\g x_0,r))$ for sufficiently large $r>0$ (depending on $\t$). If $\t$ is an atom, then (by Theorem \ref{thm:SL}) the Poincar\'e series 
\be\label{eqn:div}
g_\beta^\tb(x,x_0) = \sum_{\g\in\G} \exp\LR{-\beta \df(x,\g x_0)}
\ee
 diverges for every $\beta\ge 0$. Hence $\dF$ must be infinite. But this contradicts Proposition \ref{prop:finiteCE}.
\end{proof}

The second application of shadow lemma will be given for the following class of subgroups.

\begin{defn}[Uniformly conical]\label{def:unif_conical}
A $\tmod$-RA subgroup is called {\em uniformly conical} if for a given pair of points $x,x_0\in X$, there is a constant $r > 0$ such that for each conical limit point $\t\in \LT$, there exists a sequence $(\g_k)$ on $\G$ flag-converging to $\t$ satisfying $\dr(\g_k x_0, V(x,\st(\t))) < r$, $\forall k\in \N$. 
\end{defn}

We observe that Anosov subgroups satisfy the uniform conicality condition: 

\begin{prop}\label{lem:unifcon}
Anosov subgroups are {uniformly conical}.
\end{prop}

\begin{proof}
This follows from the fact that the orbit map $\G\ra \G x_0 \subset X$ is a Morse embedding.
Let $\t \in \LT$ be any point and $\xi\in\vb\G$ be the preimage of $\t$ under the boundary map.
Let $(\g_k)$, $\g_1 = 1_\G$ be a geodesic sequence in $\G$ asymptotic to $\xi$.
Then the sequence $(\g_k x_0)$ is a Morse quasigeodesic in $X$ that is uniformly close to $V(x, \st(\t))$ (by definition of a Morse embedding).
\end{proof}

\begin{cor}\label{cor:div}
Let $\G$ be a nonelementary uniformly conical $\tmod$-RA subgroup and $\mu$ be a $\tb$-conformal density of dimension $\beta$.
If the conical limit set $\Ltc$ is non-null, then the Poincar\'e series $g_\beta^\tb(x,x_0)$
(see (\ref{eqn:div}))
diverges.
\end{cor}

{For $\tmod$-Anosov subgroups $\G$, Theorem \ref{thm:equiv} implies that $\Ltc = \Lt(\G)$. Hence the above result applies to all Anosov subgroups with $\Ltc$ replaced by $\Lt(\G)$ in the statement.}

\begin{proof}[Proof of Corollary \ref{cor:div}]
Writing the elements of $\G$ in a sequence $(\g_n)$, define 
\[
S_N = \sum_{n\ge N} \exp(-\beta \df(x,\g_n x_0)).
\]
Convergence of the series (\ref{eqn:div}) asserts that $\lim_{N\ra \infty} S_N = 0$.
Since $\G$ is {uniformly conical}, there exists $r>0$ such that for all $N\in\mathbb{N}$, 
\[
\Ltc \subset \bigcup_{n\ge N}S(x: B(\g_n x_0,r)).
\]
Applying Theorem \ref{thm:SL}, we get
\[
\mu_x(\LT) \le \sum_{n\ge N} \mu_x\LR{S(x: B(\g_n x_0,r))} \le \const \cdot S_N
\]
and, the bound above approaches to zero as $N\ra \infty$. 
Hence we must have $\mu_x(\Ltc) = 0$.
\end{proof}

The proof of shadow lemma occupies the rest of the section.

\begin{proof}[Proof of Theorem \ref{thm:SL}]
In this proof, we equip $\Ft$ with a $G_x$-invariant Riemannian metric. 
We use the notation $L(\t)$ to denote the set of all $\t'\in\Ft$ which are not antipodal to $\t$.
The complement of $L(\t)$ in $\Ft$ is denoted by $C(\t)$.
Note that $L(\t)$ is closed and hence, compact.

Moreover, if $\t_n\ra\t_0$, then the sets $L(\t_n)$ Hausdorff-converge to $L(\t_0)$ as $n\ra\infty$.
This can be seen as follows.
First, we note that $gL(\t) = L(g\t)$ for all $g\in G$ and $\t\in \Ft$.
This is a consequence of the fact that the action $G\acts \Ft$ preserves antipodality, i.e., $\t,\t'\in\Ft$ are antipodal if and only if $g\t$ and $g\t'$ are antipodal, $g\in G$.
Now, choose a sequence $k_n\in  G_x\cong K$ such that $k_n\ra 1_G$ and $k_n \t_0 = \t_n$.
Since $\Ft$ is compact and and each $L(\tau_n)$, $n\ge 0$, is closed,
\[
 L(\tau_n) = k_n L(\tau_0) \xrightarrow{\text{Hausdorff}} L(\t_0).
\]

\begin{lem}\label{lem:SL1}
For every $\ve>0$, there exists $\delta>0$ such that,
for every $\t_0\in\Ft$ and every $\t\in B(\t_0,\d)$,
\[N_{\ve/2}(L(\t)) \subset N_{\ve}(L(\t_0)).\]
\end{lem}

\begin{proof} We equip the set 
$$
Y= \{ L(\t): \t\in\Ft\}
$$
with the Hausdorff distance $\dH$. Then, as we noted above, the function $f: \Ft \ra Y$, $\t\mapsto L(\t)$, is continuous and, hence, uniformly continuous. 
Therefore, for every $\ve>0$, there exists $\d>0$ such that $d(\t,\t_0)<\d$ implies $\dH(L(\t),L(\t_0))<\ve/2$, which then implies $L(\t) \subset N_{\epsilon/2}(L(\t_0))$.
The lemma follows from this.
\end{proof}

Let $\mathfrak{m} = \mu_x(\LT)$ denote the total mass of $\mu_x$, and $\mathfrak{l} = \sup \{ \mu_{x}(\t) \mid \t\in\LT \}$.
Since $\mu_x$ is a regular measure and $\LT$ is compact, $\mathfrak{l}$ is realized, i.e., if $\mu_x$ has an atomic part, then it has a largest atom.
Moreover, since $\G$ is nonelementary, $\supp(\mu_x)$ is not singleton.
In fact, if $\t$ is an atom, then the every point in the orbit $\G \t$ (which has infinite number of points) is an atom. In particular, $\mathfrak{l} < \mathfrak{m}$.

\begin{lem}\label{lem:SL2}
Given $\mathfrak{l} < q < \mathfrak{m}$, there exists an $\ve_0>0$ such that for all $\t\in\LT$ and all $B\in\B(\Ft)$ contained in $N_{\ve_0}(L(\t))$,
$\mu_x(B)\le q$.
\end{lem}
\begin{proof}
If this were false, then we would get a sequence $(B_n)$ of Borel sets, a sequence $(\ve_n)$ positive numbers converging to zero, and a sequence $(\t_n)$ on $\LT$ converging to a point $\t_0$ such that for every $n\in\N$,
\[
B_n \subset N_{\ve_n}(L(\t_n)), \quad \mu_x(B_n)>q.
\]

To get a contradiction, we will show that $\mu_x(\t_0)\ge q$. Let $U$ be an open neighborhood of $L(\t_0)$. As $L(\t_0)$ is compact, there exists $\ve>0$ such that $N_\ve(L(\t_0)) \subset U$.
Let $\d>0$ be a number that corresponds to this $\ve$ as in Lemma \ref{lem:SL1}.
Choose $n$ so large such that $\t_n\in B(\t_0,\d)$ and $\ve_n \le \ve/2$. By Lemma \ref{lem:SL1}, we get $N_{\ve_n}(L(\t_n)) \subset N_{\ve}(L(\t_0))$ and, consequently, $B_n \subset U$.
This shows that every open set $U$ containing $L(\t_0)$ has mass $\mu_x(U) > q$.
In particular, $\mu_x(L(\t_0)) \ge q$.

Finally, as a consequence of the fact that $\G$ is $\tmod$-antipodal, we have that $\LT \cap L(\t_0) = \tau_0$.
Since the support of $\mu_x$ is contained in $\LT$ (see Definition \ref{def:conformal_density}),
\[
 \mu_x(L(\tau_0)) = \mu_x(\LT \cap L(\t_0)) = \mu_x(\tau_0).
\]
The last sentence in the previous paragraph implies that $\mu_x(\tau_0) \ge q$.
Hence, $\mathfrak{l} \ge q$.
This is a contradiction.
\end{proof}

\begin{lem}\label{lem:SL3}
Given $\ve>0$ there exists $r_1>0$ such that for all $r\ge r_1$, the complement of $S(x:B(x_0,r))$ in $\Ft$ is contained in $N_\ve(L(\t))$, for some $\t\in S(x_0: \{x\})$.
\end{lem}
\begin{proof}
For $r>0$ and $\t_0\in\Ft$, $\t'\in C(\t_0)$, consider
\[
U(\t_0,x_0,r) = \{\t'\in\Ft \mid P(\t_0,\t') \cap B(x_0,r) \ne \emptyset\}.
\]
This is an analogue of shadows (\ref{defn:shadow}) as viewed from the infinity.
It is easy to verify that 
\[\bigcup_{r\ge 0} U(\t_0,x_0,r) = C(\t_0). \]
Moreover, for $g\in G$, these shadows from infinity transform as $gU(\t_0,x_0,r) = U(g\t_0,gx_0,r)$. 

If $k\in K=G_{x_0}$, the stabilizer of $x_0$, then $kU(\t_0,x_0,r) = U(k\t_0,x_0,r)$. 
Since $K$ is compact, there exists $M\ge 1$ such that the action $k\acts \Ft$ is $M$-Lipschitz for all $k\in K$.
Let $r_1>0$ be such that $U(\t_0,x_0,r_1/2)^c \subset N_{\ve/M}(L(\t_0))$. Here and below, for $A\subset \Ft$, $A^c= \Ft - A$. 
Then, for any $\t\in\Ft$, 
\begin{equation}\label{eqn:U}
U(\t,x_0,r/2)^c \subset N_{\ve}(L(\t)),\quad \forall r\ge r_1.
\end{equation}

For $x\in X$, let $\t\in\Ft$ be a simplex such that $x\in V(x_0,\st(\t))$. Then there exists a parameterized geodesic ray $x_t$ 
starting from $x_0$, passing through $x$ and asymptotic to some $\xi\in \st(\t)$. 

\begin{claim}
For all $r>0$, $S(x:B(x_0,2r)) \supset U(\t,x_0,r)$.
\end{claim}

\begin{proof}[Proof of claim] Pick $\t'\in U(\t,x_0,r)$ and let   $\bar{x}_0\in P(\t,\t')$ denote the nearest point projection
 of $x_0$. In addition to the ray $x_t$, we define another parameterized geodesic ray $\bar{x}_t$, starting at $\bar{x}_0$ 
and asymptotic to $\xi$. Due to the convexity of the Riemannian distance function on $X$, 
the distance $\dr(x_t,\bar{x}_t)$ monotonically decreases with $t$. 
Moreover, the cones $V(\bar{x}_t,\st(\t'))$ are nested as $t$ decreases. Then, 
\begin{align*}
\dr(x_0, V(x_t,\st(\t'))) &\le \dr(x_0, V(\bar{x}_t,\st(\t'))) + \dr(x_t,\bar{x}_t)\\
&\le \dr(x_0, V(\bar{x}_0,\st(\t'))) + r \le \dr(x_0,\bar{x}_0) + r \le 2r.
\end{align*}
Therefore, $\t' \in S(x:B(x_0,2r))$.
\end{proof}

Using (\ref{eqn:U}) it follows from the above claim that whenever $r\ge r_1$, the complement of the shadow $S(x: B(x_0,r))$ is contained in $N_{\ve}(L(\t))$ for some $\t$ satisfying $x\in V(x_0,\st(\t))$.
\end{proof}

\begin{lem}\label{lem:SL4}
For all $r>0$ and all $\t \in S(x: B(x_0,r))$,
\[
|\df(x,x_0) - \dhor_{\t}(x,x_0)| \le 2r.
\]
\end{lem}
\begin{proof}
We recall that $\df$ can alternatively be defined as 
\[
\df(y,z) = \max_{\t\in\Ft}\dhor_{\t}(y,z),
\]
where the maximum above occurs at any point in $S(y: \{z\})$ (see \cite[Sec. 5.1.2]{MR3811766}). Fix some $\t_0\in S(x,\{x_0\})$. Then, for any $\t \in S(x: B(x_0,r))$,
\begin{align*}
|\df(x,x_0) - \dhor_{\t}(x,x_0)|
&=|{b}^\tb_{\t_0}(x_0) - {b}^\tb_{\t}(x_0)|\\
&= |{b}^\tb_{\t_0}(x_0) - b^\tb_{\t_0}(k^{-1}x_0)|\\
&\le \dr(x_0,k^{-1}x_0)
=\dr(x_0,kx_0),
\end{align*}
where $k\in K$, stabilizer of $x$, any element satisfying $\t = k\t_0$. In the above we chose the normalizations of the Busemann functions at $x$.

Let $y\in V(x, \st(\tau))\cap B(x_0,r)$. Then 
$y\in V(x, \sigma)$ for some  chamber $\sigma$  in $\st(\tau)$. We identify $V(x,\sigma)$ with the model Weyl chamber $\Delta$. Let $k_1\in K$  such that 
$k_1 x_0\in V(x,\sigma)$. Then, $k_1\tau_0 = \tau$, and $k_1 x_0=d_\Delta(x,x_0)$ via the identification above. 
Moreover, since the map 
$$
X\to \Delta, \quad z\mapsto d_\Delta(x, z)
$$ 
is 1-Lipschitz (by the triangle inequality for $\Delta$-distances (\ref{eqn:ti})) and $d_\Delta(x, y)=y$, we obtain,
$$
\dr(y, k_1 x_0)\le \dr(y, x_0)\le r
$$
and, in particular, $\dr(x_0, k_1 x_0)\le 2r$. 
\end{proof}

Using the above lemmata, we now complete the proof of Theorem \ref{thm:SL}. 
We first fix some auxiliary quantities. 
Let $q\in(\mathfrak{l},\mathfrak{m})$ and $\ve_0$ be corresponding constant as given in Lemma \ref{lem:SL2}. Let $\d$ be a constant given by Lemma \ref{lem:SL1} which corresponds to $\ve=\ve_0$.
By $\L$ we denote  the  $\d$-neighborhood of $\Lt$
and let 
\[
V = \bigcup_{\t\in\L} V(x,\st(\t))\ \subset X.
\]
Since $\G$ is discrete, the elements of $\G$ which send $x_0$ outside $V$ form a finite set $\Phi$.
Let
\[ r_0 = \max\{r_1, \dr(x,\g x_0) \mid \g\in\Phi\}\]
where $r_1$ is a constant that corresponds to $\ve_0/2$ as in Lemma \ref{lem:SL3}.

For every $\g\in\G$ satisfying $\dr(x,\g x_0)> r \ge r_0$, we assign an element 
$\t_\g \in S(x: \{\g x_0\}) \cap \L$ (the intersection is nonempty by above).
Using Lemma \ref{lem:SL1}, for every such $\t_\g$ there exists $\t_0\in\Lt$ so that
\[
N_{\ve_0/2}(L(\t_\g)) \subset N_{\ve_0}(L(\t_0)).
\]
By Lemmata \ref{lem:SL2} and \ref{lem:SL3}, $\mu_x(S(\g^{-1} x: B(x_0,r))) \ge \mathfrak{m} - q$ and by properties of conformal measures,
\begin{align*}
\mu_x(S(x:B(\g x_0, r))) &= \mu_{\g^{-1}x}(S(\g^{-1} x:B(x_0, r)))\\
 &= \int_{S(\g^{-1} x:B(x_0, r))} \exp\LR{-\beta \dhor_\t(\g^{-1} x,x)}d\mu_x \\
 &\asymp \exp\LR{-\beta\df(x,\g x_0)}
\end{align*}
where in the last step we have additionally used Lemma \ref{lem:SL4}. This completes the proof.
\end{proof}

\section{Dimension of a $\tb$-conformal density}\label{sec:dim}

In this section, we establish a lower bound for the dimension of a $\tb$-conformal density. 
For Anosov subgroups, we prove that the dimension equals the $\tb$-critical exponent (see Corollary \ref{cor:dimAno}).

\begin{thm}\label{thm:lowerbound}
Suppose that $\G$ is a nonelementary $\tmod$-RA subgroup. 
Let $\mu$ be a $\G$-invariant $\tb$-conformal density of dimension $\beta$. 
Then $\beta$ has the following lower bound:
\be\label{eqn:lowerbound}
\beta \ge \dF - \dF^\con.
\ee
\end{thm}

The proof of this theorem is given at the end of this section. 
The number $\dF^\con$ above quantifies the maximal exponential growth rate of the orbit $\G x_0$ in a conical direction. 
The precise definition is given below.

\begin{defn}[Critical exponent in conical directions]\label{defn:CEcon}
Suppose that $\G$ is a $\tmod$-regular subgroup. For $\t\in \LT$, define 
\[
N_\tb^\mathrm{con} (r,c,x,x_0,\t) = \card\{\g\in \G \mid \df(x,\g x_0) < r,\ \dr(\g x_0,V(x,\st(\t))) < c\}
\]
and
\begin{equation}\label{eqn:conical_crit_expo}
\dF^\con(\G) = \sup_{\t\in\LT}\LR{\lim_{c\ra\infty}\LR{\limsup_{r\ra \infty}\frac{\log N_\tb^\mathrm{con} (r,c,x,x_0,\t)}{r}}}.
\end{equation}
\end{defn}

Note that it is sufficient to take the supremum in the definition of $\dF^\con(\G)$ over the conical limit set $\Ltc$. 
For rank-one symmetric spaces, and, more generally, for  $\smod$-regular subgroups, this number is zero. 
This can be seen as follows.
Let $\G$ be a $\smod$-regular discrete subgroup, and $\sigma\in\mathrm{Flag}(\smod)$ be a limit point.
Consider a Riemannian tubular neighborhood $T$ of $V(x,\st(\sigma)) = V(x,\sigma)$. Then the (Riemannian) volume of $T_r := \{ y\in T \mid \df(x,y) < r\}$ grows at most polynomially (of degree equal to the rank of $X$)  with respect to $r$.
Hence the number $N_\tb^\con (r,c,x,x_0,\sigma)$ also grows at most polynomially with $r$. Therefore, the limit in the innermost bracket of (\ref{eqn:conical_crit_expo}) is zero.

Below we see that for $\tmod$-Anosov subgroups also, $\dF^\con(\G) = 0$. 
It should be noted that, however, for general  discrete subgroups, $\dF^\con$ could be $\infty$. 

\begin{prop}\label{prop:CEconAno}
Suppose that $\G$ is a nonelementary $\tmod$-Anosov subgroup. 
Then the function $N(r) = N_\tb^\con (r,c,x,x_0,\t)$ grows linearly with $r$. 
In particular, $\dF^\con(\G) = 0$.
\end{prop}

\begin{proof}
Without loss of generality, we can assume that $x= x_0$.\footnote{Note that the number $\dF^\con(\G)$ does not depend on $x$ and $x_0$ as we have seen in the case of $\dF$ in Sec. \ref{sec:CE}.}

\begin{lem}\label{lem:unifcloseAno}
Fix $c>0$. For any $\t\in\LT$, the set 
\[
\{\g x_0 \mid \g\in\G,\ \dr(\g x_0,V(x,\st(\t))) < c\}
\]
is within a uniformly bounded distance from a uniform $\tmod$-Morse quasiray $\alpha$ emanating from $x_0$ and asymptotic to $\t$.
\end{lem}

\begin{proof}
Pick an arbitrary point $\t\in\LT$.
Denote the preimage of $\t$ in $\vb \G$ under the  boundary homeomorphism $\vb\G\to \LT$ by $\zeta$. 
Using discreteness of $\G$, we arrange the elements of the subset $\{\g\in \G \mid \dr(\g x_0,V(x,\st(\t))) < c\}$ in a sequence $(\g_n)$. 
The sequence $x_n= \g_n x_0$ converges conically to $\t$. 
Let $\alpha:\mathbb{Z}_{\ge 0}\ra X$ be the image (under the orbit map $\G\to \G x$)  of a parametrized geodesic ray $\mathbb{Z}_{\ge 0} \ra \G$ starting at $1_\G$ and asymptotic to $\zeta$. 
Then $\alpha$ is a uniform $\tmod$-Morse quasiray starting at $x_0$ and  asymptotic to $\t$. 
Hence $\alpha$ is uniformly close to $V(x_0,\st(\t))$.
Since both sequence $(x_n)$ and $(\alpha(n))$ are uniformly close to $V(x_0,\st(\t))$, it is enough\footnote{For instance, we can consider the nearest point projections of the sequences $\alpha(n)$ and $x_n$ to $ V(x_0,\st(\t))$. The new sequences would be uniformly close to the old ones.  Note that the projection sequence of $\alpha(n)$ is also a uniform $\tmod$-Morse quasiray.} to understand the simpler case when 
 $\alpha(n),\ x_n\in V(x_0,\st(\t))$, for all $n\in \N$.
 
We claim that  the sequence $(x_n)$ is uniformly close to $\alpha$.
Otherwise, after extraction, $(x_n)$ would diverge away from $\alpha$. Since $\alpha$ is a Morse quasiray,
$\alpha$ eventually enters each cone $V(x_n,\st(\t))$, but further and further away from the tip $x_n$ as $n$ grows.
Since the separation between two successive points on $\alpha$ (being a quasigeodesic) is uniformly bounded, we could find arbitrarily large $m$'s such that $\alpha(m)$ is uniformly close to the boundary of a cone $V(x_n,\st(\t))$ and is arbitrarily far away from its tip $x_n$. 
But this would contradict the uniform $\tmod$-regularity of the group $\G$.
\end{proof}

We continue with the notations from the proof of the lemma.
Since any $\tmod$-Anosov subgroup $\G< G$ is uniformly $\tmod$-regular (cf. Theorem \ref{thm:equiv}), we may work with the Riemannian metric $\dr$ in place of the $\df$.
Moreover, we may assume that the sequence $(x_n)$ is sufficiently spaced.
Let $\bar x_n$ denote the nearest-point projection of $x_n$ to the image of $\alpha$.
The above lemma implies that $d(x_n,\bar x_n)$ is uniformly bounded.
Since $x_n$'s are sufficiently spaced, $\bar x_n$'s are also sufficiently spaced which guarantees that $\dr(\bar x_n, x_0) \ge \mathrm{const}\cdot n$, for all large $n$, which in turn implies that $\dr(x_n, x_0) \ge \const \cdot n$.
The proposition follows from this.
\end{proof}

As a corollary of the above results, we obtain that any $\G$-invariant $\tb$-conformal density must have dimension $\dF$ when $\G$ is $\tmod$-Anosov. The Patterson--Sullivan densities constructed in Section \ref{sec:CD} also had this dimension.

\begin{cor}\label{cor:dimAno}
Suppose that $\G$ is a nonelementary $\tmod$-Anosov subgroup. Let $\mu$ be a $\G$-invariant $\tb$-conformal density of dimension $\beta$. Then $\beta = \dF$.
\end{cor}
\begin{proof}
Recall that $\G$-invariant $\tb$-conformal densities are, by definition (see Definition \ref{def:conformal_density}), supported in the limit set $\Lt(\G)$.
Since, for $\tmod$-Anosov subgroups\footnote{For $\tmod$-Anosov subgroups $\G$, $\Lt(\G) = \Ltc$ is a consequence of the fact that $\tmod$-Anosov subgroups are also $\tmod$-RCA, see Theorem \ref{thm:equiv}.} \[\Lt(\G) = \Ltc,\] 
by Corollary \ref{cor:div}, we know that the Poincar\'e series $g_\beta^\tb(x,x_0)$ diverges and, consequently, $\beta \le \dF$. The reverse inequality is obtained by a combination of Theorem \ref{thm:lowerbound} and Proposition \ref{prop:CEconAno}.
\end{proof}

To close this section, we prove Theorem \ref{thm:lowerbound}.

\begin{proof}[Proof of Theorem \ref{thm:lowerbound}]

We fix some $r\ge r_0$ where $r_0$ is given by Theorem \ref{thm:SL}.

We first assume that the stabilizer of $x_0$ in $\G$ is trivial in which case the function $N_{\tb}(R,x,x_0)$ counts the number of orbit points (in $\G x_0$) within the $R$-ball in $(X,\df)$ centered at  $x$.

We place a Riemannian ball of radius $r$ at each point in the orbit.
In this proof, we reserve the word {\em ball} to specify these balls.
Let
\[
c = \min_{1_\G\ne \g\in \G} \LRB{\dr(x_0,\g x_0)}.
\]
There exists a number $N\in\mb{N}$ that depends only on $r$, $c$, and $X$ such that any ball intersect at most $N$ other balls (including itself).
Note that the shadows in $\Ft$ (from $x$) of two distinct balls are disjoint unless they intersect some common $\tmod$-cone with tip at $x$.
Also note that, at large distances from $x$, the balls do not intersect the boundaries of the $\tmod$-cones because of the $\tmod$-regularity of the orbit.

Let $n_R$ denote the maximal number of balls in $B_\tb(x,R)$ that intersect a particular $\tmod$-cone $V(x,\st(\t))$.
It follows from the definition of $\dF^\con(\G)$ that 
\be\label{eqn:lowerbound1}
\limsup_{R\ra \infty}\frac{\log n_R}{R} \le \dF^\con(\G).
\ee
On the other hand, for each $\t\in \LT$, the maximal number of balls in $B_\tb(x,R)$ whose shadows intersect $\t$ is $n_R$. Therefore,
\be\label{eqn:lowerbound2}
\frac{N_\tb(R,x,x_0)}{N\cdot n_R}s(R) \le \mathfrak{m} = \text{total mass of }\mu_x,
\ee
where $s(R)$ is any lower bound for the measures of the shadows of balls in $B_\tb(x,R)$.
We note that the shadow lemma (Theorem \ref{thm:SL}) produces such a positive lower bound\footnote{We may need to disregard a finite number of balls from the picture.}, namely, we may take 
$
s(R) = \const\cdot e^{-\beta R}.
$
Then (\ref{eqn:lowerbound2}) yields
\[
N_\tb(R,x,x_0) \le  \frac{\mathfrak{m}N\cdot n_R}{\mathrm{const}}e^{\beta R}.
\] 
Together with (\ref{eqn:lowerbound1}), the above results in (\ref{eqn:lowerbound}).

{In the general case when $\G_{x_0}$ is non-trivial, $\G_{x_0} = K \cap \G$ is still at most finite. In this case, $N_{\tb}(R,x,x_0)$ will be a constant multiple of the number of $\G$-orbit points of $x_0$ within the $R$-ball in $(X,\df)$ centered at  $x$. So, we only need to change the constant term in the above inequality.}
\end{proof}

\section{Uniqueness of $\tb$-conformal density}\label{sec:ergodic}
Recall that an action of a group $H$ on a measure space $(S,\sigma)$ is said to be {\em ergodic} if each $H$-invariant measurable set $B \subset S$ is either null or co-null.
In \cite{MR556586}, Sullivan proved that for a discrete group $\G$ of M\"obius transformations of the Poincare ball $\mathbb{B}^3$, a $\G$-invariant $\tb$-conformal density $\mu$ of non-zero dimension is unique (here and henceforth, by ``unique'' we mean unique up-to a constant factor) in the class of all $\tb$-conformal densities of same dimension if and only if the action $\G$ on the limit set $\L(\G)$ is ergodic with respect to any $\mu_x\in\mu$. See also \cite[{Thm. 4.2.1}]{nicholls1989ergodic}.
Generalizing this statement in our setting, we obtain the following result. The proof is essentially same of Sullivan's theorem, hence we omit the details. 

\begin{thm}\label{thm:ergodic}
Suppose that $\G$ is a nonelementary $\tmod$-RA subgroup. A $\G$-invariant $\tb$-conformal density $\mu$ of dimension $\beta> 0$ is unique in the class of all $\G$-invariant $\tb$-conformal densities of dimension $\beta$ if and only if the action $\G\acts\LT$ is ergodic with respect to any $\mu_x\in\mu$.
\end{thm}

It is then natural to ask 

\begin{ques}
For which $\tmod$-regular subgroups $\G$, the action $\G \acts \LT$ is ergodic with respect to a conformal measure? 
\end{ques} 

 In this section we prove that the Anosov property is a sufficient condition:

\begin{thm}[Anosov implies ergodic]\label{thm:ergoAno}
Suppose that $\G$ is a nonelementary $\tmod$-Anosov subgroup and $\mu$ be a $\G$-invariant $\tb$-conformal density. 
Then the action $\G\acts\Lt(\G)$ is ergodic with respect to any $\mu_x\in\mu$.
\end{thm}

As a corollary, we obtain that when $\G$ is $\tmod$-Anosov, then, up to a constant factor, there is exactly one $\G$-invariant $\tb$-conformal density, namely, the Patterson--Sullivan density.

\begin{cor}[Existence and uniqueness of $\tb$-conformal density]\label{cor:uniqueAnosovDensity}
Suppose that $\G$ is a nonelementary $\tmod$-Anosov subgroup. Then, up to a constant factor, there exists a unique $\G$-invariant $\tb$-conformal density $\mu$, namely, the Patterson--Sullivan density of type $\tb$.
\end{cor}
\begin{proof}
First of all, by Proposition \ref{cor:positiveCE}, any such density must have a positive dimension. Secondly, by Corollary \ref{cor:dimAno} this dimension equals to the critical exponent $\dF$. Then the uniqueness follows from the combination of Theorems \ref{thm:ergodic} and \ref{thm:ergoAno}.
\end{proof}

Now we return to the proof of Theorem \ref{thm:ergoAno}. 

\begin{proof}[Proof of Theorem \ref{thm:ergoAno}]
Let $\mu$ be a $\G$-invariant $\tb$-conformal density. Note that the dimension $\beta$ of $\mu$ must be positive (by Proposition \ref{cor:positiveCE} and Corollary \ref{cor:dimAno}).

Let $B$ be a $\G$-invariant Borel subset of $\Lt(\G)$. We need to prove that if $B$ is not a null set, then it is co-null.
From now on, we assume that $B$ is not a null set, i.e., $\mu_x(B)>0$.

We need the following lemmata.

\begin{lem}\label{lem:ergoAno1}
There exists $r_1>0$ such that for every $r\ge r_1$ and every $\g\in\G$, the shadow $S(x, B(\g x_0,r))$ intersects $\Lt(\G)$.
\end{lem}
\begin{proof}
The proof simply follows from the Morse property of the Anosov subgroup $\G$.
\end{proof}

We assume that the $r_1$ in the lemma also satisfies the ``uniform conicality'' property for $\G$ (cf. Proposition \ref{lem:unifcon}).  

\begin{lem}\label{lem:lebesguedensity}
Let $r\ge \max\{r_0,r_1\}$ where $r_0$ is as in Theorem \ref{thm:SL}.
For $\mu_x$-a.e. $\tau\in B$ and every sequence $(\g_n)$ on $\G$, $\g_n\ra\t$, satisfying $\tau \in S_n := S(x: B(\g_nx_0,r))$, we have
\be\label{eqn:ergodicAction2}
\lim_{n\ra \infty} \frac{\mu_{x}(S_n\cap B)}{\mu_{x}(S_n)} = 1.
\ee
\end{lem}

Assuming this lemma for a moment, we complete the proof of the theorem.
The proof of this lemma is given at the end of this section. Note that, Lemma \ref{lem:ergoAno1} is used to ensure that the ratios in the above lemma are not degenerate.

Let $\tau\in B$ be a {\em density point}, i.e., $\tau$ satisfies (\ref{eqn:ergodicAction2}).
Such point exist by Lemma \ref{lem:lebesguedensity} because $B$ has positive mass.
Note that, $\G$-invariance of $B$ and $\mu$ implies that
\begin{align*}
\frac{\mu_{x}(S(\g_n^{-1} x: B(x_0,r))\cap B)}{\mu_{x}(S(\g_n^{-1} x: B(x_0,r)))}
&=\frac{\mu_{\g_n x}(S_n\cap B)}{\mu_{\g_n x}(S_n)} \\
&= 1- \frac{\mu_{\g_nx}(S_n - B)}{\mu_{\g_nx}(S_n)}\\
&= 1- \frac{\int_{S_n - B} \exp\LR{-\beta \dhor_\t(\g_n x,x)}d\mu_{x}}{\int_{S_n}\exp\LR{-\beta \dhor_\t(\g_n x,x)}d\mu_{x}}\\
&\ge 1 - \const\cdot
\frac{\mu_{x}(S_n - B)}{\mu_{x}(S_n)},
\end{align*}
where the inequality follows by Lemma \ref{lem:SL4}. Together with (\ref{eqn:ergodicAction2}),
we get
\be\label{eqn:ergodicAction3}
\lim_{n\ra\infty}\frac{\mu_{x}(S(\g_n^{-1} x: B(x_0,r)) \cap B)}{\mu_{x}(S(\g_n^{-1} x: B(x_0,r)))} = 1.
\ee

Note that by Corollary \ref{cor:nonconicalAtom}, $\mu$ is atom-free. Therefore, for every $\varepsilon>0$ there exists $r>r_1$ such that
\[
\mu_{x}(S(\g_n^{-1} x: B(x_0,r))) \ge \mathfrak{m}-\varepsilon,
\]
for all large $n$, where $\mathfrak{m}$ denotes the total mass of $\mu_x$. The above follows from the combination of Lemmata \ref{lem:SL2} and \ref{lem:SL3}. 
Therefore, by (\ref{eqn:ergodicAction3}), 
\[
\mu_{x}(B) \ge \lim_{n\ra\infty}\mu_{x}(S(\g_n^{-1} x: B(x_0,r)) \cap B) \ge \mf{m}-\varepsilon,
\]
which holds for every $\varepsilon >0$.
Hence $\mu_{x}(B) = \mathfrak{m}$.
This completes the proof of the theorem.
\end{proof}

Now we prove Lemma  \ref{lem:lebesguedensity}.
The lemma would have followed from a generalization of the Lebesgue density theorem (cf. \cite[Subsec. {2.9.11, 2.9.12}]{MR0257325}) if we 
knew that $\mu_x$ is, e.g., a {\em doubling measure}. Since this property is unclear, we adopt a 
more direct approach. The idea of the proof follows \cite[Subsec. 1E]{MR2057305} (see also \cite[Sec. 3]{Lin06}).

\begin{proof}[Proof of Lemma \ref{lem:lebesguedensity}]
The proof requires a version of the Lebesgue differentiation theorem.

\begin{sublem}\label{sublemma}
For every bounded measurable function $\Phi: \Ft\ra \R_{\ge 0}$,
\[
\Phi(\tau) =
\lim_{n\ra\infty} \frac{1}{\mu_x(S(x: B(\g_nx_0,r)))}\int_{S(x: B(\g_nx_0,r))} \Phi d\mu_x.
\]
for $\mu_x$-a.e. $\tau\in\Lt$ and all $\g_n\in\G$  satisfying $\tau \in  S(x: B(\g_nx_0,r))$.
\end{sublem}

\begin{proof}
For every bounded measurable function $\Psi: \Ft\ra \R_{\ge 0}$, define a function $\Psi^*$ on $\Ft$ which is zero outside $\LT$ and on $\LT$ it is defined by
\be\label{eqn:maxfn}
\Psi^*(\tau) = \limsup_{N\ra\infty}  \frac{1}{\mu_x(S(x: B(\g x_0,r)))}\int_{S(x: B(\g x_0,r))} \Psi d\mu_x,
\ee
Here and in the following the  limit superior is taken over all $\g\in\G$ satisfying $\dr(x,\g x_0) \ge N$ and $\tau \in  S(x: B(\g x_0,r))$.

Let $\Phi_k$ be a sequence of continuous functions converging to $\Phi$ $\mu_x$-almost surely such that
\[
\int_{\Ft} |\Phi_k - \Phi| d\mu_x <\frac{1}{k}, \quad \forall k\in\N.
\]
Then for every $\tau\in\Ft$ and $\g\in \G$, 
we have
\begin{multline}\label{eqn:ergodicAction4}
\limsup_{N\ra\infty}\bigg|   \frac{1}{\mu_x(S(x: B(\g x_0,r)))} \int_{S(x: B(\g x_0,r))} \Phi d\mu_x - \Phi(\tau) \bigg| \\ 
\shoveleft\qquad\qquad\qquad\qquad\qquad
\le |\Phi - \Phi_k|^*(\tau) + |\Phi_k(\tau) - \Phi(\tau)| \\
+\limsup_{N\ra\infty}\left|   \frac{1}{\mu_x(S(x: B(\g x_0,r)))}\int_{S(x: B(\g x_0,r))} \Phi_k d\mu_x - \Phi_k(\tau) \right| .
\end{multline}
Since $\Phi_n$ are continuous, the last quantity in the right side of the above vanishes.
Moreover, the limit of $|\Phi_k(\tau) - \Phi(\tau)|$ as $k\ra\infty$ vanishes at $\mu_x$-a.e. $\tau\in\Ft$.
Therefore, we only need to control the first term of the right side of (\ref{eqn:ergodicAction4}): We show that, for all bounded nonnegative measurable functions $\Psi$ on $\Ft$ and all $\varepsilon>0$,
\be\label{eqn:ergodicAction5}
\mu_x\LR{\LRB{\Psi^* >\varepsilon}} \le \frac{\const}{\varepsilon}\int_{\Ft}\Psi d\mu_x
\ee
where the constant does not depend on $\varepsilon$ or $\Psi$.
The sublemma follows from this as follows: 
Setting $\Psi = |\Phi - \Phi_k|$ and taking limit as $k\ra \infty$ in (\ref{eqn:ergodicAction5}), we see that $|\Phi - \Phi_k|^*$ $\mu_x$-a.s. converges to zero. Hence left-hand side of (\ref{eqn:ergodicAction4}) also converges to zero for $\mu_x$-a.e. $\tau\in\Lt$.

Now we verify (\ref{eqn:ergodicAction5}). 
Let $\varepsilon>0$ be arbitrary.
For $d\ge 0$, let $\G_d$ be the set of all elements $\g\in\G$ such that $\df(x,\g x_0) \ge d$ and
\be\label{eqn:ergoano12}
\int_{S(x:B(\g x_0,r))} \Psi d\mu_{x} \ge \frac{\varepsilon}{2}\mu_x(S(x:B(\g x_0,r))).
\ee

\begin{customclaim}{1}
 The union of all shadows $S(x:B(\g x_0,r))$ over $\g\in\G_d$ covers $\{\Psi^*>\varepsilon\}$.
\end{customclaim}

\begin{proof}[Proof of claim]
The proof is straightforward.
\end{proof}

We recursively construct a sequence of subsets, $(\G_{d,N})$, of $\G_d$ in the following way: Let $\G_{d,1} = \LRB{ \g\in\G_d \mid 0\le \df(x,\g x_0) < 1}$, and, for $N\ge 2$, define
\[
\G_{d,N} = \LRB{ \g\in\G_d ~ \left| ~ 
\begin{array}{l}
N-1 \le \df(x,\g x_0) < N \text{ and } S(x: B(\g x_0,r))\cap \\
 S(x: B(\phi x_0,r)) = \emptyset, \forall\phi\in \G_{d,1}\cup\dots\cup \G_{d,N-1}
\end{array}
\right.}.
\]
Set $\G^*_{d} = \bigcup_{N\ge 1} \G_{d,N}$.

\begin{customclaim}{2}
There exists a constant $R\ge r$ such that, for every $d\ge 0$, 
\[
{\LRB{\Psi^* >\varepsilon}} \subset \bigcup_{\phi\in\G^*_{d}} S(x: B(\phi x_0,R)).
\]
\end{customclaim}

\begin{proof}[Proof of claim]
It is enough to prove the claim for very large $d$. In fact, we assume that $d$ is so large such that $\rg{x(\g x_0)}$ is uniformly $\tmod$-regular for all $\g\in\G_d$.

Let $\tau\in {\LRB{\Psi^* >\varepsilon}}$ be arbitrary. Then there exists $\g\in\G_d$ such that $\tau\in S(x: B(\g x_0,r))$.
Assume that $\g\not\in\G^*_d$.
By construction of $\G_{d}^*$, there exists $\phi\in\G^*_d$ such that $S(x: B(\g x_0,r))\cap S(x: B(\phi x_0,r)) \ne \emptyset$ and $\df(x,\phi x_0) < \df(x,\g x_0)$.

By Lemma \ref{lem:unifcloseAno}, both $\g x_0$ and $\phi x_0$ stay uniformly close to a $\tmod$-uniform Morse quasigeodesic $\alpha$ with one endpoint at $x$. Since $\df(x,\phi x_0) < \df(x,\g x_0)$, we may assume that the other endpoint of $\alpha$ is uniformly close to $\g x_0$.
It follows that $\phi x_0$ is uniformly close to the diamond $\diamondsuit_{\Theta}(x,\g x_0)$, since $\alpha$ is, for some $\Theta\subset\tmod$.
Pick $y\in B(\g x_0,r)\cap V(x,\st(\tau))$.
Then, by uniform continuity of diamonds (cf. \cite[Thm. 3.7]{DKL18}), for some $\Theta'$ bigger than $\Theta$, $\diamondsuit_{\Theta}(x,\g x_0)$ is contained in a uniform neighborhood of $\diamondsuit_{\Theta'}(x,y)$.
Therefore, $\phi x_0$ is uniformly close to $\diamondsuit_{\Theta'}(x,y)$ and, in particular, to $V(x,\st(\tau))$.
We may choose $R$ to be this upper bound.
\end{proof}

In particular, we get
\be\label{eqn:ergoano11}
\mu_x\LR{{\LRB{\Psi^* >\varepsilon}}} \le \sum_{\phi\in\G^*_{R}} \mu_x\LR{S(x: B(\phi x_0,R))}.
\ee

\begin{customclaim}{3}\label{clm:three}
If $S(x: B(\g x_0,r))\cap S(x: B(\phi x_0,r)) \ne\emptyset$, for $\g,\phi\in\G^*_{d}$, then $\df(\g x_0, \phi x_0)$ is uniformly bounded. 
\end{customclaim}
\begin{proof}[Proof of claim]
This follows from the Gromov hyperbolicity of $(\G x_0,\df)$ (see Corollary \ref{cor:hypano}) and the fact that both $\g x_0$ and $\phi x_0$ lie in an annulus $\{x'\in X \mid N-1 \le \df(x,x')< N\}$ in the following way:
Let $\tau\in S(x: B(\g x_0,r))\cap S(x: B(\phi x_0,r))$. Let $z\in V(x,\st(\tau))$ be a point uniformly close to $\G x_0$.
By $\delta$-hyperbolicity,
\be\label{eqn:ergoano13}
\norm{\g x_0 | \phi x_0}_x + \delta \ge \min\LRB{ \norm{\g x_0 | z}_x , \norm{\phi x_0|z}_x }.
\ee
Expanding the left side, we get
\begin{align}\label{eqn:ergoano10}
\begin{split}
\norm{\g x_0 | \phi x_0}_x + \delta &= \frac{1}{2}\LR{\df(x,\g x_0) + \df(\phi x_0,x) - \df(\g x_0,\phi x_0) }+ \delta\\
&\le  \LR{\df(\phi x_0,x) - \frac{1}{2}\df(\g x_0,\phi x_0) } +\delta+\frac{1}{2},
\end{split}
\end{align}
and expanding the right side, we get
\begin{align*}
\min\LRB{ \norm{\g x_0 | z}_x , \norm{\phi x_0|z}_x } &= 
\min\LRB{
\begin{array}{l}
\frac{1}{2}\LR{\df(x,\g x_0) + \df(z,x) - \df(\g x_0,z)},\\
\frac{1}{2}\LR{\df(x,\phi x_0) + \df(z,x) - \df(\phi x_0,z)}
\end{array}}.
\end{align*}
Taking $z\ra\tau$  in the right side of the last one and using (\ref{eq:B-limit}), we get
\[
\min\LRB{\frac{1}{2}\LR{\df(\g x_0,x) + \dhor_\tau(x,\g x_0)}, 
\frac{1}{2}\LR{\df(\phi x_0,x) + \dhor_\tau(x,\phi x_0)} }.
\]
which, by Lemma \ref{lem:SL4}, is at least
\[
\min\LRB{\df(\g x_0,x), 
\df(\phi x_0,x)} - r \ge \df(\g x_0,x) - r - 1.
\]
Combining this with (\ref{eqn:ergoano13}) and (\ref{eqn:ergoano10}), we get
\[
\df(\g x_0,\phi x_0) \le 2r+ 2\delta +3.
\]
\end{proof}

In particular, for each $\tau\in \mu_x\LR{\LRB{\Psi^* >\varepsilon}}$, $\#\LRB{\phi\in\G_d^*\mid \tau\in S(x: B(\phi x_0,r))}$ is uniformly bounded, say, by $D> 0$.
Therefore,
\be\label{eqn:ergoano16}
\sum_{\phi\in\G^*_{R}} \mu_x\LR{S(x: B(\phi x_0,r))}
\le D\mu_x\LR{ \bigcup_{\phi\in\G^*_{R}} {S(x: B(\phi x_0,r))} }.
\ee

We would like to use the shadow lemma (Theorem \ref{thm:SL}). To this end, we have
\be\label{eqn:ergoano17}
\mu_x\LR{{\LRB{\Psi^* >\varepsilon}}} \le \sum_{\phi\in\G^*_{R}} \mu_x\LR{S(x: B(\phi x_0,R))}
\le C' \sum_{\phi\in\G^*_{R}} \exp \left(-\beta \df(x,\phi x_0)\right)
\ee
where the first inequality is given by (\ref{eqn:ergoano11}) and the last inequality is given by the shadow lemma with $r_0 \le r = R$. Note that the necessary condition $\df(x,\phi x_0)\ge R$ which we needed to apply the shadow lemma in the above follows from the definition of $\G^*_R$.
Moreover, applying shadow lemma again with $r_0\le r = r$, we get another constant $C>0$ such that
\be\label{eqn:ergoano18}
C^{-1}\sum_{\phi\in\G^*_{R}} \exp \left(-\beta \df(x,\phi x_0)\right) \le
\sum_{\phi\in\G^*_{R}} \mu_x\LR{S(x: B(\phi x_0,r))}.
\ee
Combined with (\ref{eqn:ergoano16}), the inequalities in (\ref{eqn:ergoano17}) and (\ref{eqn:ergoano18}) give
\[
\mu_x\LR{\LRB{\Psi^* >\varepsilon}}
\le {DC'C}\mu_x\LR{ \bigcup_{\phi\in\G^*_{R}} {S(x: B(\phi x_0,r))} }.
\]
Finally, the above and (\ref{eqn:ergoano12}) yield
\[
\mu_x\LR{\LRB{\Psi^* >\varepsilon}}
\le \frac{2DC'C}{\varepsilon} \int_{\Ft} \Psi d\mu_{x}.
\]
This proves (\ref{eqn:ergodicAction5}). 
\end{proof}

The proof of the lemma follows from the sublemma by taking $\Phi$ in the sublemma to be the indicator function for $B$.
\end{proof}

\section{Hausdorff density}\label{sec:HD}

In this section, we restrict our attention to Anosov subgroups. Usually, one defines 
Hausdorff measures and Hausdorff dimension for metric spaces. In Appendix \ref{sec:AppA}, 
we verify that the theory goes through for  premetrics  as well. 
Therefore, we choose to work with premetric spaces.
The reader who prefers to work with metrics 
can assume that $\e>0$ in the following is chosen so small so that $\dg{\epsilon}_x$ is bilipschitz equivalent  a metric on $\Lt(\G)$ (cf. Corollary \ref{cor:gmetricano}). 

We fix an $\e>0$. For $\beta\ge 0$, we let $\H_{x}^{\tb,\epsilon,\beta}$ denote the {\em $\beta$-dimensional Hausdorff measure} 
on the {\em premetric space} $(\Lt(\G), \dg{\epsilon}_x)$. These definitions can be found in  Appendix \ref{sec:AppA}.
The {\em Hausdorff dimension} of a Borel subset $B\subset \LT$ is then defined as
\[
\hd^{\tb,\e}(B) = \inf\{\beta\mid \H_{x}^{\tb,\epsilon,\beta}(B) =0\} =\sup\{\beta\mid\H_{x}^{\tb,\epsilon,\beta}(B) =\infty\}.
\]
We observe that if for some $\beta\ge 0$, $\H_{x}^{\tb,\epsilon,\beta}(B)\in(0,\infty)$, then $\hd^{\tb,\e}(B) = \beta$.

\begin{rem}
 The Hausdorff dimension $\hd^{\tb,\e}(B)$ does not depend on the choice of a base-point $x\in X$, although the definition above involved such a basepoint.
 This follows from the fact that for any $x,z\in X$, the identity map
 \[
  \mathrm{id} : (\Lt(\G), \dg{\epsilon}_x) \ra (\Lt(\G), \dg{\epsilon}_z)
 \]
 and its inverse are locally Lipschitz maps. We show this in the proof of the next proposition.
 For this reason, we have dropped the basepoint from the notation of the Hausdorff dimension.
\end{rem}

\begin{prop}\label{prop:hd}
Suppose that for some $\beta\ge 0$
\be\label{eqn:hd}
\H_{x}^{\tb,\epsilon,\beta}(\LT) \in(0,\infty).
\ee
Let $Z = \G x$. Then $\H^{\tb,\epsilon,\beta} = \{\H_{z}^{\tb,\epsilon,\beta}\}_{z\in Z}$ is a  $\G$-invariant $\tb$-conformal $Z$-density of dimension $\beta\e$.
\end{prop}

\begin{proof}
Let $y,z\in Z$. 
Define a function $f: \LT\times\LT \ra \R_{\ge 0}$ by
\[
f(\t_1,\t_2) =\left\{
\begin{array}{ll}
\frac{\dg{\e}_y(\t_1,\t_2)}{\dg{\e}_z(\t_1,\t_2)}, & \t_1\ne \t_2,\\
\exp\LR{-{\e}\dhor_{\t}(y,z)}, & \t_1 = \t_2 = \t.
\end{array}
\right.
\]
By a calculation similar to the proof of Proposition \ref{prop:dgconf}, we obtain
\[
\lim_{{\t_1,\t_2\ra\t}}\frac{\dg{\e}_y(\t_1,\t_2)}{\dg{\e}_z(\t_1,\t_2)} = \exp\LR{-{\e}\dhor_{\t}(y,z)}
\]
which shows that $f$ is continuous.
For  $\t\in\LT$ and sufficiently small $\eta>0$, let $U_\eta$ be a neighborhood of $\t$ in $\LT$ such that $\forall \t_1,\t_2\in U_\eta$,
\[
{\dg{\e}_y(\t_1,\t_2)} \le \left(\exp\LR{-{\e}\dhor_{\t}(y,z)} + \eta\right){\dg{\e}_z(\t_1,\t_2)}.
\]
Hence the identity map $\id:(\LT,\dg{\e}_z)\ra (\LT,\dg{\e}_y)$ restricted to $U_\eta$ is  $L_\eta$-Lipschitz, where $L_\eta := \exp\LR{-{\e}\dhor_{\t}(y,z)} + \eta$. 
In particular, the map $\id$ is locally Lipschitz.
Therefore, for any $B\in\B(U)$, $\H^{\tb,\epsilon,\beta}_y(B) \le L^\beta_\eta\H^{\tb,\epsilon,\beta}_z(B)$. This also shows that $\H^{\tb,\epsilon,\beta}_{y} \ll \H^{\tb,\epsilon,\beta}_z$. Taking limit as $\eta \ra 0$, we obtain
\[
\frac{d\H^{\tb,\epsilon,\beta}_y}{d\H^{\tb,\epsilon,\beta}_z}(\t) \le \exp\LR{-{\beta\e }\dhor_{\t}(y,z)},
\]
and by switching the role of $y$ and $z$ in the above we also obtain the reverse inequality. Hence
\[
\frac{d\H^{\tb,\epsilon,\beta}_y}{d\H^{\tb,\epsilon,\beta}_z}(\t) = \exp\LR{-{\beta\e }\dhor_{\t}(y,z)}
\]
which proves $\tb$-conformality.
Suppose that $y = \g z$ for some $\g\in \G$. Then 
 for any $B\in \B(\LT)$,
\[
\H_{\g z}^{\tb,\epsilon,\beta}(B) = \int_B \exp\LR{-{\beta\e}\dhor_{\t}(\g z,z)}d\H_{z}^{\tb,\epsilon,\beta}
=\int_B d\LR{\g^*\H_{z}^{\tb,\epsilon,\beta}} = \g^*\H_{z}^{\tb,\epsilon,\beta}(B)
\]
and $\G$-invariance also follows.
Therefore, $\H^{\tb,\epsilon,\beta}$ is a conformal $Z$-density of dimension $\beta\e$.
\end{proof}

\begin{rem}\leavevmode
\begin{enumerate}
\item Note that if such a family $\{\H_z^{\tb,\epsilon,\beta}\mid z\in Z\}$ exists, then it may be extended to a {\em full} $\tb$-conformal density via the correspondence in (\ref{eqn:Adensity}).

\item By the uniqueness of $\tb$-conformal density (Theorem \ref{cor:uniqueAnosovDensity}), the number $\beta$ in Proposition \ref{prop:hd} equals to $\dF/\e$.

\item In the following  we shall see that, indeed, the  $\dF/\e$-dimensional Hausdorff measure $\H^{\tb,\epsilon,\frac{\dF}{\e}}_x$  is finite and non-null (i.e., it satisfies (\ref{eqn:hd})). 
\end{enumerate}
\end{rem}

Next we show that if $\beta = \dF/\e$, then the $\beta$-dimensional Hausdorff measure $\H^{\tb,\epsilon,\beta}_x$ satisfies (\ref{eqn:hd}).
Let us first discuss the simpler case, namely, when the pseudo-metric $\df$ is a metric. There is an abundance of examples when this occurs, e.g., in the case when $X = G/K$ is an irreducible symmetric space (i.e., $G$ is simple).

Let $(Y,d)$ be a proper, geodesic, Gromov hyperbolic metric space and $\G$ be a nonelementary discrete group of isometries acting properly discontinuously on $Y$. Let $\L$ be the limit set of $\G$ in $\vb Y$. Further, assume that $\G$ is {\em quasiconvex-cocompact}, i.e., the quasiconvex hull $\mr{QCH}(\L)$ is nonempty and  $\mr{QCH}(\L)/\G$ is compact. In \cite{MR1214072}, Coornaert proved the following result.
\begin{thm}[{\cite[Cor. 7.6]{MR1214072}}]\label{thm:coor}
Suppose that the critical exponent $\d$ of $\G$ is finite. Then
the  $\d$-dimensional Hausdorff measure on $\L$ with respect to a Gromov metric $D$ is finite and non-null.
\end{thm}

To apply this theorem to our case, we need an appropriate setup. In Section \ref{sec:hypmorse}, we proved that the orbit $Z = \G x$ is a Gromov hyperbolic space with respect to the Finsler metric $\df$ (cf. Corollary \ref{cor:hypano}) and it is also proper. But $Z$ fails to be geodesic.
This problem can be remedied by taking a uniform neighborhood $Y$ of $Z$ in $X$ such that $Z$ is quasiconvex in $Y$, and then putting the intrinsic path-metric $d$ on $Y$ induced by $\df$ (this requires positivity of $\df$), and finally by completing $Y$ in this metric. Then $(Y,d)$ is proper, geodesic and Gromov hyperbolic. Moreover, $(Y,d)$ and the isometrically embedded $(Z,\df)$ are Hausdorff-close and, in particular, $(Y,d)$ is quasiisometric to $(Z,\df)$ by a $(1,A)$-quasiisometry. This implies that there is a bilipschitz homeomorphism from $\vb Y$ (equipped with the metric $D^\e$ defined by $D^{\e}(\xi_1,\xi_2) = D(\xi_1,\xi_2)^\e$ where $D$ is a Gromov metric on $\vb Y$) to $(\LT, \dg{\e}_x)$. 
Note that the action $\G \acts (Y,d)$ satisfies all the properties needed to apply Theorem \ref{thm:coor}.
Therefore, by this theorem the $\dF/\e$-dimensional Hausdorff measure on $\vb Y$ (and, consequently, also on $\LT$) is finite and non-null.

In the general case where the positivity of $\df$ is unknown, the above argument still
works after some modifications. Let us go back to our construction in the above paragraph. Let $Y$ be a uniform Riemannian neighborhood of $Z$ in which $Z$ is quasiconvex w.r.t. $\df$. Define a new $\G$-invariant metric  ${\bar d}_{\tb}$ on $Y$ by
\[
{\bar d}_{\tb}(y,z) = \max\big\{ \df(y,z),\ \ve\dr(y,z)\big\}, \quad\forall y,z\in Y
\]
where $\ve>0$ is some number that is strictly lesser than $L^{-1}$ given in (\ref{eqn:metricRF}). Note that for $y,z\in Z$, if $\df(y,z)$ is sufficiently large, then ${\bar d}_{\tb}(y,z) = \df(y,z)$. Moreover, for a given $\i$-invariant compact subset $\Theta\subset \ost(\tmod)$ and a possibly smaller $\ve$ (depending on the choice of $\Theta$), any $\Theta$-Finsler geodesic (see Definition \ref{defn:fg}) connecting these two points remains  a geodesic in this new metric. In other words, $Z$ remains quasiconvex in $Y$ with respect to ${\bar d}_{\tb}$.

Observe that the identity embedding $(Z,\df) \ra (Y,{\bar d}_{\tb})$ is a $(1,A)$-quasi\-isometric embedding for some large enough $A$ and the image is Hausdorff-close to $Y$. 
Therefore, in this case also we get a natural identification of the Gromov boundaries of $(Z,\df)$ and $(Y,{\bar d}_{\tb})$.
Next, considering intrinsic metrics, we complete $Y$ as before to get a proper, geodesic, Gromov hyperbolic space $(Y,d)$. The rest of the argument works as before.

Using Proposition \ref{prop:hd} together with the remark after the proposition, we obtain the following result.

\begin{thm}\label{thm:hd}
Suppose that $\G$ is a nonelementary $\tmod$-Anosov subgroup of $G$. 
If $\beta = \dF/\e$, then the $\beta$-dimensional Hausdorff density $\H^{\tb,\epsilon,\beta} = \{\H_{z}^{\tb,\epsilon,\beta}\mid z\in X\}$ is a $\G$-invariant $\tb$-conformal density of dimension $\dF$. 
In particular, the Hausdorff dimension with respect to the metric $\dg{\e}_x$ satisfies
\[
\hd^{\tb,\e}(\LT) = \frac{\dF}{\e}.
\]

Moreover, $\H^{\tb,\epsilon,\beta}$ equals to a non-zero multiple of the Patterson--Sullivan density of type $\tb$ corresponding to $\G$. \end{thm}

We have mostly completed the proof of this theorem. The only remaining ``moreover'' part follows from the uniqueness of the $\G$-invariant $\tb$-conformal density (Theorem \ref{cor:uniqueAnosovDensity}).

\begin{cor}
Let $\G$ be a $\tmod$-Anosov subgroup of $G$. For all $x\in X$, the Hausdorff dimension of the premetric space $(\LT, \dg{1}_x)$ equals to $\dF$.
\end{cor}

\section{Examples}\label{sec:exPAR}

\subsection{Product of two hyperbolic planes}\label{sec:product}

Let $\G_1,\ \G_2$ be isomorphic discrete cocompact subgroups of $\PSL{2}$ where the isomorphism is given by $\phi :\G_1\ra \G_2$.
We let $f: S^1\to S^1$ be the equivariant homeomorphism of ideal boundaries of hyperbolic planes determined by $\phi$. 

The discrete subgroup 
$$
\G = \{ (\g_1,\phi\g_1) \mid \g_1\in \G_1 \} < G=\PSL{2}\times \PSL{2}$$
acts  on $X= \mb{H}^2\times \mb{H}^2$ as a $\smod$-Anosov subgroup. (This follows, for instance, from the fact that $\G$ is an URU subgroup of $G$.) The $\smod$-limit set of $\G$ in the full flag manifold $S^1\times S^1$ equals the graph of the map $f$.

We denote $d_1$ (resp. $d_2$) the distance functions of the constant $-1$ curvature Riemannian metrics on the first (resp. second) factor of the product $\mb{H}^2\times \mb{H}^2$. 

Unlike in section \ref{sec:dg}, we work with the Finsler metric on $\mb{H}^2\times \mb{H}^2$ given by 
\begin{equation}\label{eqn:sec:product}
\df((x_1,x_2),(y_1,y_2)) = \frac{d_1(x_1,y_1) + d_2(x_2,y_2)}{2}.
\end{equation}
Basically, we have multiplied  the distance function in (\ref{eqn:dfrank1}) corresponding to $\tb = (\frac{1}{\sqrt{2}},\frac{1}{\sqrt{2}})$, for $p=2$, by a factor $1/\sqrt 2$ in order to avoid cumbersome radical constants.

By the formula of the Gromov predistance (\ref{eqn:dgrank1}), for $\epsilon=1$ and $x=(x_1,x_2)$, 
$\dg{1}_x(\tau_+,\tau_-)$ is bilipschitz equivalent to 
the product
$$
\sqrt{\alpha_1 \alpha_2}, 
$$ 
where $\tau_\pm= (\xi_1^\pm, \xi_2^\pm)$ and $\alpha_i$ is the angle between $\xi_i^+, \xi_i^-$ as measured from $x_i$, $i=1,2$.

By \cite[{Thm. 2 \& 3}]{MR1208564} we note that the $\tb$-critical exponent $\dF$ of $\G$ is at most $1$.
This can also be obtained by comparing the Hausdorff dimensions as follows.
Note that by the formula of the Gromov predistance, the identity map 
$$
(S^1\times S^1,\rho) \to (\Flag(\smod), \dg{1}_x)
$$
is Lipschitz, where  $\rho$ is a Riemannian distance function on  
$S^1\times S^1= \vb\mb{H}^2\times \vb\mb{H}^2$.
Moreover, the limit set of $\G$ in $S^1\times S^1$ is the graph of a BV (i.e., bounded variation) function, hence, is a rectifiable curve, and, thus, 
has Hausdorff dimension $1$ with respect to $\rho$. 
Consequently, with respect to $\dg{1}_x$, $\hd(\L_\smod(\G)) \le 1$.  By Theorem \ref{thm:hd}, $\dF \le 1$ as well.

Moreover, by \cite[{Thm. 2}]{MR1208564} and Corollary \ref{cor:div}, $\dF = 1$ if and only if $\phi$ is induced by an isometry of $\mb{H}^2$, equivalently, $f$ is 
a M\"obius transformation.

We further note that one can use \cite{MR1230298} as an alternative argument for both inequality and the equality case.

\subsection{Hilbert entropy of projective Anosov representations}\label{sec:projano}
A subgroup $\G < \SL{k+1}$, $k\ge 2$, is called {\em projective Anosov} if it is $\tmod$-Anosov for $\tmod = (1,k)$ (see Examples \ref{ex:2}, \ref{ex:2.1}, and \ref{ex:2.3} for notations). The $\tb$-critical exponent associated to the unique $\i$-invariant type 
$$\tb = \left(\frac{1}{2\sqrt{k+1}}, 0 , -\frac{1}{2\sqrt{k+1}}\right)$$  in $\tmod$ will be denoted, as usual, by $\dF$.

Let $\G < \SL{k+1}$ be a projective Anosov subgroup. In \cite{GMT19}, the authors defined the following two critical exponents of $\G$, namely, the {\em Hilbert critical exponent} (corresponding to the sum of all simple roots)
\[
\d_{1,k+1} = \limsup_{r\ra\infty} \frac{\log \card \{\g\in \G \mid \sigma_1(\g) - \sigma_{k+1}(\g)< r\}}{r}
\]
and the {\em simple root critical exponent} (corresponding to the first simple root)
\[
\d_{1,2} = \limsup_{r\ra\infty} \frac{\log \card \{\g\in \G \mid \sigma_1(\g) - \sigma_2(\g)< r\}}{r}.
\]

A direct computation yields 
$$\sqrt{k+1}\dF = \d_{1,k+1}, $$
which follows from the formula of $\df$ given by (\ref{eq:ex:2.1}).
Also note that (by (\ref{eqn:ex2.3})) for a pair of partial flags $(l_1, h_1), (l_2,h_2) \in\Ft$,
\begin{equation}\label{eqn:ineq_proj_gromov_distance}
\dg{{1}/{\sqrt{k+1}}}_x\LR{(l_1, h_1), (l_2,h_2)} = \sqrt{\sin\angle(l_1,h_2)}\sqrt{\sin\angle(l_2,h_1)}  \le \sin\angle(l_1,l_2), 
\end{equation}
where the right side equals the distance (with respect to the constant curvature Riemannian metric on $\R P^k$ determined by $x\in X$) between the points $l_1,l_2$ in $\R P^{k}$.
This together with Theorem \ref{thm:hd} implies that
\be\label{eqn:PAlb}
\d_{1,k+1} = \sqrt{k+1}\dF=\hd^{\tb,1/\sqrt{k+1}}(\LT) \le \hd^\mr{Riem}(\xi^1(\vb \G)),
\ee
where $\xi^1: \vb \G \ra \R P^{k}$ is the $\G$-equivariant embedding\footnote{Composition of the $\G$-equivariant  boundary embedding $ \vb \G\ra\Ft$ and the projection map $\Ft \ra\R P^{k} = \gr_1(\R^{k+1})$.}
of $\vb\G$ into $\R P^{k}$
and $\hd^\mr{Riem}$ denotes the Hausdorff dimension with respect to the Riemannian metric.

The critical exponent $\delta_{1,2}$ is known to give an upper bound for $\hd^\mr{Riem}(\xi^1(\vb \G))$ (see \cite[Prop. 4.1]{Pozzetti:2019aa} or \cite[Thm. 4.1]{GMT19}).
By above, we obtain a lower bound.

\begin{thm}\label{thm:PARhd}
Let $\G<\SL{k+1}$ be a projective Anosov subgroup. Then
\[
\d_{1,k+1} \le \hd^\mr{Riem}(\xi^1(\vb \G)) \le \dim \mathbb{R}P^{k} = k.
\]
\end{thm}

\begin{rem}
If one considers the $\tmod$-flag limit set $\Lt$ of $\G$ in the  flag manifold $\Ft$, then a similar lower bound can be obtained for the Hausdorff dimension corresponding to the Riemannian metric. Note that (\ref{eqn:ineq_proj_gromov_distance}) also holds if one replaces $\sin\angle(l_1,l_2)$ by $\sin\angle(h_1,h_2)$. Hence, 
\begin{equation}\label{eqn:dim_flag_limit_set}
\dg{{1}/{\sqrt{k+1}}}_x\LR{(l_1, h_1), (l_2,h_2)} \le \sqrt{2\sin^2\angle(l_1,l_2) + 2\sin^2\angle(h_1,h_2)}
\end{equation}
which shows that the identity map $(\Lt,\rho) \ra (\Lt,\dg{{1}/{\sqrt{k+1}}})$ is a Lipschitz map, where $\rho$ is a(ny) Riemannian distance function in $\Ft$. Therefore,
\[
 \d_{1,k+1} \le \hd^\mr{Riem}(\Lt).
\]
This recovers the lower bound obtained in \cite[Cor. 1.2]{GM16}.
\end{rem}

\subsection{Hilbert entropy of $\PSL{3}$-Hitchin representations}
In suitable projective Anosov classes, one may be able to improve the inequality in (\ref{eqn:ineq_proj_gromov_distance}) to get a better bound for the Hilbert critical exponent $\delta_{1,k+1}$. Here we present an example of such an improvement. 

 Let $\G = \pi_1(S)$ be a surface\footnote{More precisely, $S$ is a closed surface of genus $g\ge 2$.} group. By \cite{choi-goldman}, the $\PSL{3}$-Hitchin representations $\rho : \G \ra \PSL{3}=\SL{3}$ consist of holonomies of convex $\R P^2$-structures in $S$.
In particular, $\rho(\G)$ preserves a properly convex (open) domain $\Omega$ in $\R P^2$ with ($C^1$-) boundary $\partial \Omega = C$, and the action $\G \acts \Omega$ is properly discontinuous.

Since Hitchin representations are $P_1$-Anosov, Theorem \ref{thm:PARhd} shows that
\[
 \delta_{1,3}(\rho(\G)) \le 2.
\]
However, this {\em most general} upper bound is weak for the Hitchin representations. In Proposition \ref{prop:Hitchin-critical} below we will obtain a stronger bound. 

Let $\tmod = (1,2)$.
The $\tmod$-flag limit set of $\rho(\G)$ can be equivariantly identified with the set  of flags $\{(\xi,\xi^*) \mid \xi \in C\}$, where $\xi^*\subset \R P^2$ is the line tangent to $C$ through $\xi$. We denote the dual curve $\{\xi^* \mid \xi\in C \}\subset (\R P^2)^*$ by $C^*$; this curve $C^*$ also bounds a properly convex domain in $(\R P^2)^*$.

\begin{claim}
 There exists a constant $L$ such that, for every $\xi,\eta\in C$,
\begin{equation}\label{eqn:hit_three_gromov}
D(\xi,\eta):=\sqrt{d_{\R P^2}(\xi,\eta^*)d_{\R P^2}(\eta,\xi^*)} \le {L}\cdot  d_{\R P^2}(\xi,\eta)d_{(\R P^2)^*}(\xi^*,\eta^*).
\end{equation}
\end{claim}

\begin{proof}
 The curve $C$ is a simple closed $C^1$-curve in an affine chart $\R^2$, bounding the convex subset $\Omega$ above.  Let $d_{\R^2}(\xi,\eta) = \psi$ and $\angle(\eta^*,\xi^*) = \theta$ as in Figure \ref{fig:curve}. When $\xi\ne \eta$ are uniformly close, $0< \psi\le \psi_0$, then the angle $\theta$ is acute. See Figure \ref{fig:curve}. Then,\footnote{We remind our reader that the symbol $\asymp$ is used mean that the ratio of both sides is bounded above and below by some positive constants.} $d_{\R^2}(\xi,\eta^*) \asymp \psi\beta$, $d_{\R^2}(\eta,\xi^*) \asymp \psi\alpha$, and $\alpha +\beta = \theta$.
Thus,
\[
 \frac{d_{\R^2}(\xi,\eta^*)d_{\R^2}(\eta,\xi^*)}{d_{\R^2}(\xi,\eta)^2(\angle(\eta^*,\xi^*))^2} 
 \asymp \frac{\psi^2\alpha\beta}{\psi^2(\alpha+\beta)^2} 
 = \frac{\alpha\beta}{(\alpha+\beta)^2} \le \frac{1}{4}.
\]
So,
\[
 \sqrt{d_{\R^2}(\xi,\eta^*)d_{\R^2}(\eta,\xi^*)}\le \const\cdot {d_{\R^2}(\xi,\eta)(\angle(\eta^*,\xi^*))}.
\]

Also, since $C$ (resp. $C^*$) is bounded in the affine chart, the euclidean distances (resp. the angular distance) above are equivalent to the projective distances in (\ref{eqn:hit_three_gromov}). Hence the above inequality justifies (\ref{eqn:hit_three_gromov}). 
\end{proof}

\begin{figure}
\centering
\begin{tikzpicture}
\node[anchor=south west,inner sep=0] (image) at (0,0,0) {\includegraphics[scale=.5]{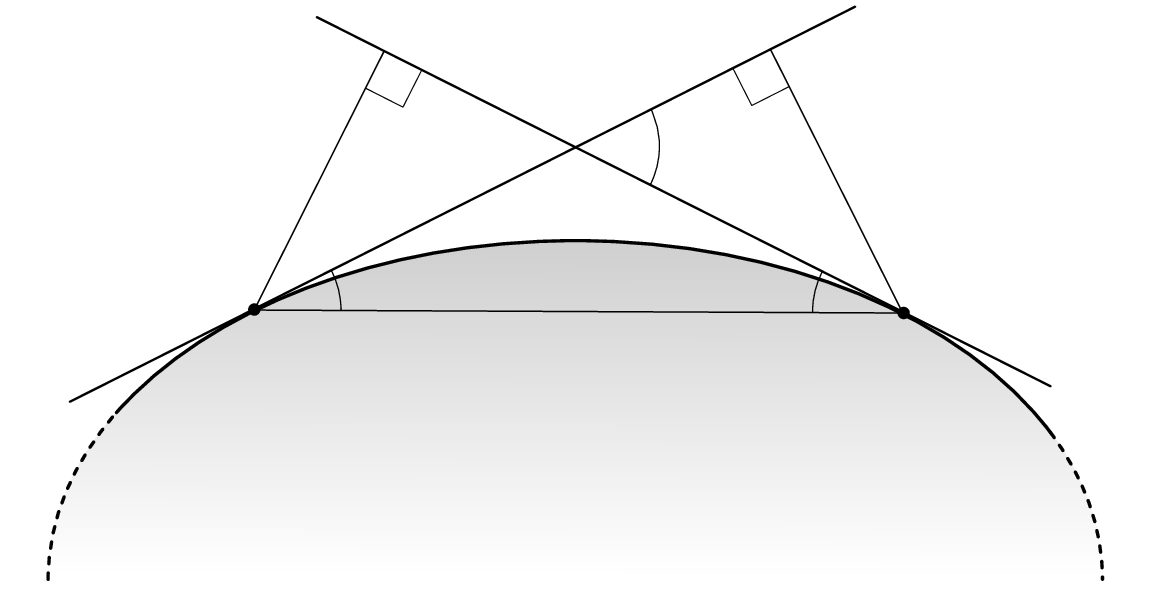}};
\begin{scope}[x={(image.south east)},y={(image.north west)}]
\node at (.22,.4) {$\xi$};
\node at (.55,.75) {$\scriptstyle\theta$};
\node at (.75,.94) {$\xi^*$};
\node at (.26,.92) {$\eta^*$};
\node at (.77,.41) {$\eta$};
\node at (.5,.43) {$\scriptstyle \psi$};
\node [rotate=65] at (.26,.7) {$\scriptstyle \psi\beta$};
\node [rotate=302] at (.75,.7) {$\scriptstyle \psi\alpha$};
\node at (.28,.498) {$\scriptstyle \alpha$};
\node at (.72,.488) {$\scriptstyle \beta$};
\node at (.98,.1) {$C$};
\end{scope}
\end{tikzpicture}
\caption{}
\label{fig:curve}
\end{figure}

Also, note that the premetric $D$ in left side of (\ref{eqn:hit_three_gromov}) is bilipschitz equivalent to the the premetric $\dg{{1/\sqrt{3}}}_x$ (see the formula in (\ref{eqn:ineq_proj_gromov_distance})), left side of the inequality in  (\ref{eqn:dim_flag_limit_set}).
But, the  square root of the right side of (\ref{eqn:hit_three_gromov}) is
\[
 \sqrt{L}\sqrt{d_{\R P^2}(\xi,\eta)d_{(\R P^2)^*}(\xi^*,\eta^*)} \le \sqrt{\frac{L}{2}} \sqrt{d_{\R P^2}(\xi,\eta)^2 + d_{(\R P^2)^*}(\xi^*,\eta^*)^2}.
\]
In the right side of the above inequality, we obtain a multiple of the Riemannian distance in $\Ft$.
Thus, the identity map
\[
 (\Lt, \rho) \ra (\Lt, \dg{{1/\sqrt{3}}}_x)
\]
is a $2$-H\"older map. Here $\rho$ denote the distance function of the Riemannian metric on $\Ft$. It is known that $\Lt \subset \Ft$ is a Lipschitz curve, and hence \[\hd^\mr{Riem}(\LT) = 1.\]

We obtain,
\[
 \d_{1,3}(\rho(\G)) = \hd^{\tb,1/\sqrt{3}}(\LT) \le \frac{\hd^\mr{Riem}(\LT)}{2} = \frac{1}{2}
\]
(cf. (\ref{eqn:PAlb})). We record this result in the next proposition.

\begin{prop}\label{prop:Hitchin-critical} 
 For any Hitchin representation $\rho: \G \ra \PSL{3}$,
 \[
 \d_{1,3}(\rho(\G))\le \frac{1}{2}.
\]
\end{prop}

Compare with the last paragraph of \cite[p. 892]{PS17}.

\appendix\normalsize
\section{Hausdorff measures on premetric spaces}
\label{sec:AppA}

Let $X$ be a metrizable topological space. Recall that an {\em outer measure} is a function $\mu : \mc{P}(X) \ra [0,\infty]$ that satisfies
\begin{enumerate}[(i)]
\setlength\itemsep{0em}
\item $\mu(\emptyset) = 0$,\label{ou:prop1}
\item for all $A, B\in\mc{P}(X)$ with $A\subset B$, $\mu(A) \le \mu(B)$, and \label{ou:prop2}
\item for all countable collection $\LRB{A_k \mid k\in\N}$ of subsets of $X$, \label{ou:prop3}
\[
\mu\LR{\bigcup_{k\in\N} A_k } \le \sum_{k\in\N} \mu(A_k).
\]
\end{enumerate}
A set $A\subset X$ is called {\em $\mu$-measurable} if for every $E\in\mc{P}(X)$, $\mu(A) = \mu(A \cap E) + \mu (A\cap E^c)$.
By Carath\'eodory's theorem (cf. \cite[Thm. 1.11]{MR1681462}), $\mu$-measurable sets form a $\sigma$-algebra to which $\mu$ restricts as a complete measure.

Assume now that $X$ is compact. The outer measure $\mu$ is called {\em good} if additionally,

\begin{enumerate}[(i)]
\setcounter{enumi}{3}
\item for all  $A,B \subset X$ with $\bar{A}\cap \bar{B} = \emptyset$, $\mu(A\cup B) = \mu(A) + \mu(B)$. \label{ou:prop4}
\end{enumerate}

The next lemma asserts that, for outer measures $\mu$ on compact metrizable spaces, 
the $\sigma$-algebra of Borel sets is a subalgebra of the $\sigma$-algebra of $\mu$-measurable sets.

\begin{lem}\label{lem:goodmeasure}
Let $X$ be a compact metrizable space. If $\mu$ is a good outer measure on $X$, then every Borel set $B\in \B(X)$ is measurable.
\end{lem}
\begin{proof}
Let $d$ be a metric on $X$. Then the condition (\ref{ou:prop4}) above implies that 
\begin{enumerate}[(i')]
\setcounter{enumi}{3}
\item for all $A,B \subset X$ with $d(A,B)>0$, $\mu(A\cup B) = \mu(A) + \mu(B)$.
\end{enumerate}
Therefore, $\mu$ is a metric outer measure on $(X,d)$.
By \cite[Prop. 11.16]{MR1681462},  Borel subsets of $X$ are measurable.
\end{proof}

\begin{defn}[Premetric space]
Let $X$ be a  topological  space.
A  symmetric continuous function $d: X\times X \ra [0,\infty]$
is called a {\em premetric} on $X$. 
A pair $(X,d)$ consisting of a metrizable topological space $X$ and a premetric $d$ on $X$ is called a {\em premetric} space.
\end{defn}

In what follows, we consider only {\em positive} premetrics, i.e.,
\[
d(x,y)>0 \ \iff \ x\ne y, \quad \forall x,y\in X
\] 

Let $(X,d)$ be a compact positive premetric space. Then $d$ satisfies the following {\em separation property}:
\be\label{eqn:sep}
d(A,B) > 0 \ \iff \ \bar{A} \cap \bar{B} = \emptyset,\quad \forall A,B \subset X.
\ee Let $\ve> 0$, $\beta >0$.
For every $A\subset X$, define
\[
\H^\beta_\ve(A) = \inf_{\mc{U}}\LRB{ \sum_{k\in\N} \diam_d(U_k)^\beta \ \bigg\vert \ 
\mc{U} = \LRB{U_k\mid k\in\N}\text{ covers }A,\ 
\mr{mesh}(\mc{U}) \le \ve }.
\]
In the above, $\mr{mesh}(\mc{U})$ is the supremum of the $d$-diameters of the members of $\mc{U}$.
Then 
$$
\H^\beta_\ve : \mc{P}(X) \ra [0,\infty]$$
 is an outer measure on $X$ (cf. \cite[Prop. 1.10]{MR1681462}).
Define the $\beta$-dimensional {\em Hausdorff measure} $\H^\beta$ by
\[
\H^\beta(A) = \lim_{\ve\ra 0} \H^\beta_\ve(A).
\]

\begin{thm}
The Hausdorff measure $\H^\beta$ is a good outer measure.
\end{thm}

\begin{proof}
We need to check the properties (\ref{ou:prop1})-(\ref{ou:prop4}) above. 
Since, for all $\ve>0$, $\H_\ve^\beta$ is an outer measure, taking limit $\ve \ra 0$, properties (\ref{ou:prop1})-(\ref{ou:prop3}) are easily verified. 
Therefore, we only need to check that $\H^\beta$ satisfies property (\ref{ou:prop4}).

Let $A, B\subset X$ such that $\bar{A} \cap\bar{B} =\emptyset$. By (\ref{eqn:sep}), $d(A,B) = d_0 > 0$.
Let $\ve < {d_0}$ be a positive number and $\mc{U}$ be a countable open cover 
of $A\cup B$ with $\mr{mesh}(\mc{U})\le \ve$. 
If such open cover does not exist, then $\H^\beta_\ve(A\cup B)$ (and hence, $\H^\beta(A\cup B)$) is infinity.
Otherwise, $\mc{U}$ can be written as a disjoint union $\mc{U}_A \sqcup \mc{U}_B$ where $\mc{U}_A$ consists of all open sets in $\mc{U}$ that intersect $A$ and $\mc{U}_B$ consists of the rest.
Clearly, $\mc{U}_A$ and $\mc{U}_B$ are open covers of $A$ and $B$, respectively.
Therefore,
\[
\sum_{E\in\mc{U}} \diam_d(E)^\beta 
= \sum_{E\in\mc{U}_A} \diam_d(E)^\beta + \sum_{E\in\mc{U}_B} \diam_d(E)^\beta
\ge \H^\beta_\ve(A) + \H^\beta_\ve(B).
\]
Since the above holds for any cover $\mc{U}$ with mesh $\le \ve$, we have
\[
\H^\beta_\ve(A\cup B) \ge \H^\beta_\ve(A) + \H^\beta_\ve(B).
\]
Taking limit $\ve\ra 0$, we get $\H^\beta(A\cup B) \ge \H^\beta(A) + \H^\beta(B)$.
The reverse inequality follows from property (\ref{ou:prop3}). Therefore, $\H^\beta(A\cup B) = \H^\beta(A) + \H^\beta(B)$.
This completes the proof.
\end{proof}

By Lemma \ref{lem:goodmeasure} and the above theorem, we obtain the following result.

\begin{cor}
Every Borel subset of $X$ is $\H^\beta$-measurable.
\end{cor}

The {\em Hausdorff dimension} of a Borel subset $B\subset (X,d)$ is then defined as
\[
\hd(B) = \inf\{\beta\mid \H^{\beta}(B) =0\} =\sup\{\beta\mid\H ^{\beta}(B) =\infty\}. 
\]

\section{A brief overview of the history of the Patterson--Sullivan theory in higher rank}\label{appen:PShigherrank}
For reader's convenience, in this appendix we discuss connection of our work  with some notable earlier papers on the Patterson--Sullivan theory for higher rank symmetric spaces (of noncompact type). 

To the best of our knowledge, Bishop--Steger \cite{MR1208564} and Burger \cite{MR1230298} were the first to investigate  Poincar\'e series associated with a discrete isometry group $\Gamma$ of a symmetric space of rank $\ge 2$.
Precisely, for a pair of Fuchsian subgroups $\Gamma_i< \isom(\mathbb{H}^2)$, $i=1,2$, and an isomorphism $\phi: \G_1\ra\G_2$, Burger took the subgroup $\Gamma = \LRB{(\g_1,\phi(\g_1)) \mid \g_1\in\G_1 }< \isom(\mathbb{H}^2 \times\mathbb{H}^2)$ and considered the {\em Finsler} Poincar\'e series
\[
\sum_{(\g_1,\g_2)\in\G} \exp\LR{ -s\LR{a d_1(\gamma_1 x_1, x_1) + b d_2(\g_2x_2,x_2) } }, \quad (a,b)\in\mathbb{R}_+^2,
\]
where $d_1$ (resp. $d_2$) and $x_1$ (resp. $x_2$) denote the distance function and a fixed point, respectively, in the first (resp. second) factor of $\mathbb{H}^2 \times\mathbb{H}^2$.
Then he studied the set of all points $(a,b) \in \mathbb{R}_+^2$ for which the Poincar\'e series has the critical exponent $s=1$ which he called the {``Manhattan curve.''}
He proved that each point $(a,b)$ in this set gives rise to a unique ``$(a,b)$-dimensional'' density with respect to which the action $\Gamma\acts \vb\mathbb{H}^2 \times\vb\mathbb{H}^2$ is ergodic, \cite[Thm. 4]{MR1230298}.
The Theorem \ref{mainA} of our paper illustrates this result, compare  Subsection \ref{sec:product},
although in that example we only consider the case $a=b$, 
the general case can also be obtained by suitably changing the weights in the formula (\ref{eqn:sec:product}) for the Finsler metric.
On the other hand, Bishop and Steger considered the closely related Poincar\'e series
\[
\sum_{(\g_1,\g_2)\in\G} \exp\LR{ -s d_1(\gamma_1 x_1, x_1) - (1-s) d_2(\g_2x_2,x_2) }, \quad s\in(0,1),
\]
and showed that it diverges if and only if $\phi$ is induced by an isometry of $\mathbb{H}^2$.
We discussed the connection of this result with our work in Subsection \ref{sec:product} for $s = 1/2$; for a general $s$, one needs to modify the Finsler metric as above.

Patterson--Sullivan measures for general symmetric spaces $X$ were introduced and studied by Albuquerque \cite{MR1675889}.
His main result is that for a generic\footnote{Here $\G$ is called generic if the support of any Patterson--Sullivan density lies in the regular part of the visual boundary.} Zariski-dense discrete subgroup $\Gamma$ of $G = \isom_0(X)$, the support of any Patterson--Sullivan density of dimension $\dR(\G)$ in the visual boundary of $X$ lies in a single regular $G$-orbit $G\cdot\xi_0$ where $\xi_0$ is the direction in the Weyl chamber along which the growth-rate of $\Gamma$-orbits in $X$ is ``maximal.''
More precisely, any Patterson--Sullivan density of dimension $\dR(\G)$ is supported on $G\cdot\xi_0 \cap \Lambda(\Gamma)$.
He also proved that the critical exponent $\dF(\Gamma)$ defined in terms of the Finsler pseudometric induced by the linear functional in the Weyl chamber dual to a vector $\xi\in \smod$, is minimal precisely when $\xi = \xi_0$.
Albuquerque further showed that this minimal Finsler critical exponent equals to the Riemannian critical exponent.

Quint \cite{MR1933790, MR1935549} generalized Albuquerque's results to arbitrary Zariski-dense discrete subgroups $\G < G$ and to general partial flag-manifolds $G/P_{\tmod}$. 
For such a group $\G$, he defined the {\em indicator of growth} function
\[
\psi_\G : \Delta \ra \R\cup \{-\infty\},
\]
which can be regarded as a higher rank analogue of the critical exponent,  
and showed that it is strictly positive in the interior of the Benoist's limit cone of $\G$. Note, however, that the orbital counting (done in the definition of the function $\psi_\G$) 
in Quint's paper is different from ours since he takes the infimum of exponents of convergence $\tau_{{\mathcal C}}$ over certain open cones ${\mathcal C}\subset \Delta$.  
For a positive linear function $\phi$ on $\Delta$ satisfying 
$\phi \ge \psi_\G$ and $\phi(x) =  \psi_\G(x)$ for some $x\in \mathrm{int}(\Delta)$, Quint (using the Patterson--Sullivan construction) defined a 
``$(\G,\phi)$-Patterson measure'' supported on the flag-limit set of $\G$ in $G/P_\theta$. Observe that each $\phi$ defines a $G$-invariant polyhedral Finsler metric on 
$X$ by the composition $\phi\circ d_\Delta$.  This part of Quint's work has nontrivial overlap with our's (specifically, the 
construction of Patterson--Sullivan densities 
in section \ref{sec:CD} of our paper). However, neither 
work subsumes the other since we do not assume Zariski density but require subgroups to be $\tmod$-RA, while Quint does not assume the $\tmod$-RA property  
but requires Zariski density. Quint's work contains wealth of other results such as a proof the concavity property of $\psi_\G$ and its applications: These issues are not discussed at all 
in our paper.

Taking inspiration from \cite{MR1230298} and building on \cite{MR1675889}, Link \cite{Lin04} considered a  class of densities (coming from the Patterson--Sullivan construction) on the visual boundary of $X$ associated with a Zariski-dense discrete subgroup $\Gamma < \isom_0(X)$ that generalize the conformal densities. 
After introducing a notion of Hausdorff measures appropriate for her setup, Link showed that for a regular point $\xi\in\vb X$, the Hausdorff dimension of $G\cdot\xi \cap \Lambda(\G)$ is bounded above by a suitable exponent of growth of $\G$-orbits in $X$; she also proved the equality 
for a certain class of groups which she calls ``radially cocompact'' (see \cite[Sec. 6]{Lin04}). 
In her subsequent work \cite{Lin06}, Link  proved that, the action of $\G$ on the ``ray-limit set'' is ergodic\footnote{More precisely, she proves that if $A$ is a measurable $\G$-invariant subset of the ray-limit set, then either $A$ is a null set or the complement of $A$ in the full limit set is null. However, this result does not exclude the possibility that the ray-limit set itself is a null set for the generalized Patterson--Sullivan measures.} with respect to these generalized Patterson--Sullivan measures\footnote{In fact, Albuquerque \cite{MR1675889} also attempted to prove ergodicity, but there was a gap in the proof. This gap was discovered and fixed by Link \cite[p. 612]{Lin06}.}.
These results are parallel to the equality of the Finsler critical exponent and the Hausdorff dimension of flag-limit sets in $G/P_{\tmod}$ 
for $\tmod$-Anosov subgroups (Theorem \ref{thm:hd}) and the ergodicity theorem (Theorem \ref{thm:ergoAno}) proved in our paper. 
However, the conicality condition for limit points of Anosov subgroups is a vast relaxation of Link's notion of radial limit points. Moreover, it follows from the main theorem of \cite{MR1437472} and Lemma \ref{lem:unifcloseAno} that, say, a Zariski dense $\tmod$-Anosov subgroup can never be radially cocompact unless $G$ has rank one. 

To conclude the discussion, we would like to repeat that Albuquerque, Quint and Link treat general Zariski-dense discrete subgroups, while 
we study Anosov and, more generally, regular antipodal subgroups, but do not assume Zariski density. 
While many of our proofs work for general $\tmod$-RA subgroups, the  Anosov property is critical in several places, for instance, in the proof of vanishing of 
$\dF^\con$ (Proposition  \ref{prop:CEconAno}), the proof of ergodicity, the proof that (in a suitable range) the Gromov pre-metric on the limit set is a metric, etc.
We would also like to point out that, even in the Zariski-dense case, parts (ii), (iv) and (v) 
(ergodicity, divergence and relation of Finsler critical exponent and Hausdorff dimension of the limit set)  
in our Theorem \ref{mainA} is new since, to our knowledge, the most general case in this direction was considered by Link \cite{Lin06} 
but was conditioned on the density of the ray-limit set. 

\section{A discussion of the main results without the $\iota$-invariance condition on the type $\tb$}\label{appen:without_iota_invariance}

In this last appendix, on our referee's suggestion, we show that Theorem \ref{mainA} (except for item (\ref{mainthm:part4})), and Theorem \ref{mainB} hold true without the $\iota$-invariance condition that we imposed on the type $\tb \in \mr{int}(\tmod)$.

Suppose that $\tb \in \smod$ is a type which is not assumed to be $\iota$-invariant.
Then the $G$-invariant pseudo-metric $d_\tb : X \times X \ra [0,\infty)$ defined in (\ref{def:df}) is in general {\em asymmetric}.
However, the following still holds: Since $d_\Delta(y,x) = \iota d_\Delta(x,y)$, we have,
\begin{equation}\label{eqn:assym_finsler_dist}
  d_\tb(y,x) = d_{\iota \tb} (x,y), \quad \forall x,y\in X.  
\end{equation}
The inequality (\ref{eqn:dist_ineq}) still holds.
The function $d_\tb$ satisfies the triangle inequality for the asymmetric distance functions:
\[
 d_\tb(x,z) \le d_\tb(x,y) + d_\tb(y,z), \quad \forall x,y,z\in X.
\]
See \cite[Subsec. 5.1.2]{MR3811766} for more details.

Let $\tb \in {\rm int}\,\tmod$.
In the asymmetric case, we define the $\tb$-critical exponent $\dF$ of a discrete group $\G < G$  in the same way as it was done for the symmetric case in Section \ref{sec:CE}, see (\ref{def:CE}).
As a consequence of the above triangle inequality, the critical exponent does not depend on the choice of a base-point in $X$ (see the paragraph after Remark \ref{rem:remtwopointfour}).
The definition of $\tb$-convergence/divergence type is also same as in the symmetric case (see Definition \ref{def:convergencetype}).
The Proposition \ref{prop:finiteCE} for uniformly $\tmod$-regular groups $\G$, which only depends on the fact that $\dr$ and $\df$ are coarsely equivalent on an $\G$-orbit in $X$, is valid in this case.

All the definitions (in particular, the crucial definition of $\tb$-conformal densities) and results in Sections \ref{sec:CD} remain valid in the asymmetric case.

We leave the results in Sections \ref{sec:hypmorse} and \ref{sec:dg} as they are; those results are mostly independent of Sections \ref{sec:SL}, \ref{sec:dim}, and \ref{sec:ergodic}.
We remark in passing that an assymetric Gromov product on $\Ft$ can be defined as follows:
Define the {Gromov product} with respect to a base point $x\in X$ by
\be\label{eqn:GP_asym}
\norm{\t_+|\t_- }^\tb_x = \frac{1}{2}\LR{\dhor_{\t_+}(x,z) + \idhor_{\t_-}(x,z)},
\quad
\tau_\pm\in{\rm Flag}(\tmod) \text{ are antipodal},
\ee
where $z$ is some point on the parallel set $P(\t_+,\t_-)$ spanned by $\t_\pm$.
Following the proof of Lemma \ref{lem:GPinvariant}, it can be checked that the definition in (\ref{eqn:GP_asym}) does not depend on the choice of $z\in P(\t_+,\t_-)$.
The above leads to an assymetric Gromov premetric $\dg{\e}$ in $\Ft$ (cf. Definition \ref{def:dg}) which satisfies
\[
 \dg{\e}_x(\t_1,\t_2) = D^{\iota\tb,\e}_x(\t_2,\t_1),\quad
 \forall \tau_1,\tau_2\in\Ft.
\]

The results in Sections \ref{sec:SL} and \ref{sec:dim} are valid verbatim in the assymetric case.
In the proof of Theorem \ref{thm:ergoAno} in Section \ref{sec:ergodic}, we only need to modify the justification of Claim \ref{clm:three} in the proof of the crucial Sublemma \ref{sublemma} in the asymmetric case.
The statement in that claim is: {\em If $S(x: B(\g x_0,r))\cap S(x: B(\phi x_0,r)) \ne\emptyset$, for $\g,\phi\in\G^*_{d}$, then $\df(\g x_0, \phi x_0)$ is uniformly bounded.}
The proof of this claim uses the main result of Section \ref{sec:hypmorse} and it was given for symmetric functions $d_\tb$ (note that in that proof, the requirement that $\tb$ has unit length is unimportant, only the hyperbolicity of $\G$-orbits in $X$ is relevant).
When $d_\tb$ is asymmetric, we consider the {\em symmetrization}
function $d_\tb^{\rm sym} = (d_\tb + d_{\iota\tb})/2$,
\[
 d_\tb^{\rm sym}(x,y) = \left\langle d_\Delta(x,y) | \frac{\tb+\iota\tb}{2}\right\rangle, \quad x,y\in X,
\]
 and the previous proof implies that, under the hypothesis, $d^{\rm sym}_\tb(\g x_0,\phi x_0)$ is uniformly bounded.
Since $d_\tb \le 2d^{\rm sym}_\tb$, $d_\tb(\g x_0,\phi x_0)$ is also uniformly bounded under the same hypothesis.

\bibliographystyle{plain}

\vfill 

\noindent
S.D.

Department of Mathematics,
   Yale University, 
   10 Hillhouse Ave, New Haven, CT 06511

\texttt{subhadip.dey@yale.edu}

\medskip
\noindent
M.K.

{Department of Mathematics,
   University of California, Davis,
   One Shields Ave, Davis, CA 95616}

\texttt{kapovich@math.ucdavis.edu}

\end{document}